\documentclass[11pt, leqno]{amsart}
\usepackage[leqno]{amsmath}
\usepackage{amssymb}
\usepackage{enumerate}

\usepackage{subfigure}
\usepackage{mathrsfs}
\usepackage{graphicx}
\usepackage{bm}
\usepackage[all]{xy}
   \usepackage[margin=1in]{geometry}

\usepackage[usenames,dvipsnames]{xcolor}
\usepackage[
colorlinks=true,linkcolor=NavyBlue,urlcolor=RoyalBlue,citecolor=PineGreen,bookmarks=true,bookmarksopen=true,bookmarksopenlevel=2,unicode=true,linktocpage]{hyperref}

\usepackage{microtype}
\usepackage{centernot}
\usepackage{stmaryrd}
\usepackage{comment}

\numberwithin{equation}{section}

\newtheorem{theorem}{Theorem}[section]

\newtheorem{lemma}[theorem]{Lemma}
\newtheorem{proposition}[theorem]{Proposition}
\newtheorem{corollary}[theorem]{Corollary}

\newtheorem{definition}[theorem]{Definition}

\theoremstyle{remark}
\newtheorem{remark}[theorem]{Remark}

\newcommand{\C}{\mathbf{C}}
\newcommand{\D}{\mathbf{D}}
\newcommand{\E}{\mathbf{E}}
\newcommand{\F}{\mathbf{F}}
\newcommand{\h}{\mathbf{H}}
\newcommand{\N}{\mathbf{N}}
\newcommand{\Z}{\mathbf{Z}}
\newcommand{\p}{\mathbf{P}}
\newcommand{\Q}{\mathbf{Q}}
\newcommand{\R}{\mathbf{R}}

\newcommand{\Fh}{\mathfrak {h}}

\newcommand{\CC}{\mathcal {C}}
\newcommand{\CD}{\mathcal {D}}

\newcommand{\CF}{\mathcal {F}}

\newcommand{\CP}{\mathcal {P}}

\newcommand{\CS}{\mathcal {S}}

\newcommand{\CZ}{\mathcal {Z}}

\newcommand{\SLE}{{\rm SLE}}

\newcommand{\dist}{\mathrm{dist}}

\newcommand{\diam}{\mathrm{diam}}

\newcommand{\im}{\mathrm{Im}}
\newcommand{\re}{\mathrm{Re}}

\newcommand{\one}{{\bf 1}}

\newcommand{\wt}{\widetilde}
\newcommand{\wh}{\widehat}
\newcommand{\ol}{\overline}

\newcommand{\giv}{\,|\,}

\newcommand{\fan}{{\mathbf F}}

\newcommand{\eps}{\epsilon}

\newcommand{\BES}{\mathrm{BES}}

\newcommand{\BESQ}{\mathrm{BESQ}}

\newcommand{\strip} {\mathscr{S}}

\newcommand{\ee}{\varepsilon}
\newcommand{\dd}{\delta}
\newcommand{\aal}{\alpha}

\renewcommand{\gg}{\gamma}

\newcommand{\la}{\lambda}
\newcommand{\ff}{\varphi}

\newcommand{\tht}{\theta}
\newcommand{\DD}{\Delta}

\newcommand{\kk}{\kappa}
\renewcommand{\ss}{\sigma}
\newcommand{\zz}{\zeta}

\def\n#1{\left\lvert#1\right\rvert}

\newcommand{\nv}{^{-1}}
\newcommand{\gen}[1]{\left\langle #1 \right\rangle}
\newcommand{\del}{\partial}
\newcommand{\sm}{\setminus}
\newcommand{\slekr}{\SLE_\kk(\underline{\rho})}
\newcommand{\und}{\underline}
\newcommand{\ov}{\overline}

\newcommand{\comm}[1]{\textcolor{magenta}{[#1]}}

\renewcommand{\comm}[1]{}

\begin{document}

\title[Connectivity of the adjacency graph of the $\SLE$ fan]{Connectivity of the adjacency graph of\\ complementary components of the $\SLE$ Fan}

\author{Cillian Doherty, Konstantinos Kavvadias, and Jason Miller}

\begin{abstract}
Suppose that $h$ is an instance of the Gaussian free field (GFF) on a simply connected domain $D \subseteq \C$ and $x,y \in \partial D$ are distinct.  Fix $\kappa \in (0,4)$ and for each $\theta \in \R$ let $\eta_\theta$ be the flow line of $h$ from $x$ to $y$.  Recall that for $\theta_1 < \theta_2$ the \emph{fan} $\fan(\theta_1,\theta_2)$ of flow lines of $h$ from $x$ to $y$ is the closure of the union of $\eta_\theta$ as $\theta$ varies in any fixed countable dense subset of $[\theta_1,\theta_2]$.  We show that the adjacency graph of components of $D \setminus \fan(\theta_1,\theta_2)$ is a.s.\ connected, meaning it a.s.\ holds that for every pair $U,V$ of components there exist components $U_1,\ldots,U_n$ so that $U_1 = U$, $U_n = V$, and $\partial U_i \cap \partial U_{i+1} \neq \emptyset$ for each $1 \leq i \leq n-1$.  We further show that $\fan(\theta_1,\theta_2)$ a.s.\ determines the flow lines used in its construction.  That is,  for each $\theta \in [\theta_1,\theta_2]$ we prove that $\eta_\theta$ is a.s.\ determined by $\fan(\theta_1,\theta_2)$ as a set.
\end{abstract}

\date{\today}
\maketitle

\setcounter{tocdepth}{1}
\tableofcontents

\newcommand{\hcap}{{\mathrm {hcap}}}

\parindent 0 pt
\setlength{\parskip}{0.2cm plus1mm minus1mm}

\section{Introduction}
\label{sec:intro}

The Schramm-Loewner evolution ($\SLE_{\kappa}$, $\kappa > 0$) is a one parameter family of curves which connect two boundary points of a simply connected domain.  It was introduced by Schramm in \cite{s2000sle} as a candidate to describe the scaling limit of the interfaces of various two-dimensional discrete models from statistical mechanics at criticality.  A number of such convergence results have been proved in the case of planar lattices \cite{lsw2004lerw,s2001cardy,s2010ising,ss2009contours} and random planar maps \cite{s2016inventory,kmsw2019bipolar,lsw2017wood,gkmw2018active, gwynne2021convergence,gm2021saw}.  The parameter $\kappa \geq 0$ determines the roughness of an $\SLE_{\kappa}$ curve.  An $\SLE_0$ is a smooth curve and $\SLE_{\kappa}$ curves become more fractal as $\kappa$ increases.  There are three regimes of $\kappa$ values which exhibit rather different behavior: for $\kappa \in [0,4]$ an $\SLE_{\kappa}$ is a.s.\  simple, for $\kappa \geq 8$ it is a.s.\  space-filling, and  for $\kappa \in (4,8)$ it is a.s.\ self-intersecting but not space-filling \cite{rs2005basic}.  Furthermore,  the a.s.\ dimension of the range of an $\SLE_{\kappa}$ curve is given by $\min(1+\kappa/8 ,  2)$ \cite{rs2005basic,beffara2008dimension}.

Schramm's original definition of $\SLE_\kappa$ is in terms of the chordal Loewner equation driven by $\sqrt{\kappa} B$ where $B$ is a standard Brownian motion and in the intervening years since its introduction several other representations of $\SLE_\kappa$ have been discovered.  In this work, its coupling with the Gaussian free field (GFF) will be particularly important.  Recall that the GFF $h$ on a simply connected domain $D \subseteq \C$ is the Gaussian field with covariance given by the Green's function $G$ for $\Delta$ on $D$.  Since $G(x,y) \sim -\log|x-y|$ as $x \to y$, the GFF has infinite variance at points and so takes values in the space of distributions on $D$ rather than a space of functions.  Nevertheless, it is very close to being a function in the sense that $h$ takes values in the Sobolev space $H^{-\epsilon}(D)$ for every $\epsilon > 0$ but not~$L^2$.  Consequently, it exhibits many of the features of a function (but with an extra complication to reflect that it is only distribution valued).  For example, it was shown in \cite{she2016zipper} (building on \cite{schramm2013contour}) that it is possible to make sense of the flow lines of the formal vector field $e^{i (h(\eta(t))/\chi + \theta )}$,  i.e.,  solutions to the $\text{ODE}$
\begin{align*}
\eta'(t) = e^{i ( h(\eta(t))/\chi + \theta)} \quad \text{for} \quad t>0 \quad\text{where}\quad \chi = \frac{2}{\sqrt{\kappa}}-\frac{\sqrt{\kappa}}{2} \quad\text{for}\quad \kappa \in (0,4)
\end{align*}
and they are $\SLE_\kappa$ type curves.  The theory of how the GFF flow lines behave and interact with each other was developed in \cite{dub2009gff, ms2016imag1,ms2017ig4}.

\begin{figure}[ht!]
\includegraphics[width=0.48\textwidth]{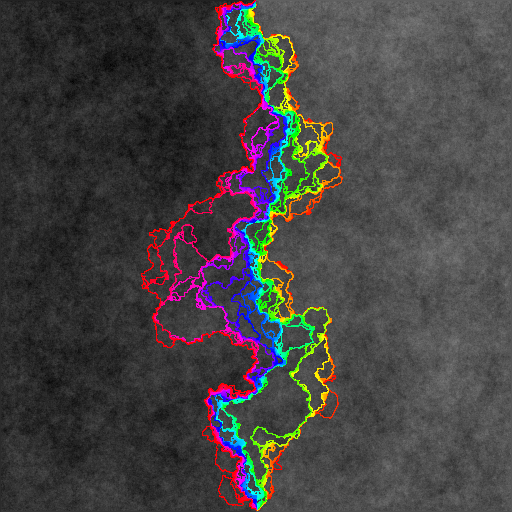} \hspace{0.01\textwidth} \includegraphics[width=0.48\textwidth]{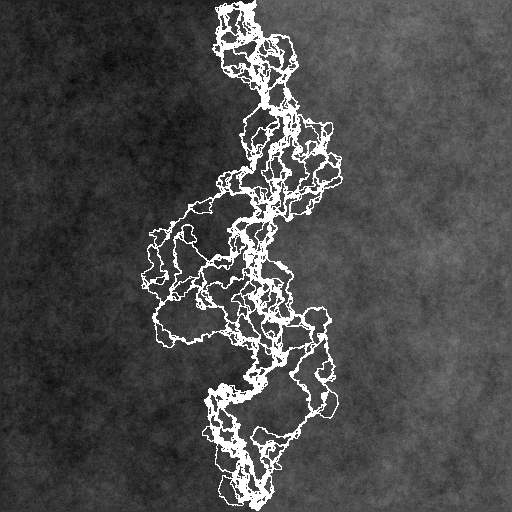}
\caption{\label{fig:fan_simulation} Simulation of $\fan$ for a GFF $h$ on $[-1,1]^2$ from $-i$ to $i$ and $\kappa = 1$.   The boundary conditions for $h$ are chosen so that the $0$-angle flow line is an $\SLE_1$ and the values of $h$ are shown in grayscale.  {\bf Left:}  Flow lines with different angles are shown with different colors.   {\bf Right:} All of the flow lines are shown in white.  In Theorem~\ref{thm:fan_determines_flow_lines},  we prove that one can recover the flow lines which make up $\fan$ from $\fan$ (as a closed set).  That is, the colors in the left picture are determined by the right picture.}
\end{figure}

Using the GFF, one can consider families of $\SLE_\kappa$-type curves in order to build various types of random fractals.  The main focus of the present work is on the so-called \emph{fan} considered in \cite{ms2016imag1}.  More precisely, fix $\kappa \in (0,4)$ and let $\lambda = \pi/\sqrt{\kappa}$.  Let $h$ be a $\text{GFF}$ on $\h$ with boundary values given by $-a$ (resp.\ $b$) on~$\R_-$ (resp.\ $\R_+$) where $a,b$ satisfy
\begin{equation}
\label{eqn:fan_bd}
\frac{a+\lambda}{\chi} > \frac{\pi}{2} \quad\text{and}\quad \frac{\lambda + b}{\chi} > \frac{\pi}{2}.
\end{equation}
For $\theta \in \R$,  we let $\eta_{\theta}$ be the flow line of $h$ of angle $\theta$ from $0$ to $\infty$.  The boundary conditions of $h$ ensure that $\eta_{\theta}$ is a.s.\ well-defined for $-\pi/2 \leq \theta \leq \pi/2$.  Let $\Theta$ be any countable dense subset of $[-\pi/2,\pi/2]$.  The fan $\fan$ of $h$ is defined to be the closure of $\cup_{\theta \in \Theta} \eta_{\theta}$; the definition turns out not to depend on the choice of $\Theta$ (see Proposition~\ref{prop:bounded_hausdorff}).  It was shown in \cite{ms2016imag1} that the Lebesgue measure of $\fan$ is a.s.\ equal to~$0$ and in \cite{miller2017dimension} that the a.s.\ Hausdorff dimension of $\fan$ is $1 + \kappa/8$,  the same as that of ordinary $\SLE_{\kappa}$.  More generally, if we fix $\theta_1 < \theta_2$ and assume that 
\begin{equation}
\label{eqn:fan_gen_bd}
\frac{a+\lambda}{\chi}  \geq \theta_2 \quad\text{and}\quad  \frac{\lambda + b}{\chi} \geq -\theta_1
\end{equation}
then we define $\fan(\theta_1,\theta_2)$ to be the closure of $\cup_{\theta \in \Theta} \eta_\theta$ where $\Theta$ is any fixed countable dense subset of $[\theta_1,\theta_2]$ (or of $(\tht_1, \tht_2)$ if we have equality in (\ref{eqn:fan_gen_bd}), see Section~\ref{subsec:sle_fan}).
As in the case of $\fan$, we have that $\fan(\theta_1,\theta_2)$ does not depend on the choice of $\Theta$ and has a.s.\ Hausdorff dimension $1+\kappa/8$.

The main focus of this work is on the so-called \emph{adjacency graph of complementary components} of~$\fan(\theta_1,\theta_2)$.   That is, we consider the graph whose vertices are the components of $\h \setminus \fan(\theta_1,\theta_2)$ and we say that components $U, V$ are connected by an edge if $\partial U \cap \partial V \neq \emptyset$.  We say that the adjacency graph of such components is connected if every pair of such components can be connected by a finite length path.  Our main result is the following theorem.

\begin{theorem}
\label{thm:fan_connectivity}
Fix $\kappa \in (0,4)$ and assume that we have the setup described above.  Then a.s.\ we have that the adjacency graph of components of $\h \setminus \fan(\theta_1,\theta_2)$ is connected.
\end{theorem}

Let us give some further context and background to motivate Theorem~\ref{thm:fan_connectivity}.  First, the corresponding result with an $\SLE_{\kappa'}$ process for $\kappa' \in (4,8)$ in place of $\fan(\theta_1,\theta_2)$ was posed in \cite[Question 11.2]{dms2021mating} and partially solved in \cite{gp2020adj}.  Namely, in \cite{gp2020adj} it was shown that there exists $\kappa_0' \in (4,8)$ such that for all $\kappa' \in (4,\kappa_0')$ the graph of components of $\h \setminus \eta'$ for $\eta' \sim \SLE_{\kappa'}$ is connected.   Determining whether the adjacency graph of complementary components of an $\SLE_{\kappa'}$ is connected for all $\kappa' \in (4,8)$ remains an open question.  This is in contrast with Theorem~\ref{thm:fan_connectivity}, which gives the connectivity of the adjacency graph of complementary components of $\fan(\theta_1,\theta_2)$ for all $\kappa \in (0,4)$.  Another long-standing open problem is to determine whether the adjacency graph of complementary components (in $\C$) of the range of a planar Brownian motion run for one unit of time is a.s.\ connected (see \cite[Open Problem 4]{mp2010} or \cite[Problem 2]{bur}).

The connectivity of the adjacency graph of complementary components is a useful topological property which has many other applications in the study of such random fractals.  Let us describe one such motivation, which comes from the relationship between $\SLE_\kappa$-type curves and Liouville quantum gravity (LQG) surfaces.   Recall that an $\text{LQG}$ surface is a random two-dimensional Riemannian manifold which is formally described by the metric tensor
\begin{equation}
\label{eqn:lqg_metric_tensor}
e^{\gamma h(z)}\left( dx^2 + dy^2\right) \quad\text{for}\quad z = x + iy
\end{equation}
where $dx^2 + dy^2$ is the Euclidean metric on a domain $D \subseteq \C$,  $h$ is (some form of) the $\text{GFF}$ on $D$,  and $\gamma \in (0,2]$ is a parameter.  Since the $\text{GFF}$ $h$ and its variants are random variables which live in the space of distributions rather than the space of functions,  some care is required in order to make sense of~\eqref{eqn:lqg_metric_tensor}  \cite{ds2011lqg}.  It was shown in \cite{she2016zipper} that for $\gamma \in (0,2)$ if certain types of $\text{LQG}$ surfaces are conformally welded according to quantum boundary length (i.e., the boundary length measure associated with~\eqref{eqn:lqg_metric_tensor}) then the conformal welding is well-defined and the welding interface is an $\SLE_{\kappa}$ curve for $\kappa = \gamma^2 \in (0,4)$.  The result of \cite{she2016zipper} was extended to the case $\gamma=2$ and $\kappa = 4$ in \cite{hp2021welding} and analogous welding results for $\kappa' > 4$ were established in \cite{dms2021mating}.

In general, it is non-trivial to determine whether a given conformal welding is unique.  The uniqueness is known to hold if the welding interface is conformally removable.  Recall that a set $K \subseteq \C$ is said to be \emph{conformally removable} if every homeomorphism $\varphi \colon 	\C \to \C$ which is conformal on $\C \setminus K$ is conformal on $\C$.  It is also known that the range of an $\SLE_{\kappa}$ curve for $\kappa \in (0,4)$ is a.s.\   conformally removable (see \cite{js2000removability,rs2005basic}).  The conformal removability of $\SLE_4$ was proved in \cite{kms2022sle4} and for $\SLE_{\kappa'}$ with $\kappa' \in (4,8)$ when the adjacency graph of complementary components is connected in \cite{kms2023nonsimpleremove}.  We believe that it is possible to combine Theorem~\ref{thm:fan_connectivity} with the methods used in \cite{kms2023nonsimpleremove} to show that the range of $\fan$ is a.s.\ conformally removable.

It is clear that $\fan(\theta_1,\theta_2)$ determines $\eta_{\theta_1}$ and $\eta_{\theta_2}$ since the latter (resp.\ former) is simply the left (resp.\ right) boundary of $\fan(\theta_1,\theta_2)$.  Is is therefore interesting to ask whether $\fan(\theta_1,\theta_2)$ determines~$\eta_\theta$ for $\theta \in (\theta_1,\theta_2)$ or if additional information is required in order to recover~$\eta_\theta$ from $\fan(\theta_1,\theta_2)$.  Concretely, the question is whether the colors in the left hand side of Figure~\ref{fig:fan_simulation} can be recovered from only observing the right hand side of Figure~\ref{fig:fan_simulation}.  We note that the analog of this question has been analyzed in the case of a $\SLE_{\kappa'}$ process with $\kappa' \in (4,8)$ in \cite{msw2020nonsimple}, in which it is shown that is \emph{not} possible to recover the curve by only observing its range.  The following theorem gives that, in contrast to \cite{msw2020nonsimple}, it is possible to recover the flow lines which make up $\fan(\theta_1,\theta_2)$ when only observing $\fan(\theta_1,\theta_2)$.

\begin{theorem}
\label{thm:fan_determines_flow_lines}
Fix $\kappa \in (0,4)$ and assume that we have the setup described above.  For each $\theta \in [\theta_1,\theta_2]$ we have that $\fan(\theta_1,\theta_2)$ a.s.\ determines $\eta_{\theta}$.
\end{theorem}

\subsection*{Outline.}  The remainder of this article is structured as follows.  In Section \ref{sec:prelim} we give brief introductions to $\SLE$, the GFF, and the flow line coupling.  In Section~\ref{sec:hausdorff_convergence} we will study the continuity properties of $\SLE_\kappa(\rho)$ processes as $\rho$ varies and these results will be important for the proof of Theorem~\ref{thm:fan_connectivity}, which we will complete in Section~\ref{sec:proofs}. 
We prove Theorem~\ref{thm:fan_determines_flow_lines} in Section~\ref{sec:fan_determines_flow_lines}.

\subsection*{Acknowledgements} C.D.\ was supported by EPSRC grant EP/W524633/1 and a studentship from Peterhouse, Cambridge.  K.K. and J.M.\ were supported by ERC starting grant 804116 (SPRS).

\section{Preliminaries}\label{sec:prelim}

We will first review the basics of Bessel processes in Section~\ref{subsec:bessel}.  We will next review the $\SLE_\kappa$ and $\SLE_\kappa(\rho)$ processes in Section~\ref{subsec:sle}.  The purpose of Section~\ref{subsec:gff} is to review the GFF and in Section~\ref{subsec:ig} the theory of the flow lines of the GFF.  Finally, in Section~\ref{subsec:sle_fan} we will recall some basic facts about the fan.

\subsection{Bessel processes}
\label{subsec:bessel}
We are now going to review some basic facts about Bessel processes.  We refer the reader to \cite{revuz2013continuous} for more details.  Fix $\delta \in \R$.  The starting point for the construction of the law of a Bessel process of dimension $\delta$ ($\BES^\delta$) is the construction of the law of the so-called square Bessel process of dimension $\delta$ $(\BESQ^{\delta})$.  The law of a $\BESQ^{\delta}$ is described by the SDE
\begin{equation}\label{eqn:bessel_square_sde}
dY_t = \delta dt + 2\sqrt{Y_t} dB_t,\quad Y_0 = y_0 >0,
\end{equation}
where $B$ is a standard Brownian motion.  Standard results for SDEs imply that there is a unique strong solution to~\eqref{eqn:bessel_square_sde},  at least up until the first time that the process hits $0$.  In the case that $\delta>0$,  there is a unique strong solution to~\eqref{eqn:bessel_square_sde} for all $t \geq 0$ and the solution remains non-negative for all times a.s.

A process $X$ has the law of a $\BES^{\delta}$ process if it can be expressed as $X = \sqrt{Y}$ where $Y$ is a $\BESQ^{\delta}$.  By applying It\^o's formula,  we obtain that $X$ solves the SDE
\begin{equation}
\label{eqn:bessel_sde}
dX_t = \frac{a}{X_t}dt + dB_t,\quad X_0 = x\geq 0 \quad\text{where}\quad a = \frac{\delta-1}{2},
\end{equation}
at least up until it hits $0$.  If $\delta \in (0,2)$,  then $X_t$ will a.s.\ hit and be instantaneously reflected at~$0$; if $\delta = 2$, then $X_t$ will a.s.\  not hit $0$ but will get arbitrarily close to $0$; and if $\delta > 2$,  then a.s.\  $\lim_{t \to \infty} X_t = \infty$.  An important property of $\BES^{\delta}$ processes is that they satisfy Brownian scaling, i.e.,  if $X$ is a $\BES^{\delta}$ then so is the process $t \mapsto r^{-1}X_{r^2 t}$ for each $r>0$ fixed.

When $\delta \in (1,2)$,  the process solves~\eqref{eqn:bessel_sde} in integrated form for all $t \geq 0$ hence is a semimartingale.  When $\delta = 1$,  the process is equal in distribution to $|B|$ where $B$ is a standard Brownian motion and solves~\eqref{eqn:bessel_sde} with an extra correction coming from the local time of $X$ at $0$ so is also a semimartingale.  When $\delta \in (0,1)$, the process does not solve~\eqref{eqn:bessel_sde} when hitting $0$ in integrated form and is not a semimartingale.  In order to make sense of it as a solution to~\eqref{eqn:bessel_sde},  we need to make a so-called principal value correction (see \cite{s2009cle} for more details).
Let
\[ \nu = \frac{\delta}{2}-1\]
and let $I_{\nu}$ be the Bessel function of parameter $\nu$.  Then the transition kernel of a $\BES^{\delta}$ process is given by
\begin{align*}
p_t(x,y) &= t^{-1}\left(\frac{y}{x}\right)^{\nu} y \exp(-(x^2 + y^2) / 2t) I_{\nu}\left(\frac{xy}{t}\right) \quad\text{and}\\
p_t(0,y) &= 2^{-\nu} t^{-(\nu + 1)} \Gamma(\nu + 1)^{-1} y^{2\nu + 1} \exp(-y^2 / 2t) \quad\text{for each}\quad x,y,t > 0.
\end{align*}

Fix $\delta \in (0,2)$ and suppose that $X$ is a $\BES^{\delta}$ starting from $0$.  Then we recall from the It\^o excursion decomposition of $X$ that we can sample from the law of $X$ in the following way:
\begin{itemize}
\item Pick a Poisson point process $\Lambda$ from the measure $c_\delta du \otimes t^{\delta/2-2} dt$ where both $du$ and $dt$ denote Lebesgue measure on $\R_+$ and $c_\delta = \delta/2$. 
\item For each $(u,t) \in \Lambda$, we sample a Bessel excursion $e_{u,t}$ of length $t$ from $0$ to $0$.
\item We concatenate together the $e_{u,t}$ ordered according to the associated $u$ value.
\end{itemize}
In this construction, we have that $u$ gives the local time for $X$ at $0$ at which $e$ occurs.  Also,  we recall that a $\BES^\delta$ excursion with $\delta \in (0,2)$ and length $t$ from $0$ to $0$ can be sampled as follows. First,  we start with a $\BES^{4-\delta}$ process $Y$ starting from $0$ and then we weight it locally by $J_s = \exp(- Y_t^2 / (2(t-s)))$.  It\^o's formula implies that if $Y$ satisfies
\[ dY_s = \frac{1-a}{Y_s}ds + dB_s \quad\text{with}\quad a = \frac{\delta - 1}{2},\]
then 
\begin{align*}
dJ_s = J_s \left(\frac{a - \frac{3}{2}}{t-s} ds -\frac{Y_s}{t-s} dB_s\right),
\end{align*}
which shows that
\[ M_s = \left( \frac{t}{t-s} \right)^{\frac{3}{2}-a} J_s\]
is a local martingale for $s<t$ satisfying
\[ dM_s = -\frac{Y_s}{t-s} M_s dB_s.\]
Then,  Girsanov's theorem implies that if we weight by the local martingale $M_s$,  we obtain that $Y$ solves the SDE
\begin{equation}\label{eqn:bessel_excursion_sde}
dY_s = \left( \frac{1-a}{Y_s} - \frac{Y_s}{t-s} \right) ds + dW_s,
\end{equation}
where $W$ is a standard Brownian motion under the new measure. Altogether, if we consider the paths up to time $t-\epsilon$,  then the Radon-Nikodym derivative of the excursion with respect to a $\BES^{4-\delta}$ process is proportional to $\exp(-Y_{t-\epsilon}^2 / (2(t-\epsilon)))$.

\subsection{Schramm--Loewner Evolution}
\label{subsec:sle}

$\SLE_{\kappa}$ is a one-parameter family of random curves indexed by $\kappa>0$ and introduced by Schramm in \cite{s2000sle} as a candidate to describe the scaling limit of the interfaces in discrete models in two dimensions at criticality.  $\SLE_{\kappa}$ is defined using the so-called Loewner equation.  More precisely,  let $W \colon \R_+ \to \R$ be a continuous function and for each $z \in \h$ we let $g_t(z)$ be the unique solution to the ODE
\begin{equation}
\label{eq:loewner}
\partial_t g_t(z) = \frac{2}{g_t(z)-W_t},\quad g_0(z)=z
\end{equation}
up to time $\tau_z = \sup\{t \geq 0 \colon \im (g_t(z)) > 0\}$ and $K_t = \{z \in \h \colon \tau_z \leq t\}$.  We let $\h_t = \h \setminus K_t$. The map $g_t$ is then the unique conformal transformation $\h_t \to \h$ satisfying $g_t(z) - z \to 0$ as $z \to \infty$.   To define $\SLE_{\kappa}$,  we fix $\kappa \geq 0$ and then take $W = \sqrt{\kappa}B$ where $B$ is a standard Brownian motion.  It was proved in \cite{rs2005basic,lsw2004lerw} (see also \cite{am2022sle8}) that there a.s. exists a random curve $\eta$ in $\h$ from~$0$ to~$\infty$ such that for each $t \geq 0$ we have that $\h_t$ is the unbounded component of $\h \setminus \eta([0,t])$.  Then we say that~$\eta$ has the law of an $\SLE_{\kappa}$ curve.

Almost surely, $\eta$ is a simple curve for $\kk \in [0,4]$, is self-intersecting but not space-filling when $\kk \in (4,8)$, and is space-filling when $\kk \geq 8$ \cite{rs2005basic}. Given a simply connected domain $D$ and two distinct prime ends $x, y \in \del D$ let $\ff \colon D \to \h$ be a conformal map sending $x$ to $0$ and $y$ to $\infty$. $\SLE_\kk$ in $D$ started at $x$ and targeted at $y$ is the random curve $\ff\nv(\eta)$ where $\eta$ is an $\SLE_\kk$ in $\h$ (started at $0$ and targeted at $\infty$) as defined above.

We will also need to consider the $\SLE_\kk(\underline{\rho})$ processes, which are a variant of SLE where the driving function $W$ in~\eqref{eq:loewner} can be more general than a multiple of Brownian motion.  In that case,  we need to keep track of extra marked points called force points.  More precisely,  we define vectors of \textit{force points} $\underline{x}^L = (x_1^L, \ldots, x_\ell^L)$ and $\und{x}^R = (x_1^R, \dots, x_r^R)$ in $\partial \h$ where $-\infty< x_\ell^L < \cdots < x_1^L \leq 0$ and $0 \leq x_1^R < \cdots < x_r^R <\infty$ and associated real-valued vectors of \textit{weights} $\und{\rho}^L = (\rho_1^L, \dots, \rho_\ell^L)$ and $\und{\rho}^R = (\rho_1^R, \dots, \rho_r^R)$ in $\R$. In this case, we let $W$ be the solution to the SDE
\begin{align}
\label{eqn:slekrsde}
    dW_t = \sqrt{\kk}dB_t + \sum_{q \in \{L,R\}} \sum_{i} \frac{\rho_i^q}{W_t - V_t^{i,q}}dt, \quad
    dV_t^{i,q} = \frac{2}{V_t^{i,q} - W_t}dt,\quad V_0^{i,q} = x_i^q.
\end{align}
As shown in \cite{ms2016imag1}, there exists a unique solution to~\eqref{eqn:slekrsde} up until the \textit{continuation threshold} is hit; this is the first time that either
\begin{align*}
\sum_{i : V_t^{i,L} = W_t}\rho^{i,L} \leq -2 \quad \text{or} \quad \sum_{i : V_t^{i,R} = W_t} \rho^{i,R} \leq -2.
\end{align*}
Note that $V_t^{i,q}$ encodes the evolution of the force points under the Loewner flow. Similarly to the previous case, the hulls $K_t$ up until the continuation threshold is hit are generated by a continuous random curve $\eta$ starting at $0$ \cite{ms2016imag1}. We say that a random curve $\eta$ generated in this way has the law of an $\SLE_\kk(\underline{\rho})$. We refer to \cite{ms2016imag1} for a more detailed description of $\slekr$ processes. $\slekr$ processes can also be defined in domains other than $\h$ as in the case of standard $\SLE_\kk$. In this case force points are located at prime ends of $\del D$ and are mapped to $\R$ by a conformal map $D \to \h$.

In the special case where we have a single force point located at $0^+$ (resp.\ $0^-$) of weight $\rho>-2$,  the process can be defined continuously for all times a.s.  Also,  if $\rho \in (-2,\kappa/2-2)$ then it hits $(0,\infty)$ (resp.\ $(-\infty,0)$) a.s. while if $\rho \geq \kappa/2-2$,  then the process intersects $\partial \h$ only at the origin a.s.  Moreover we have the following absolute continuity results for $\SLE_\kappa(\und{\rho})$ processes.  Suppose that $z^1,\ldots,z^n \in \partial \h$ and $\rho_1,\ldots,\rho_n \in \R$.  We also let $z_t^j = x_t^j + i y_t^j = g_t(z^j)$ where $(g_t)$ is the Loewner flow associated with an $\SLE_\kappa$ process.  We recall from \cite[Theorem~6]{sw2005coordinate} that the process
\begin{equation}
\label{eqn:sw_mg}
M_t = \prod_{j=1}^n \left( | g_t'(z_j)|^{(8-2\kappa+\rho_j)\rho_j/(8\kappa)} (y_t^j)^{\rho_j^2/8\kappa}| W_t - z_t^j|^{\rho_j/\kappa} \right) \prod_{1 \leq j < j' \leq n} (| z_t^j - z_t^{j'}| |\ol{z_t^j} - z_t^{j'}|)^{\rho_j \rho_{j'}/(4\kappa)}
\end{equation}
is a continuous local martingale for an $\SLE_\kappa$ process and that if we weight the law of an $\SLE_\kappa$ process by $M_t$ stopped at the first time that it disconnects (or hits) from $\infty$ one of the points $z^1,\ldots,z^n$ then we get an $\SLE_\kappa(\rho_1,\ldots,\rho_n)$ process with force points at $z^1,\ldots,z^n$ and stopped at the corresponding time.

\subsection{The Gaussian free field}
\label{subsec:gff}

Let $D \subseteq \C$ be a simply connected domain with harmonically non-trivial boundary (that is, a Brownian motion started at any point in $D$ a.s.\ exits $D$). Let $C^\infty_0(D)$ be the space of smooth functions on $D$ with compact support and let $H_0(D)$ be the closure of this space with respect to the Dirichlet inner product
\[ \gen{f,g}_\nabla = \frac{1}{2\pi} \int_D \nabla f(x) \cdot \nabla g(x) dx. \]
Let $(f_n)_{n \geq 1}$ be an orthonormal basis for $H_0(D)$ with respect to this norm and let $(\aal_n)_{n \geq 1}$ be a sequence of i.i.d. standard normal random variables. We define the zero-boundary GFF on $D$ to be the formal sum
\[ h = \sum_{n=1}^\infty \aal_n f_n. \]
This sum a.s.\ does not converge in $H_0(D)$ but does converge a.s.\ in the space of distributions on~$D$.   More precisely,  it converges a.s.  in the space $H^{-1}(D)$ (at least when $D$ is bounded) which is the dual space of $H_0(D)$ when the latter is endowed with the norm induced by the Dirichlet inner product. The law of $h$ does not depend on the choice of orthonormal basis. The GFF with non-zero boundary conditions is defined to be the sum of a zero-boundary GFF and a harmonic function with these boundary conditions. For a more detailed introduction to the GFF we refer the reader to \cite{she2007gff}.

We can also define the \textit{whole-plane GFF} on $\C$. To do so, let $H_s(\C)$ be the set of functions $f \in C_0^\infty(\C)$ where $\int_\C f(x)dx = 0$. We let $H(\C)$ be the Hilbert space closure of $H_s(\C)$ with respect to the Dirichlet inner product, let $(f_n)_{n \geq 1}$ be an orthonormal basis for this space, and let $(\aal_n)_{n \geq 1}$ be defined as before. As above, we define the whole-plane GFF as
\[h = \sum_{n=1}^\infty \aal_n f_n,\]
which is a.s.\ a well-defined distribution on $\C$ but is defined only up to an additive constant. We can fix this constant by requiring, for instance, that $(h,g_0) = 0$ for some fixed function $g_0 \in C_0^\infty(\C)$ with $\int_\C g(x)dx = 1$. See \cite{ms2017ig4} for more details.

\subsection{Imaginary geometry}
\label{subsec:ig}

Now we focus on the coupling between $\SLE$ and the $\text{GFF}$ \cite{ms2016imag1,ms2017ig4}.  We fix $\kappa \in (0,4)$ and let 
\begin{align*}
\lambda = \frac{\pi}{\sqrt{\kappa}},  \quad \chi = \frac{2}{\sqrt{\kappa}} - \frac{\sqrt{\kappa}}{2},\quad\text{and}\quad \kappa' = \frac{16}{\kappa}.
\end{align*}
We fix weights $\und{\rho}^L,\und{\rho}^R$ and force points $\und{x}^L,\und{x}^R$ in $\partial \h$ such that $|\und{\rho}^L| = |\und{x}^L| = \ell,  |\und{\rho}^R| = |\und{x}^R| = r$ and the $\und{x}^L$ (resp.\ $\und{x}^R$) are to the left (resp.\ right) of $0$ and given in decreasing (resp.\ increasing) order.  We also set $x_0^L = 0^-,x_0^R = 0^+, x_{\ell+1}^L = -\infty,x_{r+1}^R = \infty$,  and $\rho_0^L = \rho_0^R = 0$.  Let $h$ be a $\text{GFF}$ on $\h$ with boundary conditions given by
\begin{align*}
 -\lambda \left(1+\sum_{i=1}^j \rho_i^L\right) \quad&\text{in} \quad (x_{j+1}^L,x_j^L] \quad\text{for each}\quad 0 \leq j \leq \ell \quad\text{and}\\
 \lambda \left(1+\sum_{i=1}^j \rho_i^R\right) \quad&\text{in}\quad(x_j^R,x_{j+1}^R] \quad\text{for each}  \quad 0 \leq j \leq r.
 \end{align*}
 Then it is shown in \cite[Theorem~1.1]{ms2016imag1} that there exists a coupling between $h$ and an $\SLE_{\kappa}(\und{\rho}^L;\und{\rho}^R)$ process $\eta$ in $\h$ from $0$ to $\infty$ so that the following is true.  Let $(K_t)$ be the increasing family of hulls associated with $\eta$ and let $f_t = g_t - W_t$ be its centered Loewner flow.  Then for each stopping time $\tau$ that a.s.  occurs before the continuation threshold is hit,  $K_{\tau}$ is a local set of $h$ in the sense of \cite{schramm2013contour} and conditionally on $K_{\tau}$,  the field $h \circ f_{\tau}^{-1} - \chi \arg(f_{\tau}^{-1})'$ has the law of a $\text{GFF}$ on $\h$ with boundary conditions given by
 \begin{align*}
 -\lambda \quad&\text{in}\quad (f_{\tau}(x_0^L),0^-]\quad\text{and}\quad \lambda \quad\text{in}\quad (0^+,f_{\tau}(x_0^R)]\\
 -\lambda \left(1+\sum_{i=1}^j \rho_i^L\right) \quad&\text{in}\quad (f_{\tau}(x_{j+1}^L),f_{\tau}(x_j^L)] \quad\text{for each}\quad 0 \leq j \leq \ell \quad\text{and}\\
 \lambda \left(1+\sum_{i=1}^j \rho_i^R\right) \quad&\text{in}\quad  (f_{\tau}(x_j^R),f_{\tau}(x_{j+1}^R)] \quad\text{for each}\quad 0 \leq j \leq r. 
 \end{align*}
 Moreover,  under this coupling,  $\eta$ is a.s.  determined by $h$ and is referred to as a \emph{flow line} of the field.

For a given angle $\tht \in \R$, we define the flow line of angle $\tht$ of $h$ to be the flow line (as defined above) of the GFF $h + \tht \chi$. Thus the original definition of a flow line corresponds to a flow line of angle $0$. Flow lines of the GFF defined on domains other than $\h$ can be defined via conformal mapping, although we refer to \cite{ms2016imag1} for further details.

The results of \cite{ms2016imag1} were extended in \cite{ms2017ig4} to the case of a whole-plane $\text{GFF}$ $h$.  More precisely,  it is shown in \cite[Theorem~1.1]{ms2017ig4} that we can generate flow lines of $h$ starting from different points and with different angles and that adding $2\pi \chi$ to the field does not change its flow lines.  Also,  the marginal law of such a flow line is that of a whole-plane $\SLE_{\kappa}(2-\kappa)$ process from $0$ to $\infty$ (see \cite[Section~2]{ms2017ig4} for the whole-plane version of $\SLE$) and \cite[Theorem~1.2]{ms2017ig4} implies that these flow lines are a.s.  determined by the field.  Moreover,  if $h$ is a $\text{GFF}$ on a domain $D \subseteq \C$ viewed as a distribution modulo $2\pi \chi$,  then the law of $h$ is locally absolutely continuous with respect to the law of a whole-plane $\text{GFF}$ modulo $2\pi \chi$.  Therefore,  we can make sense of the flow lines of $h$ starting from interior points of $D$ and stopped at the first time that they hit $\partial D$. The interaction between flow lines started from different points and with different angles is characterized in \cite[Theorem 1.5, Proposition 7.4]{ms2016imag1} and \cite[Theorems 1.7, 1.11]{ms2017ig4}.

We note that we can view the collection of flow lines in the whole-plane case as a type of a planar space filling tree in the following sense.  We fix a countable dense subset $(z_n)$ of $\C$.  Then it is shown in \cite{ms2017ig4} that the collection of flow lines starting at these points with the same angle has the property that a.s.  each pair of the above flow lines eventually merges and no two flow lines ever cross each other.  Moreover,  it is shown in \cite{ms2017ig4} that there exists a space-filling path that traces through the above tree of flow lines which is the so-called space-filling $\SLE_{\kappa'}$.  Furthermore,  the path traces the tree in a natural order in the sense that $z_n$ is hit before $z_m$ for $n \neq m$ when the flow line with angle $\pi/2$ starting from $z_n$ merges into the right side of the flow line with angle $\pi/2$ starting from $z_m$.

\subsection{The $\SLE$ fan}
\label{subsec:sle_fan}
Fix $\kappa \in (0,4)$ and let $h$ be a GFF on $\h$ with boundary conditions given by $-a$ on $\R_-$ and $b$ on $\R_+$ satisfying~\eqref{eqn:fan_bd}.  This hypothesis on the boundary conditions of $h$ ensures that for every $\theta \in [-\pi/2,\pi/2]$, if we start the $\theta$-angle flow line $\eta_{\theta}$ of $h$ from $0$ to $\infty$,  then the values of the weights of the force points associated with $\eta_{\theta}$ exceed $-2$. The \textit{$\SLE$ fan} $\fan$ was introduced in \cite{ms2016imag1} and defined to be the closure of $\cup_{\tht \in \Theta}\eta_\tht$ where $\Theta$ is any countable dense subset of $[-\tfrac\pi2, \tfrac\pi2]$. Proposition~\ref{prop:bounded_hausdorff} shows that the resulting object does not depend on the choice of $\Theta$.  It is also shown in \cite{ms2016imag1} that $\fan$ has Lebesgue measure $0$ a.s.\ and in \cite{miller2017dimension} that its dimension is $1+\kappa/8$.  The same results hold if we fix $\theta_1 < \theta_2$, $a,b$ satisfy~\eqref{eqn:fan_gen_bd} and we take $\fan(\theta_1,\theta_2)$ to be the closure of $\cup_{\theta \in \Theta} \eta_\theta$ where $\Theta$ is any fixed countable dense subset of $[\theta_1,\theta_2]$. 
If $\tht_2 = (a+\la)/\chi$ (resp.\ $\tht_1 = -(b+\la)/\chi$) note that the flow line of angle $\tht_2$ (resp. $\tht_1$) hits the continuation threshold immediately, since this case corresponds to $\rho^R = -2$ (resp.\ $\rho^L = -2$). We solve this issue either by defining $\eta_{\tht_1} = \R_+$ (resp. $\eta_{\tht_2} = \R_-$) or simply by choosing $\Theta \subseteq (\tht_1, \tht_2)$. Proposition~\ref{prop:bounded_hausdorff} ensures these two approaches are equivalent. We remark that choosing $\tht_1, \tht_2$ so that we have equality in (\ref{eqn:fan_gen_bd}) corresponds to the largest fan $\F(\tht_1, \tht_2)$ we can define for a field $h$ with these boundary conditions.

\section{Continuity of $\SLE_\kappa(\rho)$ in $\rho$ in the Hausdorff topology}
\label{sec:hausdorff_convergence}

\subsection{Main statements}
\label{subsec:main_statements}

The purpose of this section is to prove Propositions~\ref{prop:bounded_hausdorff} and~\ref{prop:intersections_stopped}, which are formally stated just below.  The former gives that a flow line of a GFF on $\h$ with boundary conditions given by $-a$ (resp.\ $b$) on $\R_-$ (resp.\ $\R_+$) from $0$ to $\infty$ whose angle is sufficiently close to $-(\lambda+b)/\chi$ is likely to be very close to $\R_+$ and the latter gives that two flow lines of a whole-plane GFF from $0$ to $\infty$ whose angles are close are likely to be close to each other.  On a first reading, the proofs in this section can be skipped when reading the rest of this paper.

Let $\varphi \colon \ol{\h} \to \ol{\D}$ be the conformal map given by $z \mapsto (z-i)/(z+i)$.  The \emph{bounded metric} on $\ol{\h}$ is defined by $d(z,w) = |\varphi(z) - \varphi(w)|$.  The \emph{bounded Hausdorff metric} is the Hausdorff metric on the compact subsets of $\ol{\h}$ with respect to the bounded metric.

\begin{proposition}
\label{prop:bounded_hausdorff}
Fix $\kappa \in (0,4)$ and $a,b \in \R$ with $a+b > -2\lambda$.  Suppose that $h$ is a GFF on $\h$ with boundary conditions given by $-a$ on $\R_-$ and $b$ on $\R_+$.  For each $\theta \in \R$ we let $\eta_\theta$ be the flow line of $h$ from~$0$ to~$\infty$ with angle $\theta$.  Then~$\eta_\theta$ converges to~$\R_+$ in probability in the bounded Hausdorff metric as $\theta \downarrow -(\lambda+b)/\chi$. Furthermore, for any $\dd_0, p \in (0,1)$ and $R > 0$, there exists $\tht_0 > -(\la + b)/\chi$ (depending only on $\kk$, $R$, $\dd_0$, $p$, $a$, and $b$) such that for any $\tht \in (-(\la+b)/\chi, \tht_0)$ the following holds with probability at least $1-p$: for each $x \in [0,R]$, there exists $y \in \eta_\tht \cap [0,R]$ such that $\n{x-y} \leq \delta_0$.
\end{proposition}

We note that the condition $a+b > -2\lambda$ in the statement of Proposition~\ref{prop:bounded_hausdorff} is so that there exists $\theta \in \R$ so that $\eta_\theta$ is non-trivial, meaning that $\eta_\theta \sim \SLE_\kappa(\rho_1; \rho_2)$ with $\rho_1, \rho_2 > -2$.  In the limit $\theta \downarrow -(\lambda+b)/\chi$, we have that $\rho_2 \downarrow -2$.  The second statement in the proposition arises naturally in the proof of the first, and we will make use of both to prove the main results of the paper.  To state our result for flow lines of a whole-plane GFF we first introduce a definition.

\begin{definition}\label{def:delta_close}
Fix $\kk \in (0,4)$ and let $h$ be a whole-plane GFF with values modulo a global multiple of $2\pi\chi$. Let $D$ be a bounded open set containing $0$, let $\dd > 0$ and define $D_\dd = \{z \in \C \colon d(z, D) < \dd\}$. For $\theta \in \R$ let $\eta_{\theta}$ be the flow line of $h$ starting from $0$ with angle $\theta$ and for any set $A$ let $\tau_\tht(A) = \inf\{t \geq 0\colon \eta_\tht(t) \notin A\}$.
For $\tht \neq 0$, $\eta_0$ and $\eta_\tht$ are said to be \emph{$\dd$-close until $\eta_0$ exits $D$} if the following conditions hold. For all $t \in (0, \tau_0(D)]$ there exist $t_1, t_2, t'_1, t'_2$ where $0 < t_1 < t < t_2 < \tau_0(D_\dd)$ and $0 < t'_1 < t'_2 < \tau_\tht(D_\dd)$ such that for each $j = 1,2$, $\eta_0(t_j) = \eta_\tht(t'_j)$ and at this point $\eta_\tht$ hits $\eta_0$ on the \emph{left} side of $\eta_0$  (if $\tht > 0$) with angle gap $\tht$ (as defined in \cite{ms2017ig4}). Furthermore, for any $s \in [t_1, t_2]$ and $s' \in [t'_1, t'_2]$, we have $|\eta_0(s) - \eta_\tht(s')| < \dd$. If $\tht < 0$, we replace `left side' by `right side' above.
\end{definition}

\begin{proposition}\label{prop:intersections_stopped}
     Fix $\kk \in (0,4)$, let $h$ be a whole-plane GFF with values modulo a global multiple of $2\pi\chi$ and let $D$ be a bounded open set containing $0$. Fix $\dd, p \in (0,1)$. Then,  there exists $\tht_0 \in (0,1)$, depending only on $\kappa , \dd, p$ and $D$, such that the following holds with probability at least $1 - p$. For any fixed $\tht \in (0, \tht_0)$, $\eta_0$ and $\eta_\tht$ are $\dd$-close until $\eta_0$ exits $D$.
\end{proposition}

We remark that Definition~\ref{def:delta_close} makes sense in the setting when $\wt{h}$ is instead a GFF on a domain $D_0 \subseteq \C$, $\ov{D}_\dd \subseteq D_0$, and where the flow lines are started from any point $z \in D$ and have angles $\tht_1$ and $\tht_2$.
Proposition~\ref{prop:intersections_stopped} also holds in this case since the laws of $h$ and $\wt{h}$, both restricted to $\ov{D}_\dd$ are mutually absolutely continuous by \cite[Lemma~4.1]{mq2020geodesics} and conformal invariance. We will in particular often apply this proposition in the case that $\wt{h}$ is a GFF on the unit disk.

In order to prove Proposition~\ref{prop:bounded_hausdorff}, we will first work in the setting of a single force point $\SLE_\kappa(\rho)$ process in Section~\ref{subsec:single_force_point} and then extend to the setting of Proposition~\ref{prop:bounded_hausdorff} in Section~\ref{subsec:two_force_point} using an absolute continuity argument. In Section~\ref{subsec:whole_plane} we will use Proposition~\ref{prop:bounded_hausdorff} to prove Proposition~\ref{prop:intersections_stopped}.

\subsection{The case of a single force point}
\label{subsec:single_force_point}

\begin{proposition}
\label{prop:single_force_point}
Fix $\rho > -2$, $\kappa \in (0,4)$, and suppose that $\eta$ is an $\SLE_\kappa(\rho)$ process in $\h$ from $0$ to $\infty$ with the force point located at $0^+$.  Then $\eta$ converges to $\R_+$ in probability as $\rho \downarrow -2$ in the bounded Hausdorff metric.
\end{proposition}

Fix $\kappa \in (0,4)$, $\rho > -2$, and suppose that $\eta$ is an $\SLE_\kappa(\rho)$ process in $\h$ from $0$ to $\infty$ with the force point located at $0^+$.  Then we recall that we can sample from the law of the driving pair $(W,V)$ for $\eta$ in the following way.  Let $X$ be a $\BES^{\delta}$ process with
\begin{equation}
\label{eqn:delta_rho}
\delta = 1 + \frac{2(\rho+2)}{\kappa}
\end{equation}
and then set
\begin{equation}
\label{eqn:v_w_x_formula}
V_t = \frac{2}{\sqrt{\kappa}} \int_0^t \frac{1}{X_s} ds \quad\text{and}\quad W_t = V_t - \sqrt{\kappa} X_t.
\end{equation}

We are going to prove Proposition~\ref{prop:single_force_point} by first showing in Lemmas~\ref{lem:bessel_convergence} and~\ref{lem:w_behavior} that $V_t, W_t \to \infty$ and $\inf_{0 \leq s \leq t} W_s \to 0$ in probability for each fixed $t > 0$ as $\rho \downarrow -2$.  We will then deduce from this in Lemma~\ref{lem:exit_on_right} that this implies that $\eta$ is very likely to exit a thin rectangle on its right side which will imply that at least part of $\eta$ is likely to be close to a segment of $\R_+$.  We will complete the proof of Proposition~\ref{prop:single_force_point} using a time-reversal argument to get that the remainder of $\eta$ must be close to $\R_+$ (since this corresponds to an initial part of the time-reversal). Lemma~\ref{lem:intersect_boundary}, which deals with points where the flow line intersects the boundary, arises naturally in the proof of this result.  Throughout, we let
\[ \rho_0 = \frac{\kappa}{4}-2\]
so that the value of $\delta$ from~\eqref{eqn:delta_rho} is equal to $3/2$ (the value $3/2$ is not special; any fixed value in $(1,2)$ would suffice for what follows).

\begin{lemma}
\label{lem:bessel_convergence}
For each $t > 0$ we have that $V_t \to \infty$ in probability as $\rho \downarrow -2$.
\end{lemma}
\begin{proof}
Fix $\epsilon, t, R > 0$.  Let $\ell$ be the local time process for $X$ at $0$ and let $\tau_\epsilon = \inf\{s \geq 0 : \ell_s = \epsilon\}$.  Then it suffices to show that we have both
\begin{align}
 \p[ \tau_\epsilon \geq t ] \to 0 \quad\text{as}\quad \epsilon \to 0 \quad\text{and} \label{eqn:tau_epsilon_ubd} \\
 \p[ V_{\tau_\epsilon} \geq R] \to 1 \quad\text{as}\quad \rho \downarrow -2. \label{eqn:v_tau_epsilon_lbd}
 \end{align}
 The rate of convergence in~\eqref{eqn:tau_epsilon_ubd} will be uniform in $\rho \in (-2,\rho_0]$.  We first note that the probability that $X$ has an excursion from $0$ of length at least $1$ before time $\tau_\epsilon$ is $O(\epsilon)$,  where the implicit constant is universal.  Indeed,  the Poissonian structure of the excursions that $X$ makes from $0$ (Section~\ref{subsec:bessel}) implies that the above probability is equal to
 \begin{align}
 1 - \exp\left(-\frac{\delta}{2} \epsilon \int_1^{\infty}t^{\frac{\delta}{2}-2}dt \right) \leq \frac{\delta}{2-\delta}\epsilon \leq 2 \epsilon \quad\text{for each}\quad  \delta \in (1 ,  3/2). \label{eqn:big_jumps_bound}
 \end{align}
Moreover the expected sum of the lengths of excursions that $X$ makes from $0$ of length at most $1$ by time $\tau_{\epsilon}$ is equal to 
 \begin{align}
 \E\left[ \sum_{(u,t) \in \Lambda,  u \leq \epsilon,  t \leq 1} t \right] = \frac{\delta}{2}\epsilon \int_0^1 t \cdot  t^{\frac{\delta}{2}-2}dt = \epsilon. \label{eqn:small_jumps_bound}
\end{align}
Combining~\eqref{eqn:big_jumps_bound} with~\eqref{eqn:small_jumps_bound} and applying Markov's inequality we thus have that
\begin{align*}
\p[\tau_{\epsilon} \geq t] \leq 2\epsilon + \frac{\epsilon}{t} = O(\epsilon),
\end{align*}
which proves~\eqref{eqn:tau_epsilon_ubd}.

We now turn to the second assertion.  For each $k \geq 0$ we let $N_k^\epsilon$ be the number of excursions that $X$ makes from $0$ of length in $[2^{-k-1},2^{-k})$ and with maximum in $[0,2^{-k/2}]$ by time $\tau_\epsilon$.  Then we have that
\begin{align*}
V_{\tau_\epsilon}
&= \frac{2}{\sqrt{\kappa}} \int_0^{\tau_\epsilon} \frac{1}{X_s} ds 
  \geq \frac{2}{\sqrt{\kappa}} \sum_{k=0}^\infty N_k^\epsilon \cdot 2^{-k-1} \cdot 2^{k/2}
 = \frac{1}{\sqrt{\kappa}} \sum_{k=0}^\infty 2^{-k/2} N_k^\epsilon. 
\end{align*}

Let $p$ be the transition kernel for a $\BES^{4-\delta}$ process starting from $0$.  We claim that the density for the law of a Bessel excursion of dimension $\delta$ from $0$ to $0$ of length $1$ at time $1/2$ is given by
\begin{equation}\label{eqn:density_of_excursion}
 f(x) = \lim_{\epsilon \to 0} \frac{p_{1/2}(0,x) p_{1/2}(x,\epsilon)}{p_1(0,\epsilon)}.
 \end{equation}
Indeed,  \cite[Chapter XI,  Exercise 3.6]{revuz2013continuous} implies that the law of a Bessel bridge of dimension $4-\delta$ from $0$ to $0$ of length $1$ is given by $Z_u = (1-u) Y_{u/(1-u)}$ for $0 \leq u <1$, where $Y$ is a $\BES^{4-\delta}$ process starting from $0$.  Then it is easy to see that $(Z_u)_{0 \leq u <1}$ is a solution to~\eqref{eqn:bessel_excursion_sde} for $t=1$ by e.g.,  applying It\^o's formula.  Moreover,  \cite[Chapter XI]{revuz2013continuous} also implies that the density of $Z_{1/2}$ is given by~\eqref{eqn:density_of_excursion} from which the claim follows.  Combining with the explicit form of the transition kernel of a $\BES^{4-\delta}$ process (Section~\ref{subsec:bessel}),  it follows that there exists a universal constant $p_0 \in (0,1)$ so that for all $\delta \in (1,3/2)$ the probability that a $\BES^\delta$ excursion from $0$ to $0$ of length $1$ exceeds $2$ is at least $p_0$. By Brownian scaling, it follows that for all $\delta \in (1,3/2)$ the probability that a $\BES^\delta$ excursion from $0$ to $0$ of length $t \in [2^{-k-1},2^{-k}]$ exceeds $2^{-k/2}$ is also at least $p_0$. Note that if $Z$ is a Poisson random variable with mean $\lambda$, we have that
\begin{align}
\p[ Z \geq \alpha \lambda ] &\leq \exp(\lambda(\alpha - \alpha \log \alpha -1)] \quad\text{for all}\quad \alpha > 1 \quad\text{and} \label{eqn:poisson1}\\
\p[ Z \leq \alpha \lambda ] &\leq \exp(\lambda(\alpha - \alpha \log \alpha -1)] \quad\text{for all}\quad \alpha \in (0,1). \label{eqn:poisson2}
\end{align}
Note also that $N_k^{\epsilon}$ is a Poisson random variable with mean $\lambda_\epsilon$ which is at least
\begin{align*}
p_0 \frac{\delta}{2} \epsilon \int_{2^{-k-1}}^{2^{-k}} t^{\frac{\delta}{2}-2}dt.
\end{align*}
It follows that there exists a universal constant $c_0 > 0$ so that $\lambda_\epsilon \geq c_0 \epsilon 2^{(1-\delta/2)k}$  for each $k \in \N$ and $\delta \in (1 , 3/2)$.  By~\eqref{eqn:poisson2}, this implies that there exist universal constants $c_1,c_2>0$ such that
\begin{align}
\label{eqn:n_k_epsilon_tail_bound}
\p[N_k^{\epsilon} \leq c_1 \epsilon 2^{(1-\delta/2) k}]\leq \p[N_k^{\epsilon} \leq \lambda_\epsilon/2] \leq 2 \exp(- c_2 \epsilon 2^{(1-\delta/2)k})
\end{align}
for each $k \in \N$ and $\delta \in (1 ,  3/2)$.  Let $K_0^\epsilon$ be the first $K$ so that $N_k^\epsilon \geq c_1 \epsilon 2^{(1-\delta/2) k}$ for every $k \geq K$.  By applying a union bound to~\eqref{eqn:n_k_epsilon_tail_bound}, we have that
\begin{equation}
\label{eqn:k_0_lbd}
\p[ K_0^\epsilon \geq k] \to 0 \quad\text{as}\quad k \to \infty
\end{equation}
faster than any negative power of $k$ and at a rate which is uniform in $\delta \in ( 1 , 3/2 )$.  Since we have that
\begin{align*}
 V_{\tau_\epsilon} \geq \frac{c_1}{\sqrt{\kappa}} \epsilon \sum_{k=K_0^\epsilon}^\infty 2^{ (1-\delta) k/2},
\end{align*}
it follows from~\eqref{eqn:k_0_lbd} that for fixed $q \in (0,1)$,  there exists $K \in \N$ uniformly in $\delta \in ( 1 ,  3/2 )$ such that with probability at least $1-q$ we have
\begin{align*}
V_{\tau_{\epsilon}} \geq \frac{c_1}{\sqrt{\kappa}} \epsilon \sum_{m=K}^{\infty} 2^{(1-\delta)m/2}.
\end{align*}
Also, there exists $\rho_1 \in (-2,\kappa/2-2)$ such that
\[ \frac{c_1}{\sqrt{\kappa}}\epsilon \sum_{m=K}^{\infty} 2^{(1-\delta)m/2} \geq R \quad\text{for each}\quad  \rho \in (-2,\rho_1)\]
(where $\delta$ and $\rho$ are related as in~\eqref{eqn:delta_rho}) and so~\eqref{eqn:v_tau_epsilon_lbd} follows.  This completes the proof of the lemma.
\end{proof}

\begin{lemma}
\label{lem:w_behavior}
We have for each fixed $t > 0$ that
\[ W_t \to \infty \quad\text{and}\quad \inf_{0 \leq s \leq t} W_s \to 0\]
in probability as $\rho \downarrow -2$.
\end{lemma}
\begin{proof}
First, we recall from~\eqref{eqn:v_w_x_formula} that $W_t = V_t - \sqrt{\kappa}X_t$.  We start by proving the first assertion of the lemma.  Fix $t,M>0$ and $q \in (0,1)$.  By \cite[Chapter XI,  Theorem~1.2]{revuz2013continuous},  we can couple~$X$ in the same probability space with an independent $\BES^{4-\delta}$ process~$X'$ starting from~$0$ so that $Y = \sqrt{X^2 + (X')^2}$ has the law of a $\BES^{4}$ starting from $0$.  Fix $R > 0$ large (to be chosen).  Then,  under the above coupling,  we have that $\p[\sqrt{\kappa}X_t \geq R] \leq \p[\sqrt{\kappa}Y_t \geq R]$ and so we pick $R>0$ sufficiently large and independent of $\delta$ such that $\p[\sqrt{\kappa}Y_t \geq R] \leq q/2$.  Then Lemma~\ref{lem:bessel_convergence} implies that there exists $\rho_1 \in (-2,\kappa/2-2)$ such that $\p[V_t \geq M+R] \geq 1 - q/2$ for every $\rho \in (-2,\rho_1)$ and so 
\begin{align*}
\p[W_t \geq M] \geq \p[V_t \geq M+\sqrt{\kappa}X_t] \geq 1-q
\end{align*}
for every $\rho \in (-2,\rho_1)$.  This proves the first assertion of the lemma.

As for the second assertion of the lemma,  we fix $a,t>0$ and $q \in (0,1)$.  Fix also $s \in (0,t)$ (its exact value will be chosen later).  Then it holds that
\begin{align}
\inf_{0 \leq u \leq t} W_u
&= \min\left( \inf_{s \leq u \leq t}(V_u - \sqrt{\kappa}X_u) ,  \inf_{0 \leq u \leq s}(V_u - \sqrt{\kappa}X_u)\right) \notag\\
&\geq \min\left( V_s - \sqrt{\kappa}\sup_{s \leq u \leq t} X_u ,  -\sqrt{\kappa} \sup_{0 \leq u \leq s} X_u \right). \label{eqn:inf_lbd}
\end{align}
We pick $s \in (0,t)$ small enough that 
\begin{align}
\p[\sqrt{\kappa}\sup_{0 \leq u \leq s} X_u \leq a] \geq \p[\sqrt{\kappa}\sup_{0 \leq u \leq s} Y_u \leq a] \geq 1 - \frac{q}{3} \label{eqn:inf_bound1}
\end{align}
for each $\delta \in (1,3/2)$.  Also,  there exists $R>a$ sufficiently large such that
\begin{align}
\p[\sqrt{\kappa}\sup_{s\leq u \leq t} X_u \leq R] \geq \p[\sqrt{\kappa}\sup_{s \leq u \leq t} Y_u \leq R] \geq 1 - \frac{q}{3} \label{eqn:inf_bound2}
\end{align}
for each $\delta \in (1,3/2)$.  Moreover Lemma~\ref{lem:bessel_convergence} implies that there exists $\rho_1 \in (-2,\kappa/2-2)$ such that
\begin{equation}
\label{eqn:inf_bound3}
\p[V_s \geq 2R] \geq 1 - q/3
\end{equation}
for each $\rho \in (-2,\rho_1)$.  Combining \eqref{eqn:inf_bound1}--\eqref{eqn:inf_bound3} with~\eqref{eqn:inf_lbd}, we obtain that $\p[\inf_{0\leq u\leq t} W_u \geq -a] \geq 1-q$ for each $\rho \in (-2,\rho_1)$.  This proves the second assertion and completes the proof of the lemma.
\end{proof}

For each $\epsilon > 0$ we let $R_\epsilon = [-\epsilon^{1/2},1/\epsilon] \times [0,\epsilon]$.  Let $\sigma_\epsilon = \inf\{t \geq 0 : \eta(t) \in \partial R_\epsilon \setminus \R\}$.  We note that $\eta(\sigma_\epsilon)$ can either be in the left, right, or top sides of $\partial R_\epsilon$.  We are now going to show that the probability that $\eta(\sigma_\epsilon)$ is in the right side of $\partial R_\epsilon$ tends to $1$ as $\rho \downarrow -2$.

\begin{lemma}
\label{lem:exit_on_right}
For each $\epsilon > 0$, the probability that $\eta(\sigma_\epsilon)$ is in the right side of $\partial R_\epsilon$ tends to $1$ as $\rho \downarrow -2$.
\end{lemma}
\begin{proof}
Let $\zeta_\epsilon = \inf\{t \geq 0 : \im(\eta(t)) = \epsilon\}$.  
First we note that \cite[Lemma~1]{lalley2009geometric} implies that $\hcap(\eta([0,\zeta_{\epsilon}])) \geq \epsilon^2 / 2$.  It therefore follows that if $\eta$ exits $R_\epsilon$ through the top of $\partial R_\epsilon$ then $\sigma_\epsilon = \zeta_\epsilon \geq  \epsilon^2 / 4$ (recall that $\hcap(\eta([0,t])) = 2t$ for all $t \geq 0$).

Fix $t \geq 0$ and let $\h_t$ be the unbounded component of $\h \setminus \eta([0,t])$.  Let us first write down a formula for $W_t$.  Suppose that we are on the event that $W_t \geq 0$.  Let $B$ be a Brownian motion in $\C$ which is independent of $\eta$ and let $\tau = \inf\{t \geq 0 : B_t \notin \h_t\}$ and $\tau_\h = \inf\{t \geq 0 : B_t \notin \h\}$.  Then we have that
\begin{align}\label{eq:bm_hitting_expression}
W_t = \lim_{y \to \infty} \pi y \int_0^{W_t/y} \frac{1}{\pi(1+s^2)} ds = \lim_{y \to \infty} \pi y \left( \frac{1}{2} - \p_{iy}[ B_{\tau_\h} \in [W_t,\infty) ] \right).
\end{align}
Let $A_t$ be the part of $\partial \h_t$ which is to the right of $\eta(t)$.  Let $(g_t)$ be the Loewner flow associated with $\eta$ and write $g_t = u_t + i v_t$.  Using the conformal invariance of Brownian motion, we have that
\begin{align}
\label{eqn:bm_conf_inv}
\p_{iy}[ B_\tau \in A_t]
&= \p_{g_t(i y)}[ B_{\tau_\h} \in [W_t,\infty) ]
 = \p_{i v_t(iy)}[ B_{\tau_\h} \in [W_t - u_t(iy),\infty) ].
\end{align}
As $y \to \infty$, we have both
\begin{equation}
\label{eqn:mapping_out_asymptotics}
v_t(iy) \sim y \quad\text{and}\quad y u_t(iy) \to 0.
\end{equation}
Combining~\eqref{eq:bm_hitting_expression} and~\eqref{eqn:bm_conf_inv} with~\eqref{eqn:mapping_out_asymptotics} proves that
\[ W_t = \lim_{y \to \infty} \pi y \left(\frac{1}{2}- \p_{iy}[ B_\tau \in A_t] \right).\]
A similar formula holds if $W_t \leq 0$.

Now we fix $a_\epsilon ,  b_\epsilon > 0$ depending only on $\epsilon$ (to be chosen) and let $\zeta_{\epsilon}$ be the first time that $W$ hits $\{-a_\epsilon,b_\epsilon\}$.  We claim that for an appropriate choice of $a_\epsilon$ and $b_\epsilon$,  we have that $\eta$ exits $R_{\epsilon}$ in the right side of $\partial R_{\epsilon}$ if $\zeta_{\epsilon} < \epsilon^2 / 4$ and $W_{\zeta_{\epsilon}} = b_\epsilon$.  Note that
\begin{equation}\label{eqn:first_exit_condition}
\p[W_{\zeta_{\epsilon}} = b_\epsilon ,  \zeta_{\epsilon} < \epsilon^2 / 4] \to 1 \quad \text{as}\quad \rho \downarrow -2
\end{equation}
by Lemma~\ref{lem:w_behavior} and so proving the claim would complete the proof of the lemma.

To prove the claim, suppose that $\zeta_\eps < \eps^2/4$ and $W_{\zeta_\eps} = b_\epsilon$. Suppose first that $\eta([0,\zeta_{\epsilon}]) \subseteq R_{\epsilon}$.  Let $g_{R_\epsilon} \colon \h \setminus R_\epsilon \to \h$ be the unique conformal map with $g_{R_\epsilon}(z) - z \to 0$ as $z \to \infty$.  Then for each $y>0$ sufficiently large we have that
\begin{align}
\p_{iy}[B_{\tau} \in A_{\zeta_{\epsilon}}]
&\geq \p_{iy}[B_{\tau} \in [\epsilon^{-1},\infty)] = \p_{g_{R_{\epsilon}}(iy)}[B_{\tau_\h} \in [g_{R_{\epsilon}}(\epsilon^{-1}),\infty)] \notag\\
&\geq \p_{g_{R_{\epsilon}}(iy)}[B_{\tau_\h} \in [2\epsilon^{-1},\infty)] \quad\text{(by \cite[Corollary~3.44]{lawler2008conformally})}. \label{eqn:bm_exit_lbd}
\end{align}
Moreover, by~\eqref{eqn:mapping_out_asymptotics} applied to $g_{R_\epsilon}$ we have that $v_{R_\eps}(iy) \sim y$ and $u_{R_\eps}(iy) \to 0$ as $y \to \infty$ where $g_{R_{\epsilon}} = u_{R_{\epsilon}} + iv_{R_{\epsilon}}$.  It follows that
\begin{align*}
\lim_{y \to \infty}\left( \pi y \n{\p_{g_{R_{\epsilon}}(iy)}[B_{\tau_{\h}} \in [2\epsilon^{-1},\infty)] - \p_{iy}[B_{\tau_{\h}} \in [2\epsilon^{-1},\infty)]}\right) = 0
\end{align*}
and hence (using~\eqref{eq:bm_hitting_expression} with $W_{\zeta_\eps}$ and later $2\eps\nv$ in place of $W_t$ as well as~\eqref{eqn:bm_exit_lbd})
\begin{align*}
W_{\zeta_{\epsilon}} = b_\epsilon \leq \limsup_{y \to \infty} \left( \pi y \left(\frac{1}{2} - \p_{iy}[B_{\tau_{\h}} \in [2\epsilon^{-1},\infty)]\right)\right) = 2\epsilon^{-1}.
\end{align*}
This leads to a contradiction if we choose $b_\epsilon$ so that $b_\epsilon > 2\epsilon^{-1}$.

Therefore,  we must have that $\eta([0,\zeta_{\epsilon}]) \nsubseteq R_{\epsilon}$ which implies that $\sigma_{\epsilon} < \zeta_{\epsilon}$.  Since $\zeta_{\epsilon} < \epsilon^2 / 4$,  it follows that $\eta(\sigma_{\epsilon})$ lies either on the left or the right side of $\partial R_{\epsilon} \setminus \R$.  If it lies on the right side,  then the claim of the lemma holds.  Suppose that it lies on the left side.  Note that there exists a universal constant $c_1>0$ such that for all $y>0$ sufficiently large,  with probability at least $1 - c_1 \epsilon/y$,  a Brownian motion in $\C$ starting from $iy$ exits $\h$ without hitting $B(-\epsilon^{1/2} ,  2\epsilon)$.  It follows that
\begin{align*}
\p_{iy}[B_{\tau} \in A_{\sigma_{\epsilon}} ] \geq -\frac{c_1 \epsilon}{y} + \p_{iy}[ B_{\tau_{\h}} \in [-\epsilon^{1/2} ,  \infty)]
\end{align*}
and so
\begin{align*}
W_{\ss_\eps} = \lim_{y \to \infty} \left(\pi y \left(\frac{1}{2} - \p_{iy}[B_{\tau} \in A_{\sigma_{\epsilon}}]\right)\right)&\leq c_1 \pi \epsilon + \lim_{y \to \infty}\left(\pi y \left( \frac{1}{2} - \p_{iy}[B_{\tau_{\h}} \in [-\epsilon^{1/2},\infty)]\right)\right)\\
&\leq c_1 \pi \epsilon -\epsilon^{1/2}.
\end{align*}
Therefore,  we obtain a contradiction if $a_\epsilon <- c_1 \pi \epsilon + \epsilon^{1/2}$ since $-a_\epsilon \leq W_{\sigma_{\epsilon}}$.  Thus $\eta(\sigma_{\epsilon})$ lies to the right side of $\partial R_{\epsilon} \setminus \R$ and so this completes the proof of the lemma.
\end{proof}

\begin{lemma}
\label{lem:intersect_boundary}
For each $\dd, p \in (0,1)$ and $R > 0$, there exists $\rho_1 > -2$ (depending only on $\kk$, $R$, $\dd$ and $p$) such that for any $\rho \in (-2, \rho_1)$ the following holds with probability at least $1-p$: for each $x \in [0,R]$, there exists $y \in \eta \cap [0,R]$ with $|x-y| \leq \delta$ and $\eta$ hits $y$ before hitting the vertical line $L_{R+1} = \{\re(z) = R+1\}$.
\end{lemma}
\begin{proof}

 Fix $\dd, p \in (0,1)$ and $R > 0$.
For any $\rho \in (-2, \kappa/2-2)$, the process $\eta$ can be viewed the flow line of angle $\tht = \rho\la/\chi$ of a $\text{GFF}$ $h$ on $\h$ with boundary conditions given by $-\lambda(1 + \rho)$ on $\R_-$ and $\lambda$ on $\R_+$.  Set $\xi = \dd/2$ and let $k = R/\xi$.  By decreasing the value of $\delta$ if necessary, we may assume that $k \in \N$.  For each $j = 1,\ldots,k$, set $x_j  = j\xi$ and for $\phi \in \R$ and $1 \leq j \leq k-1$, let $\eta_j^\phi$ be the flow line of $h$ starting from $x_j$ with angle $\phi$. First, we note that the law of $\eta_j^\phi$ is that of an $\SLE_{\kappa}(\rho_2^L,\rho_1^{L};\rho_1^R)$ process in $\h$ from $x_j$ to $\infty$ where
\[ \rho_2^L = \rho + 2,\quad \rho_1^L =  -2 - \frac{\phi \chi}{\la},\quad\text{and}\quad \rho_1^R = \frac{\phi \chi}{\la},\]
and the force points are located at $0$, $x_j^-$, and $x_j^+$ respectively. Fix $\phi > -2\la/\chi$ small enough that $\rho_1^L \geq \kappa/2 - 2$ and $\rho_1^R \in (-2, \kappa/2-2)$ and for brevity write $\eta_j$ for each $\eta_j^{\phi}$.
This implies that $\eta_j$ can be drawn up until reaching $\infty$, that it intersects $(x_j,\infty)$ a.s.\ and that it does not intersect $(0, x_j)$ a.s.  We choose $\phi$ in a way which does not depend on $\rho$ so that $\rho_1^L$ and $\rho_1^R$ do not depend on $\rho$.  Note that $\phi > \theta = \rho \lambda/\chi$ for $\rho$ sufficiently close to $-2$.

Now fix $\rho \in (-2, \kappa/2 - 2)$ and $q \in (0,1)$. For $r > 0$ and $1 \leq j \leq k-1$ let $E_j(r)$ be the event that $\eta_j$ hits the horizontal line $H_r := \{z \colon \im(z) = r\}$ and subsequently hits $\R_+ \cap (x_j, x_{j+1})$ before exiting $B(x_j, \xi/4$). Let $E(r)$ be the event that $E_j(r)$ holds for every $1 \leq j \leq k-1$. We first show that it is possible to choose $r > 0$ small enough that $E(r)$ occurs with probability at least $1 - q$.
 
Up until the first time $\eta_j$ gets within some fixed positive distance of $\R_-$, its law is absolutely continuous with respect to that of an $\SLE(\rho_1^L; \rho_1^R)$ process started at $x_j$ and targeted at $\infty$. In particular, by our choice of $\phi$, $\eta_j$ a.s.\ intersects $(x_j,x_j + \dd_0)$ for any $\dd_0 > 0$ and does so before exiting $B(x_j, 2\dd_0)$. Since $\eta_j$ is a.s.\ a continuous curve which does not trace $\R$, there must exist some $r_j > 0$ such that $\eta_j$ hits $H_{r_j}$ and then hits $(x_j, x_{j+1})$, all before it exits $B(x_j, \xi/4)$.  By the a.s.\ existence of these $r_j$, it follows that there exists some $r > 0$ such that $\p[E(r)] \geq 1-q$.

Notice that $E(r)$ depends only on $h$ restricted to $A = [\xi/4, R] \times [0,1]$.  As we vary $\rho \in (-2,  \kk/2 - 2)$, the boundary values of $h$ on $\R_+$ remain fixed and the boundary values of $h$ on $\R_-$ remain in some bounded interval. Therefore, by \cite[Remark 3.5]{ms2016imag1}, it follows that the laws of $h|_A$ as $\rho$ varies in $(-2, \kk/2 - 2)$ are mutually absolutely continuous with a uniformly controlled Radon--Nikodym derivative. Hence, we can choose $r > 0$ (not depending on $\rho$) in such a way that $\p[E(r)] \geq 1-q/2$ for all $\rho \in (-2, \kappa/2 - 2)$.

Now we complete the proof. Choose $\eps > 0$ such that $\eps\nv > R$ and $\eps < r$, which can be done independently of any choice of $\rho$. Using Lemma~\ref{lem:exit_on_right}, choose $\rho_1$ such that for all $\rho \in (-2, \rho_1)$, the probability that $\eta$ exits $R_\eps$ on the right side is at least $1 - q/2$. Also, ensure $\rho_1$ is small enough that $\theta < \phi$ for all $\rho \in (-2, \rho_1)$. In this case, with probability at least $1-q$, $E(r)$ occurs and $\eta$ exits $R_\eps$ on its right boundary, and hence exits the rectangle $[-\eps^{1/2}, R] \times [0, r]$ on its right boundary. In particular, this ensures that since $E_j(r)$ occurs, $\eta$ must intersect $\eta_j$ before the latter hits $H_r$. By the flow line interaction rules \cite{ms2016imag1}, $\eta$ crosses $\eta_j$  (from left to right) the first time they meet and subsequently does not cross it again. By the definition of $E_j(r)$, $\eta_j$ then hits $(x_j, x_{j+1})$, meaning that $\eta$ must also hit this interval.

It follows that with probability at least $1-q$, $\eta$ hits every interval $(x_j,x_{j+1})$ for $1 \leq j \leq k-1$ and does so before hitting $L_{R+1}$. In this case, for any $x \in [\xi, R]$, there exists $y \in \eta \cap [0,R]$ such that $|x-y| < \xi$, from which we can conclude that for every $x \in [0,R]$ there exists $y \in \eta \cap [0,R]$ such that $|x-y \leq {\delta}$ and where $\eta$ hits $y$ before $L_{R+1}$.
\end{proof}

\begin{proof}[Proof of Proposition~\ref{prop:single_force_point}]

Fix $q \in (0,1)$ and let $\wt{\eta}$ be the image of $\eta$ under the conformal map $\h \rightarrow \h$,  $z \mapsto -1/z$.  Then by the time-reversal symmetry of $\SLE_\kappa(\rho_1;\rho_2)$ processes for $\kappa \in (0,4)$ and $\rho_1,\rho_2 > -2$ established in \cite[Theorem~1.1]{ms2016imag2},  we have that $\wt{\eta}$ has the law of an $\SLE_{\kappa}(\rho)$ process in $\h$ from $0$ to $\infty$ with the force point located at $0^-$.  Set $\wt{R}_{\epsilon} = [-\epsilon^{-1},\epsilon^{1/2}] \times [0,\epsilon]$ for each $\epsilon \in (0,1)$.  Then, Lemmas~\ref{lem:exit_on_right} and \ref{lem:intersect_boundary} together imply that we can choose $\eps_0$, $R$, and $\dd$ in such a way as to show that there exists $\epsilon_0 \in (0,1)$ such that the following holds.  For each $\epsilon \in (0,\epsilon_0)$, there exists $\rho_1 \in (-2, \kappa/2-2)$ with the property that for each $\rho \in (-2,\rho_1)$,  with probability at least $1-q$ we have that $\eta$ (resp.\ $\wt{\eta}$) intersects $[1,\infty)$ (resp.\ $(-\infty,-1]$) before exiting $R_{\epsilon}$ (resp.\ $\wt{R}_{\epsilon}$) for the first time and exits $R_{\epsilon}$ (resp.\ $\wt{R}_{\epsilon}$) for the first time on the right (resp.\ left) side of $\partial R_{\epsilon}$ (resp.\ $\partial \wt{R}_{\epsilon}$).

Fix $\epsilon$ and $\rho$ as above and suppose that the above events occur.  Note that $\eta$ (resp.\ $\wt{\eta}$) does not hit $1$ (resp.\ $-1$) a.s.\  Let $U$ (resp.\ $\wt{U}$) be the component of $\h \setminus \eta$ (resp.\ $\h \setminus \wt{\eta}$) such that $1 \in \partial U$ (resp.\ $-1 \in \partial \wt{U}$).  Let $\tau$ (resp.\ $\wt{\tau}$) be the last (resp.\ first) time that $\eta$ (resp.\ $\wt{\eta}$) hits $\partial U$ (resp.\ $\partial \wt{U}$).  Equivalently, $\tau$ (resp.\ $\wt{\eta}$) is the first time that $\eta$ (resp.\ $\wt{\eta}$) disconnects $1$ (resp.\ $-1$) from $\infty$.  Fix $z = x + iy \in \eta$.  Suppose that $\eta$ hits $z$ before time $\tau$.  Then $z \in R_{\epsilon}$ and so $d(x,z) \leq 2y \leq 2\epsilon$ where we recall that $d$ is the bounded metric on $\h$.  Suppose that $\eta$ hits $z$ after time $\tau$.  Then since $\wt{\eta}([0,\wt{\tau}])$ is given by the image of $\eta([\tau,\infty))$ under $w \mapsto - 1/w$,  we have that $\wt{z} = -1/z \in \wt{R}_{\epsilon}$.  Also,
\[ \wt{z} =- \left(\frac{x}{x^2 + y^2}\right) + \left(\frac{y}{x^2 + y^2}\right)i\]
and hence $d(w,z) = d(\wt{w},\wt{z}) \leq 2\epsilon$ where $\wt{w} = -x/(x^2 + y^2)$ and $w = -1/\wt{w}$.  Therefore, each point on $\eta$ is within distance $2\eps$ of $\R_+$ with respect to the bounded metric. The same can be shown with the roles of $\eta$ and $\R_+$ exchanged, ensuring that the Hausdorff distance between $\eta$ and $\R_+$ with respect to $d$ is at most $2\epsilon$.  This completes the proof of the proposition.
\end{proof}

\subsection{The case two force points}
\label{subsec:two_force_point}

We now want to extract Proposition~\ref{prop:bounded_hausdorff} from Proposition~\ref{prop:single_force_point}.  We are going to make use of the absolute continuity results for $\SLE_\kappa(\rho)$ processes.   Fix $x_1 < 0$.  By taking ratios of processes of the form~\eqref{eqn:sw_mg}, it follows that if we weight the law of an $\SLE_\kappa(\rho_2)$ process with a single force point at $0^+$ (note that such a process does not disconnect $x_1$ from $\infty$ for $\kappa \in (0,4)$) by
\begin{equation}
\label{eqn:rn_formula}
 M_t = |g_t'(x_1)|^{(8-2\kk{+2\rho_1})\rho_1 / (8\kappa)}  |W_t - g_t(x_1)|^{\rho_1/\kappa} |g_t(x_1) - x_t^2|^{\rho_1 \rho_2 / (2\kappa)}
\end{equation}
where $x_t^2$ the location of the force point corresponding to $0^+$ then we get an $\SLE_\kappa(\rho_1; \rho_2)$ process with force points at $x_1$ and $0^+$.

Let us make the following observation using~\eqref{eqn:rn_formula}.

\begin{lemma}
\label{lem:rn_bound}
Fix $\delta_2 > \dd_1 > 0$ and $\epsilon \in (0,\delta_1^2/2)$.  Fix $R > 0$ and suppose that $\rho_1, \rho_2 \in [-R,R]$.  There exists a constant $C > 0$ depending only on $\kappa$,  $\delta_1, \dd_2$, $\epsilon$, and $R$ so that the following is true.  Suppose that $\eta$ is an $\SLE_\kappa(\rho_1; \rho_2)$ process in $\h$ from $0$ to $\infty$ with force points located at $x_1 = -\delta \in (-\dd_2, -\dd_1)$ and $x_2= 0^+$.  Let $R_\epsilon$, $\sigma_{\epsilon}$ be as in Section~\ref{subsec:single_force_point}.  Then the law of $\eta|_{[0,\sigma_{\epsilon}]}$ is absolutely continuous with respect to that of an $\SLE_\kappa(\rho_2)$ process stopped at the corresponding time with Radon-Nikodym derivative which is at most $C$.
\end{lemma}
\begin{proof}

Let $\wt{\eta}$ be an $\SLE_{\kappa}(\rho_2)$ process in $\h$ from $0$ to $\infty$ with the force point located at $0^+$.  We also set $\wt{\sigma}_{\epsilon} = \inf\{t \geq 0 : \wt{\eta}(t) \in \partial R_{\epsilon} \setminus \R\}$ and let $(g_t)$, $M$, and $(x_t^2)$ be as above.  Then the law of $\wt{\eta}|_{[0,\wt{\sigma}_{\epsilon}]}$ weighted by $M_{\wt{\sigma}_{\epsilon}}$ is that of an $\SLE_{\kappa}(\rho_1 ; \rho_2)$ process in $\h$ from $0$ to $\infty$ with the force points located at $-\delta$ and $0^+$ respectively,  and stopped at the first time that it hits $\partial R_{\epsilon} \setminus \R$. 

To prove the claim of the lemma,  it suffices to give upper and lower bounds depending only on $\delta_1$, $\dd_2$, and $\epsilon$ on $|g'_{\wt{\sigma}_{\epsilon}}(-\delta)|$,  $|W_{\wt{\sigma}_{\epsilon}} - g_{\wt{\sigma}_{\epsilon}}(-\delta)|$, and $|g_{\wt{\sigma}_{\epsilon}}(-\delta) - x_{\wt{\sigma}_{\epsilon}}^2|$.  First,  we note that $\text{rad}(\wt{\eta}([0,\wt{\sigma}_{\epsilon}])) \leq 2\epsilon^{-1}$ and so \cite[Corollary~3.44]{lawler2008conformally} implies that $|g_{\wt{\sigma}_{\epsilon}}(z)-z| \leq 6\epsilon^{-1}$ for all $z$ in the unbounded connected component of $\h \setminus \wt{\eta}([0,\wt{\sigma}_{\epsilon}])$.  It follows that there exists a constant $c_1$ depending only on $\dd_1$, $\dd_2$, and $\eps$ such that $\n{g_{\wt{\sigma}_{\epsilon}}(-\delta)},  |W_{\wt{\sigma}_{\epsilon}}|, |x_{\wt{\sigma}_{\epsilon}}^2|$ and $|g_{\wt{\sigma}_{\epsilon}}(-\delta)-W_{\wt{\sigma}_{\epsilon}}|,  |g_{\wt{\sigma}_{\epsilon}}(-\delta)-x_{\wt{\sigma}_{\epsilon}}^2|$ are all bounded by $c_1$.

It remains to bound $|g'_{\wt\ss_{\eps}}(-\dd)|$. If $\tau$ is the first time that a Brownian motion in $\C$ exits $\h \setminus \wt{\eta}([0,\wt{\sigma}_{\epsilon}])$,  we have that 
\begin{align*}
\p_{iy}[B_{\tau} \in [-\delta,-\epsilon^{1/2}]]\geq \p_{g_{R_{\epsilon}}(iy)}[B_{\tau_{\h}} \in [g_{R_{\epsilon}}(-\delta),g_{R_{\epsilon}}(-\epsilon^{1/2})]]
\end{align*}
for all $y>0$ sufficiently large.  
Also \cite[Lemma~3.52]{lawler2008conformally} implies that there exists a universal constant $c>0$ such that for all $\epsilon \in (0,1)$, we have $|g_{R_{\epsilon}}(z)-z| \leq c \epsilon^{3/4}$ for each $z \in \h \setminus R_{\epsilon}$ which implies that $[-\delta + c\epsilon^{3/4},   -\epsilon^{1/2} - c\epsilon^{3/4}] \subseteq [g_{R_{\epsilon}}(-\delta),g_{R_{\epsilon}}(-\epsilon^{1/2})]$ and hence
\begin{align*}
W_{\wt{\sigma}_{\epsilon}}-g_{\wt{\sigma}_{\epsilon}}(-\delta)&\geq g_{\wt{\sigma}_{\epsilon}}(-\epsilon^{1/2}) - g_{\wt{\sigma}_{\epsilon}}(-\delta) = \lim_{y \to \infty}\left(\pi y \p_{iy}[B_{\tau} \in [-\delta ,  -\epsilon^{1/2}]]\right)\\
&\geq \lim_{y \to \infty}\left(\pi y \p_{g_{R_{\epsilon}}(iy)}[B_{\tau_{\h}} \in [g_{R_{\epsilon}}(-\delta),g_{R_{\epsilon}}(-\epsilon^{1/2})]]\right) \geq \delta - \epsilon^{1/2} - 2c\epsilon^{3/4}>0
\end{align*}
for $\epsilon \in (0,1)$ sufficiently small.  Moreover we have that $x_{\wt{\sigma}_{\epsilon}}^2 - g_{\wt{\sigma}_{\epsilon}}(-\delta) \geq W_{\wt{\sigma}_{\epsilon}} - g_{\wt{\sigma}_{\epsilon}}(-\delta)$ and Koebe's $\frac{1}{4}$-theorem implies that
\begin{align*}
&|g_{\wt{\sigma}_{\epsilon}}'(-\delta)| \geq \frac{|g_{\wt{\sigma}_{\epsilon}}(-\delta)-g_{\wt{\sigma}_{\epsilon}}(-\epsilon^{1/2})|}{4\delta} \gtrsim 1 \quad\text{and}\quad |g_{\wt{\sigma}_{\epsilon}}'(-\delta)| \leq \frac{4|g_{\wt{\sigma}_{\epsilon}}(-\delta)-W_{\wt{\sigma}_{\epsilon}}|}{\delta-\epsilon^{1/2}}\lesssim 1
\end{align*}
where the implicit constants depend only on $\delta_1$, $\dd_2$ and $\epsilon$.  Combining everything,  we obtain that there exists a constant $C<\infty$ depending only on $\kappa$, $\delta_1$, $\dd_2$, $\epsilon$, and $R$ such that $M_{\wt{\sigma}_{\epsilon}} \leq C$ a.s.  This completes the proof.
\end{proof}

\begin{proof}[Proof of Proposition~\ref{prop:bounded_hausdorff}]

Fix $\eps_0, \dd_0, q \in (0,1)$ and $R > 0$. We will show that there exists $\tht_0 > -(\la + b)/\chi$ such that for all $\tht \in (-(\la + b)/\chi, \tht_0)$, with probability at least $1-q$, the Hausdorff distance between $\eta$ and $\R_+$ with respect to the bounded metric is at most $\eps_0$, and that for all $x \in [0,R]$ there exists $y \in \eta \cap [0,R]$ such that $\n{x-y} \leq \delta_0$.

Let $\eta_{\theta}$ be the flow line of $h$ from $0$ to $\infty$ with angle $\theta$. Then $\eta_\tht$ has the law of an $\SLE_\kk(\rho_1;\rho_2)$ process in $\h$ from $0$ to $\infty$ where
\[ \rho_1 = -1 + \frac{a - \tht \chi}{\lambda} \quad\text{and}\quad  \rho_2 = -1 + \frac{b + \theta \chi}{\lambda}\]
and where the force points are located at $0^-$ and $0^+$.
Fix $\phi \in \R$ such that
\[-\frac{\lambda + b}{\chi} < \phi < \min\left( \frac{\lambda+a}{\chi},\frac{\lambda-b}{\chi} - \pi\right)\]
and let $\eta_{\phi}$ be the flow line of $h$ from $0$ to $\infty$ with angle $\phi$ so that in this case $\rho_1 > -2$ and $\rho_2 \in (-2, \frac\kk2 - 2)$.
We note that the flow lines interaction rules \cite[Theorem~1.7]{ms2017ig4} imply that $\eta_{\theta}$ lies to the right of $\eta_{\phi}$ if $\theta < \phi$.

Fix $\eps_1, \dd > 0$ small enough (depending only on both $\ee_0$, $\dd_0$ and satisfying conditions we will mention later) and suppose that $-(\lambda+b)/\chi < \theta < \phi$. For $\delta \in (0,1)$ we set $x_{\delta} = \inf\{[\delta,\infty) \cap \eta_\phi\}$ and $y_{\delta} = \inf\{[\delta,\infty) \cap \eta_{\theta}\}$,  and $\tau_{\delta} = \inf\{ t \geq 0 : \eta_{\theta}(t) = y_{\delta}\}$. We note that $\delta \leq y_{\delta} \leq x_{\delta}$.
We can choose $\delta \in (0,1)$ such that with probability at least $1 - q/2$,  we have that the hull of $\eta_{\phi}$ stopped at the first time it hits $x_{\delta}$ is contained in $[-\eps_1,\eps_1] \times [0,\eps_1]$.  Hence,  for each $\theta$ as above we have that with probability at least $1 - q/2$,  the hull of $\eta_{\theta}$ stopped at the time it hits $y_{\delta}$ is contained in $[-\eps_1,\eps_1] \times [0,\eps_1]$.  From now on,  we assume that the latter event, which we call $E_1$, occurs.

In this case, define $\wt{\eta}_{\theta} = g_{\tau_{\delta}}(\eta_{\theta}) - W_{\tau_{\delta}}$ which has the law of an $\SLE_{\kappa}(\rho_1;\rho_2)$ process in $\h$ from $0$ to $\infty$ with the same values of $\rho_1, \rho_2$ as $\eta_\tht$, and with force points located at $x_1 < 0$ and at $0^+$. Let $\tau$ be the first time that a Brownian motion in $\C$ exits the complement in $\h$ of the hull of $\eta_\tht([0,\tau_\dd])$.  
Then
\begin{align*}
\lim_{y \to \infty}\left( \pi y \p_{iy}[B_{\tau}\,\,\text{lies on the left side of $\eta_\tht([0,\tau_\dd])$}\right) \geq \lim_{y \to \infty}\left( \pi y \p_{iy}[B_{\tau_{\h}} \in [0,\delta]]\right) = \delta.
\end{align*}
Therefore we can conclude that $x_1 \leq -\dd$. Note that $\text{rad}(\eta_{\theta}([0,\tau_{\delta}])) \leq 2\eps_1$ so \cite[Corollary~3.44]{lawler2008conformally} and
arguing as in the proof of Lemma~\ref{lem:rn_bound} implies that $\n{f_{\tau_\dd}(z) - z} \leq 13\eps_1$ for all $z$ in the unbounded component of $\h\sm\eta_\tht([0, \tau_\dd])$ where $f_{\tau_{\delta}}(z) = g_{\tau_{\delta}}(z) - W_{\tau_{\delta}}$.
Hence, $|f_{\tau_{\delta}}(z)| \leq 15\eps_1$ for all $z \in [-\eps_1,\eps_1] \times [0,\eps_1]$ such that $z$ lies in the unbounded connected component of $\h \setminus \eta_{\theta}([0,\tau_{\delta}])$ and, in particular, the force point of $\wt{\eta}_{\theta}$ with weight $\rho_1$ lies in $[-15\eps_1,-\delta]$.

Let $M$ be the Radon-Nikodym derivative of the law of $\wt{\eta}_{\theta}$ stopped at the first time that it exits $R_{\eps_1}$ (as defined in Section~\ref{subsec:single_force_point}) with respect to the law of an $\SLE_{\kappa}(\rho_2)$ process in $\h$ from $0$ to $\infty$ with the force point located at $0^+$ and stopped at the first time that it exits $R_{\delta^2/2}$. Then Lemma~\ref{lem:rn_bound} implies that there exists a constant $C<\infty$ depending only on $\kappa$, $\phi$, $\delta$, $\eps_0$, $\eps_1$, $a$, and $b$ such that $M \leq C$ a.s.
Therefore, combining with Lemmas~\ref{lem:exit_on_right} and \ref{lem:intersect_boundary} we obtain that there exists $\theta_0 \in (-(\lambda+b)/\chi,\phi)$ such that for every $\theta \in (-(\lambda+b)/\chi,\theta_0)$, with probability at least $1 - q/2$, the following event $E_2$ occurs.
Firstly, $\wt\eta_\tht$ exits $R_{\eps_1}$ through its right side, and secondly, for every $\wt x \in [0,R-20\eps_1]$, there exists $\wt y \in \wt\eta_\tht \cap [0,R]$ such that $\n{\wt x- \wt y} \leq \dd$ and $\wt\eta_\tht$ hits $\wt y$ before hitting the vertical line $L_{R - 20\eps_1 + 1}$.

Therefore, for each $\tht \in (-(\la+b)/\chi,\theta_0)$, with probability at least $1-q$, both $E_1$ and $E_2$ occur.
We can choose $\eps_1, \dd > 0$ small enough, both depending only on both $\eps_0$ and $\dd_0$ such that the following all hold.
First,  that $\wt\eta_\tht$ exiting $R_{\delta^2/2}$ on its right side guarantees that $\eta_\tht$ exits $R_{\eps_0}$ on its right side.
Also, if $\eps_1, \dd$ are small enough then we can conclude from $E_1$ and $E_2$ that for every $x \in [0, R]$ there exists $y \in \eta_\tht \cap [0,R]$ such that $\n{x - y} \leq \dd_0$ and that $\eta_\tht$ hits $y$ before $L_{R+1}$.

It follows that for each $\tht \in (-(\la+b)/\chi,\theta_0)$, with probability at least $1-q$, $\eta_\tht$ exits $R_{\eps_0}$ on its right side and for all $x \in [0,R]$, there exists $y \in \eta_\tht \cap [0,R]$ such that $\n{x - y} \leq \dd_0$, where $\eta_\tht$ hits each such $y$ before $L_{R+1}$. It follows for the correct choice of $R, \dd_0$ and $\eps_0$ that we can show that with probability at least $1-q$, $\eta_\tht$ exits $R_{\eps_0}$ on its right side and hits $[1, \infty)$ before doing so.

Finally,  to complete the proof,  we let $\wh{\eta}_{\theta}$ be the image of $\eta_{\theta}$ under the conformal map $\h \rightarrow \h$,  $z \mapsto -1/z$ and set $\wt{R}_{\eps_0} = [-\epsilon_0^{-1},\epsilon_0^{1/2}] \times [0,\epsilon_0]$.  Then possibly by taking $\theta > -(\lambda+b)/\chi$ to be smaller,  we can assume that with probability at least $1-q$,  the curve $\wh{\eta}_{\theta}$ exits $\wt R_\eps$ for the first time on its left side and intersects $[-\epsilon^{-1},-1]$ before doing so.  This follows by combining the results of the previous paragraph with the fact that $\wh{\eta}_{\theta}$ has the law of an $\SLE_{\kappa}(\rho_2;\rho_1)$ process in $\h$ from~$0$ to~$\infty$ with the force points located at $0^-$ and $0^+$ respectively by \cite[Theorem~1.1]{ms2016imag2}.  We then conclude the proof by arguing as at the end of the proof of Proposition~\ref{prop:single_force_point}.
\end{proof}

\subsection{The case of the whole-plane GFF}
\label{subsec:whole_plane}

In this section we will prove Proposition~\ref{prop:intersections_stopped}. The first step is proving the following weaker version of the proposition. We remark that the conclusions of following lemma, and of Proposition~\ref{prop:intersections_stopped}, are also valid when the roles of $\eta_0$ and $\eta_\tht$ are exchanged, or when we instead work with $\eta_{\tht_1}$ and $\eta_{\tht_2}$ provided $|\tht_1 - \tht_2| < \tht_0$.
We will now use Proposition~\ref{prop:bounded_hausdorff} to prove the following weaker version of lemma of Proposition~\ref{prop:intersections_stopped}.

\begin{lemma}\label{lem:intersections_times}
    Fix $\kk \in (0,4)$ and let $h$ be a whole-plane GFF with values modulo a global multiple of $2\pi\chi$. Fix $\dd, p \in (0,1)$. Then,  there exists $\tht_0 \in (0,1)$ depending only on $\kappa , \dd$ an $p$ such that the following holds with probability at least $1 - p$ for any fixed $\tht \in (0, \tht_0)$. For all $t \in (0, \tau_0(\D)]$ there exist $t_1, t_2, t'_1, t'_2$ such that $0 < t_1 < t < t_2 < \tau_0(\D_\dd)$ and $0 < t'_1 < t'_2$, and for each $j = 1,2$, $\eta_0(t_j) = \eta_\tht(t'_j)$ and at this point $\eta_\tht$ hits $\eta_0$ on the left side of $\eta_0$ with angle gap $\tht$. Furthermore, for any $s \in [t_1, t_2]$ we have $|\eta_0(s) - \eta_0(t)| < \dd$.
\end{lemma}

The reason this is weaker than Proposition~\ref{prop:intersections_stopped} is that this lemma does not show that $t'_2 \leq \tau_\tht(\D_\dd)$ or tell us anything about the flow line $\eta_\tht$ on the intervals $[t'_1, t'_2]$. To obtain this information and to prove Proposition~\ref{prop:intersections_stopped}, we will apply this lemma twice, once as above, and the second time with the roles of $\eta_0$ and $\eta_\tht$ exchanged.

\begin{proof}
\textit{Case 1. $\kk \in (0, 8/3]$.} In this case, $\eta_0$ is a.s.\ simple and transient. Let $\phi \colon \C\setminus \eta_0 \to \h$ be a conformal map fixing $0$ and $\infty$ which is $\ss(\eta_0)$-measurable.
Then we can choose (deterministic) constants $R \in (0,\infty)$ large enough and $\dd_2 \in (0,1)$ small enough such that with probability at least $1-p/2$ we have $\phi(\eta_0(\tau_0(\D))) < R$ and $|\phi\nv(x) - \phi\nv(y)| < \dd$ for all $x,y \in [0,R+1]$ with $|x-y| < \dd_2$. We assume that these events hold.

By \cite{ms2017ig4}, conditional on $\eta_0$, the curve $\eta_\tht$ has the law of an $\SLE(\rho^L(\tht); \rho^R(\tht))$ in $\C \setminus \eta_0$ started from~$0$ and targeted at $\infty$, where 
\begin{equation}\label{eq:rhoL}
    \rho^L(\tht) = -2 + \frac{(2\pi - \tht)\chi}{\la}, \quad \rho^R(\tht) = -2 + \frac{\tht\chi}{\la}.
\end{equation}
Then $\wt\eta_\tht := \phi(\eta_{\tht})$ has the law of a flow line of angle $\tht$ of a GFF $h$ on $\h$ with $\la - 2\pi\chi$ boundary conditions on $\R_-$ and $-\la$ boundary conditions on $\R_+$.
By Proposition~\ref{prop:bounded_hausdorff} there exists $\tht_0 > 0$ such that if $\tht \in (0,\tht_0)$ is fixed, then with probability at least $1 - p/2$, for each $x \in [0,R+1]$ there exists $y \in \wt\eta_\tht \cap [0,R+1]$ such that $|x-y| \leq \delta_2/8$. Suppose this event also holds.

In this case, fix $t \in (0, \tau_0(\D)]$, set $x = \phi(\eta_0(t))$ (when viewed as a prime end on the left side of $\eta_0$) and note that $x \in (0,R]$. It can then be shown that we can find $y_1, y_2 \in \wt{\eta}_\tht \cap (0, R+1)$ with $y_1 < x < y_2$ and $|y_2 - y_1| < \dd_2$ (for $x$ close to $0$ we use the fact that $\wt{\eta}_\tht$ a.s.\ intersects $\R_+$ arbitrarily close to $0$, for other $x$ we use the conclusion on Proposition~\ref{prop:bounded_hausdorff}).
Define $0 < t_1 < t_2$ and $0 < t'_1 < t'_2$ by $\eta_0(t_j) = \eta_\tht(t'_j) = y_j$ for $j = 1,2$. Note that by our assumption on the behavior of $\phi\nv$ at the boundary, we have that $s \in [t_1, t_2]$ implies $|\eta_0(s) - \eta_0(t)| < \dd$ and similarly that $t_2 < \eta_0(\D_\dd)$.
By construction the angle gap at these intersection points is $\tht$ and $\eta_\tht$ hits $\eta_0$ on the left side of $\eta_0$. We must have $t'_1 < t'_2$ since if $z_1, z_2$ are two intersections with such an angle gap, and $\eta_0$ hits $z_1$ before $z_2$, then $\eta_\tht$ must also hit $z_1$ before $z_2$.

\textit{Case 2. $\kk \in (8/3, 4)$.}  By \cite{ms2017ig4}, $\eta_0$ a.s.\ divides $\C\sm\eta_0$ into countably many connected components $P$, which we refer to as pockets, each of which has an opening point $z$ and a closing point $w$. By \cite[Theorem 1.11, Proposition 3.28]{ms2017ig4}, $\eta_\tht$ visits each of these pockets in order, and in each pocket $P$ has the law of an $\SLE(\rho^L(\tht); \rho^R(\tht))$ process in $P$ started from $z$ and targeted at $w$, where $\rho^L(\tht)$ and $\rho^R(\tht)$ are given as in the previous case by \eqref{eq:rhoL}. In particular, a.s.\ in each pocket $P$, $\eta_\tht$ will intersect the counterclockwise boundary arc of $P$ from the opening to the closing point, which we denote by $\del^+ P$, arbitrarily close to both $w$ and $z$, and will necessary do so by hitting $\eta_0$ on the left side of $\eta_0$ with angle gap $\tht$.

Set $\dd_1 = \dd/8$.
By the almost sure continuity and transience of $\eta_0$, there a.s.\ exists a random integer $N \in \N$ such that there are at most $N$ components $P_1, \dots, P_N$ of $\C \setminus \eta_0$ that intersect $\overline\D$ and have diameter at least $\delta_1$, so if $\diam(P) \geq \delta_1$ then $P$ must be one of these $P_j$. 
For every $1 \leq j \leq N$,  we let $z_j$ (resp.  $w_j$) be the opening (resp.  closing) point of $P_j$ and let $\phi_j$ be a conformal transformation mapping $P_j$ onto $\h$ such that $\phi_j(z_j) = 0 ,  \phi_j(w_j) = \infty$ and such that $\phi_j$ is $\sigma(\eta_0)$-measurable.
Let $\partial P_j^+$ denote the right-hand boundary of $P_j$, by which we mean that part of $\partial P_j$ with $\phi_j(\partial P_j^+) = \R_+$.
Almost surely, there also exists $R \in (0,\infty)$ such that $\phi_j^{-1}([R,\infty)) \subseteq B(w_j , \delta_1)$ for every $1 \leq j \leq N$.  
Moreover,  there exists $\dd_2 > 0$ sufficiently small such that $|\phi_j^{-1}(x) - \phi_j^{-1}(y)| < \delta_1$ for every $x ,  y \in [0,R]$ such that $|x-y| < \delta_2$,  and every $1 \leq j \leq N$.
Let $E(N,R,\dd_2)$ be the event that the above statements hold for a fixed choice of constants $N,R$ and $\delta_2$. We can conclude that there exist some deterministic choice of these constants such that $E_1 := E(N,R,\dd_2)$ holds with probability at least $1 - p/2$.
We assume from now on that this event holds.

Fix $t \in (0, \tau_0(\D)]$. There are now three cases. First, suppose $\eta_0(t)$ is part of $\del^+ P$ where $\diam(P) \leq \dd_1$. Then by the above there will exist intersection points of $\eta_0$ and $\eta_\tht$ arbitrarily close to the opening and closing points of this pocket meaning we can define $t_1, t_2, t'_1, t'_2$ satisfying the conditions of the lemma. A second possibility is that $\eta_0(t)$ is not part of $\del^+ P$ for any pocket $P$. In this case, since $\eta_0$ does not trace itself when it intersects itself (by \cite{ms2017ig4} this a.s.\ does not happen simultaneously for all parts of the curve) there will exist times $t^- < t < t^+$ arbitrarily close to $t$ such that times $\eta_0(t^-)$ (resp. $\eta_0(t^+)$) is the closing (resp. opening) point of some pocket $P^-$ (resp. $P^+$). We can again find intersection times $t_1, t_2, t'_1, t'_2$ satisfying the conditions of the lemma.

The final possibility is that $\eta_0(t)$ is part of the clockwise boundary arc of one of the pockets $P_j$ where $\diam(P_j) \geq \dd_1$.
Fix $\tht > 0$ and work conditional on $E_1$ so that if $\diam(P) \geq \dd_1$ above then $P$ must be one of the components $P_j$.
Within each component $P_j$, $\eta_\tht$ has the law of an $\SLE(\rho^L(\tht); \rho^R(\tht))$ in $P_j$ started from $z_j$ and targeted at $w_j$, where $\rho^L(\tht)$ and $\rho^R(\tht)$ again by \eqref{eq:rhoL}. Equivalently, $\eta_j^{\tht} := \phi_j(\eta_\tht)$ has the law of a flow line of angle $\tht$ of a GFF $h$ on $\h$ with $\la - 2\pi\chi$ boundary conditions on $\R_-$ and $-\la$ boundary conditions on $\R_+$.
By Proposition~\ref{prop:bounded_hausdorff} there exists $\tht_0$ such that if $\tht \in (0,\tht_0)$ is fixed, then with probability at least $1 - p/2N$, for each $x \in [0,R+1]$ there exists $y \in \eta_j^{\tht} \cap [0,R+1]$ such that $|x-y| \leq \delta_2/8$.
We assume these events also hold for each component $P_j$ (of which we have assumed there are at most $N$).
Set $x = \phi_j(\eta_0(t)) \in (0,\infty)$. We can then argue as in Step 1 to show that there exist $t_1, t_2, t'_1, t'_2$ as in the lemma statement (if $x > R$ we use that $\phi\nv([R,\infty)) \subseteq B(w_j, \dd_1)$ and that $\eta_\tht$ will intersect $\del^+ P_j$ arbitrarily close to the closing point a.s.). We can check that $t_1, t_2, t'_1, t'_2$ satisfy the properties of the lemma as before.
\end{proof}

\begin{proof}[Proof of Proposition~\ref{prop:intersections_stopped}]
Since $D$ is bounded, by the scale invariance of the GFF and by possibly replacing $\dd$ with a smaller value it suffices to prove the lemma in the case that $D, D_\dd$ are contained in $B(0, 1/2)$. Fix $\dd_1 = \dd/8$. By two applications of Lemma~\ref{lem:intersections_times} for this value of $\dd_1$, where in the second application we exchange the roles of $\eta_0$ and $\eta_\tht$, there must exist $\tht_0 > 0$ such that if $\tht \in (0, \tht_0)$ then with probability at least $1 - p$, the event described in the lemma statement holds (with $\dd_1$ in place of $\dd$), and similarly the same event holds with the roles of $\eta_0$ and $\eta_\tht$ exchanged. We assume now that these two events hold and show in this case that $\eta_0$ and $\eta_\tht$ are $\dd$-close until $\eta_0$ exits $D$.

Fix $t \in (0, \tau_0(D)]$. Note that $\tau_0(D_\dd) < \tau_0(\D)$ so by our application of Lemma~\ref{lem:intersections_times} there exist times $t_1, t_2, t'_1, t'_2$ such that $\eta_0, \eta_\tht$ intersect at these times and where $0 < t_1 < t < t_2 < \tau_0(\D_{\dd_1})$ and $0 < t'_1 < t'_2$, and for any $t_* \in [t_1, t_2]$ we have $|\eta_0(t) - \eta_0(t_*)| < \dd_1$. The final condition ensures that $t_2 < \tau_0(D_{\dd_1}) < \tau_0(D_\dd)$. We will show $t'_2 < \tau_\tht(D_\dd)$ and for any $s' \in [t'_1, t'_2]$ that $|\eta_\tht(t') - \eta_\tht(t'_1)| < \dd$.
In the following we will repeatedly use the fact that if $\eta_0(r_1) = \eta_\tht(r'_1), \eta_0(r_2) = \eta_\tht(r'_2)$ and $r_1 < r_2$, then $r'_1 < r'_2$.

We will first prove that $t'_2 < \tau_\tht(D_\dd)$. Suppose this is not the case. By the continuity of $\eta_\tht$ and since $D_\dd \subseteq B(0,1/2)$ there exists some $s' \in (0, t'_2]$ such that $\eta_\tht(s') \notin D_\dd$ and $s' < \tau_\tht(\D)$. Since the conclusions of Lemma~\ref{lem:intersections_times} hold with the roles of $\eta_\tht$ and $\eta_0$ exchanged, there must exist times $s'_1, s'_2, s_1, s_2$ with intersections at $\eta_\tht(s'_1), \eta_\tht(s'_2)$ such that $0 < s'_1 < s' < s'_2 < \tau_\tht(\D_{\dd_1})$ and $0 < s_1 < s_2$, and for all $s'_* \in [s'_1, s'_2]$ we have $|\eta_\tht(s') - \eta_\tht(t')| < \dd_1$. 
In particular, since $s'_1 < s' < t'_2$ we must have $s_1 < t_2 < \tau_0(D_{\dd_1})$. Therefore $\eta_\tht(s'_1) = \eta_0(s_1) \in D_{\dd_1}$ which implies that $\eta_\tht(s') \in D_{2\dd_1} \subseteq D_\dd$, a contradiction.

Now fix $s' \in [t'_1, t'_2]$. Again there must exist times $s'_1, s'_2, s_1, s_2$ with intersections at $\eta_\tht(s'_1), \eta_\tht(s'_2)$ such that $0 < s'_1 < s' < s'_2 < \tau_\tht(\D_{\dd_1})$ and $0 < s_1 < s_2$, and for all $s \in [s'_1, s'_2]$ we have $|\eta_\tht(s') - \eta_\tht(t')| < \dd_1$.
Then $s'_1 < s' \leq t'_2$ meaning that $s_1 < t_2$, and similarly $t'_1 \leq s' < s'_2$ meaning $t_1 < s_2$. There are now two possibilities. Firstly, we could have $t'_1 \leq s'_1 < t'$. In this case, we would have $t_1 \leq s_1 < t_2$ meaning that $|\eta_0(t) - \eta_\tht(s'_1)| < \dd_1$. Otherwise, $s'_1 < t'_1 < s'_2$ and we have $|\eta_0(t_1) - \eta_0(s'_1)| < \dd_1$. In either case we can apply the triangle inequality using the other relations we have to conclude that $|\eta_\tht(s') - \eta_0(t_*)| < \dd$ for all $t_* \in [t_1, t_2]$, from which we can show that $\eta_0$ and $\eta_\tht$ are $\dd$-close until $\eta_0$ leaves $D$.
\end{proof}

\section{The adjacency graph is connected}\label{sec:proofs}

\subsection{Overview}\label{subsec:overview}
In this section, we prove Theorem~\ref{thm:fan_connectivity}. Our goal is to show that any two connected components $U, V$ of $\h \sm \F(\tht'_1, \tht'_2)$ can be connected by a finite chain of components $U = U_0, U_1, \dots, U_m = V$ such that $\del U_j \cap \del U_{j+1} \neq \varnothing$ for each $j$. We will first give an informal overview which should motivate the more detailed arguments of Sections~\ref{subsec:annulus_events} and~\ref{subsec:connectivity_proof}. First note that the interval $[\tht'_1, \tht'_2]$ can be partitioned using $\tht'_2 = \tht_0 > \tht_1 > \dots > \tht_n = \tht'_1$ in such a way that $|\tht_{j+1} - \tht_j| \leq \eps_0$ for all $j$, for some small $\eps_0 > 0$ to be chosen later. We then define the \emph{finite fan}, $\F_f = \cup_{j=0}^n \eta_{\tht_n}$, which is depicted in Figure~\ref{fig:finite_fan_and_smaller_fan}.

The strategy of the proof of Theorem~\ref{thm:fan_connectivity} is to first show that such a finite chain of components exist in the case that $U$ and $V$ are in the same connected component $G$ of $\h \sm \F_f$. Suppose that $G$ lies between $\tht_j$ and $\tht_{j+1}$. Then the only flow lines in $\F$ that enter $G$ will have angle $\tht \in (\tht_{j+1}, \tht_j)$, so in particular can be represented as flow lines of angle $-\epsilon$ where $\epsilon \in (0, \eps_0)$ of a GFF with boundary conditions changed by a constant. The focus of Section~\ref{subsec:annulus_events} is to use the results of Section~\ref{sec:hausdorff_convergence} to understand the behavior of these flow lines. After we have completed the proof in the case that $U, V$ are in $G$, we then prove that the adjacency graph of the finite fan is itself connected, and finally that a finite chain of components as described above exists also in the case that $U$ and $V$ are in different connected components of $\h \sm \F_f$. The proof of Theorem~\ref{thm:fan_connectivity} using the results of Section~\ref{subsec:annulus_events} is carried out in Section~\ref{subsec:connectivity_proof}.

\subsection{Annulus events}\label{subsec:annulus_events}
Let $h$ be a zero boundary $\text{GFF}$ on $\D$ and fix $\kappa \in (0,4)$. The following setup is depicted in Figure~\ref{fig:most_annuli}. Fix annuli $A_3 \subsetneq A_4 \subseteq \D \setminus B(0,1/2)$ centered at $0$.
Fix three annuli $A_1 \subsetneq A_1^a \subsetneq A_1^b \subsetneq A_3$. Finally fix annuli $A_0 \subsetneq A_0^a$ lying between $\del^{\text{in}} A_3$ (the inner boundary of $A_3$) and $\del^{\text{in}} A_1^b$, and annuli $A_2 \subsetneq A_2^a$ lying between $\del^{\text{out}}A_1^b$ (the outer boundary of $A_1^b$) and $\del^{\text{out}}A_3$, as shown in Figure~\ref{fig:most_annuli}. All of the above annuli are centered at $0$.

\begin{figure}[t]
    \centering
    \includegraphics[scale=0.8]{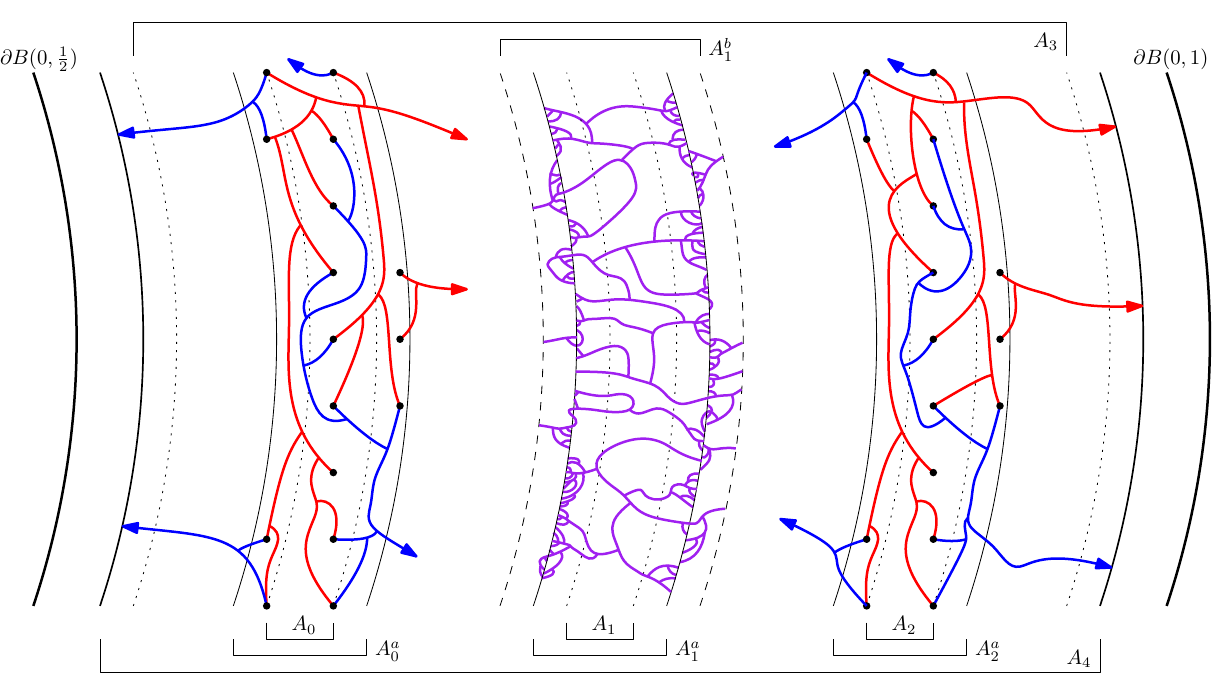}
    \caption{This figure summarizes the setup of Lemmas~\ref{lem:annulus_condition}, \ref{lem:chain_of_pockets} and~\ref{lem:middle_annulus}. The outermost annuli shown are $A_3 \subsetneq A_4 \subsetneq \D\sm B(0,1/2)$. The central annuli are $A_1, A_1^a$ and $A_1^b$. The components $V_j$ and path $\gg$ as described in (\ref{it :shield_of_flow_lines}) are contained in $A_1$. The existence of these components and path is discussed in Lemma~\ref{lem:middle_annulus}. In purple are the flow lines of angle $0$ started from the countable dense subset $(s_j)$ of $\del A_1^a$, which are stopped when they first exit $A_1^b$. The flow lines of angle $-\eps$ are omitted. Inside $A_0^a$ is the grid of points $a\Z^2 \cap A_0^a$. From each point we start flow lines of angle $0$ (red) and $-\pi$ (blue). In Lemma~\ref{lem:chain_of_pockets} we show they form a chain of pockets in $A_0$ which disconnect $A_0$ from $\del^{\text{in}} A_3$. The situation in $A_2^a$ is analogous to that in $A_1^a$. The flow lines $\eta^{\tht_j}_{a_k}$ started from $\del A_3$ are not shown.}
    \label{fig:most_annuli}
\end{figure}

For every $\theta \in \R,  x \in \D$,  we let $\eta_x^{\theta}$ be the flow line of $h$ starting from $x$ with angle $\theta$. Let $S = (s_j)$ (resp.\ $(a_j)$) be a fixed,  countable and dense subset of $\partial A_1^a$ (resp.\ $\del A_3$). 
Fix $\eps > 0$ and let $(\phi_j)$ be a fixed, countable and dense subset of $[-\eps, 0]$.
Let $W^{-\eps}$ be the closure of the union of the $\eta_{s_j}^0 \cup \eta_{s_j}^{-\epsilon}$'s when the flow lines are stopped at the first time that they exit $A_1^b$.
Let $E^h(\eps)$ be the event that the following hold.
\begin{enumerate}[(i)]
\item \label{it :shield_of_flow_lines}
There exist connected components $V_1, V_2, \dots, V_k$ of $A_1 \setminus W^{-\eps}$ and a point $z \in \partial V_1 \cap \partial V_k$ such that there exist simple paths $\gamma_1,\dots,\gamma_k$ parameterized by $[0,1]$ such that $\gamma_i((0,1)) \subseteq V_i$ for every $1 \leq i \leq k$,  $\gamma_{i-1}(1) = \gamma_i(0) \in \partial V_{i-1} \cap \partial V_i$ for every $2 \leq i \leq k$ and $\gamma_1(0) = \gamma_k(1) = z$.
Furthermore,  if $\gamma$ is the concatenation of the $\gamma_j$'s,  we have that $\gamma$ disconnects $\partial^{\text{in}}A_1$ from $\partial^{\text{out}}A_1$.

\item \label{it: flow_lines_do_not_enter} For every $\ell,m \in \N$,  we have that $\eta_{a_{\ell}}^{\phi_{m}}$ does not enter $\cup_{j=1}^k V_j$ before exiting $A_4$.
\end{enumerate}
Note that $E^h(\epsilon)$ is determined by $h|_{A_4}$. The main result of this section is the following lemma.

\begin{lemma}\label{lem:annulus_condition}

Fix $p \in (0,1)$.  Then,  there exists $\epsilon_0 \in (0,1)$ depending only on $\kappa$ and $p$ such that $\p[E^h(\epsilon)] \geq 1-p$ for every $\epsilon \in (0,\epsilon_0)$.
\end{lemma}

\begin{remark}\label{rem:annulus_condition_different scales}
For $z \in \C,  r > 0,  \epsilon \in (0,1)$ and a field $h$ on $B(z,r)$,  we let $E_{z,r}^h(\epsilon)$ be the event defined in the same way as $E^h(\epsilon)$ except that in conditions \eqref{it :shield_of_flow_lines}-\eqref{it: flow_lines_do_not_enter} we replace the annulus $A_i$ by $\phi_{z,r}(A_i)$ for every $1 \leq i \leq 4$,  and the sequences of points $(s_j)$ and $(a_j)$ by $(\phi_{z,r}(s_j))$ and $(\phi_{z,r}(a_j))$ respectively,  where $\phi_{z,r} : w \mapsto z + rw$.  Then,  the conformal invariance of the $\text{GFF}$ implies that $\p[E_{z,r}^h(\epsilon)]$ does not depend on $z$ or $r$ if $h$ has the law of a zero boundary $\text{GFF}$ on $B(z,r)$.
\end{remark}

In a moment we will give a summary of the proof of this lemma, which will make use of Lemmas~\ref{lem:chain_of_pockets} and~\ref{lem:middle_annulus}, but first let us outline how the event $E^h(\eps)$ will be used in the proof of Theorem~\ref{thm:fan_connectivity}.
Let $h_1$ be a GFF on $\h$, fix $z \in \h$ and suppose that $E^{h_1}_{z,r}(\eps)$ holds. That is, the event described in (\ref{it :shield_of_flow_lines}) and (\ref{it: flow_lines_do_not_enter}) above (suitably translated and rescaled) holds, but where now all flow lines we consider are flow lines of the field $h_1$ rather than being flow lines of a zero boundary GFF on $B(z,r)$. Suppose $\phi_j \in [-\eps, 0]$ is one of the countable dense set of angles described earlier. If $\eta_{\phi_j}$ is a flow line of $h_1$ started from $0$, upon entering $\phi_{z,r}(A_3)$ for the first time, it will a.s.\ merge with one of the flow lines of angle $\phi_j$ started from some point $\phi_{z,r}(a_k)$. Therefore, it cannot enter the components $V_j$ which form a chain around $z$ which will allow us to deduce information about the connectivity of $\h\sm\F$. This is only a vague motivation, and the details of exactly how Lemma~\ref{lem:annulus_condition} is used to prove the theorem are shown in Lemma~\ref{lem:annulus_condition_at_dense_scales} and Section~\ref{subsec:connectivity_proof}.

Next, we give an outline of the proof of Lemma~\ref{lem:annulus_condition}. In Lemma~\ref{lem:middle_annulus}, we will show that for a fixed $p \in (0,1)$, we can choose $\eps > 0$ small enough that \eqref{it :shield_of_flow_lines} holds with probability at least $1-p$. We want to then show, conditional on this event, that any flow line $\eta_{a_j}^{\phi_k}$ cannot enter any of the components $V_i$ defined in \eqref{it :shield_of_flow_lines} and thus can intersect the loop $\gg$ at at most finitely many points.
This is not necessarily true however without some additional information, which leads us to define the pairs of annuli $A_0, A_0^a$ and $A_2, A_2^a$. We fix a small parameter $a > 0$ and start flow lines of angle $0$ and $-\pi$ from each point $z \in a\Z^2 \cap (A_0^a \cup A_2^a)$. In Lemma~\ref{lem:chain_of_pockets} we show that when $a$ is small then with high probability these flow lines create a ``chain of pockets" which cover $A_0$ and are contained in $A_0^a$ (similarly for $A_2$ and $A_2^a$). This is explained in more detail below and depicted in Figure~\ref{fig:most_annuli}. If this event occurs, every flow line $\eta_{a_j}^{\phi_k}$ started from $\del^{\text{out}}A_3$ must pass through one of the pockets in $A_2^a$. Then, we show that it must be trapped between flow lines of angles $0$ and $-\eps$ started from a single point $z \in a\Z^2 \cap A_2^a$ until it exits $A_4$ (an analogous result holds for $\del^{\text{in}}A_3$ and $A_0^a$). We can use the results of Section~\ref{sec:hausdorff_convergence} to ensure these flow lines stay close together, which finally allows us to show that $\eta_{a_j}^{\phi_k}$ does not enter any of the $V_i$ defined in \eqref{it :shield_of_flow_lines}.

Next, we set up the event we want to consider in Lemma~\ref{lem:chain_of_pockets}. For every point $x$ in $(a \Z)^2 \cap A_0^a$, we say that the pocket with opening point $x$ exists if there exists a distinct point $z \in (a \Z)^2 \cap A_0^a$ such that $\eta_z^0$ merges with $\eta_x^0$ on the left side of $\eta_x^0$, and if $\eta_z^{-\pi}$ merges with $\eta_0^{-\pi}$ on the right side of $\eta_x^{-\pi}$, and if this occurs in both cases before any of these flow lines leave $A_0^a$.
We define $A_{x,z}$ to be the closure of the union of the connected components of $\C \setminus (\eta_x^0 \cup \eta_x^{-\pi})$ which are traced entirely by $\eta_x^0 \cup \eta_x^{-\pi}$ before $\eta_x^0$ (resp.\ $\eta_x^{-\pi}$) merges with $\eta_z^0$ (resp.\ $\eta_z^{-\pi}$) together with the connected components of $\C \setminus (\eta_z^0 \cup \eta_z^{-\pi})$ which are traced entirely by $\eta_z^0 \cup \eta_z^{-\pi}$ before $\eta_z^0$ (resp.\ $\eta_z^{-\pi}$) merges with $\eta_x^0$ (resp.\ $\eta_x^{-\pi}$),  and with the connected component of $\C \setminus (\eta_x^0 \cup \eta_z^0 \cup \eta_x^{-\pi} \cup \eta_z^{-\pi})$ formed at the time that the flow lines merge.
If the above event occurs for some point $z \in (a \Z)^2 \cap A_0^a$, then a.s.\ we can find such a point $z$ such that furthermore, $A_{x,z}$ contains no other points in $(a \Z)^2 \cap A_0^a$ in its interior (this follows from an inductive argument and the flow line interaction rules in \cite{ms2017ig4}).
If there exists such a point $z$, the pocket with opening point $x$ is $A_{x,z}$.  
The pocket with opening point $x$ in $(a\Z)^2 \cap A_2^a$ is defined analogously, if it exists.

Let $A_4^a$ be an annulus slightly bigger than $A_4$ so that $A_4 \subsetneq A_4^a \subsetneq \D$.
In this section, all parameters $\dd,\dd_j$ will be chosen to be much smaller than the distance between the boundaries of any two of the distinct annuli.

\begin{figure}[t]
    \centering
    \includegraphics[scale=0.8]{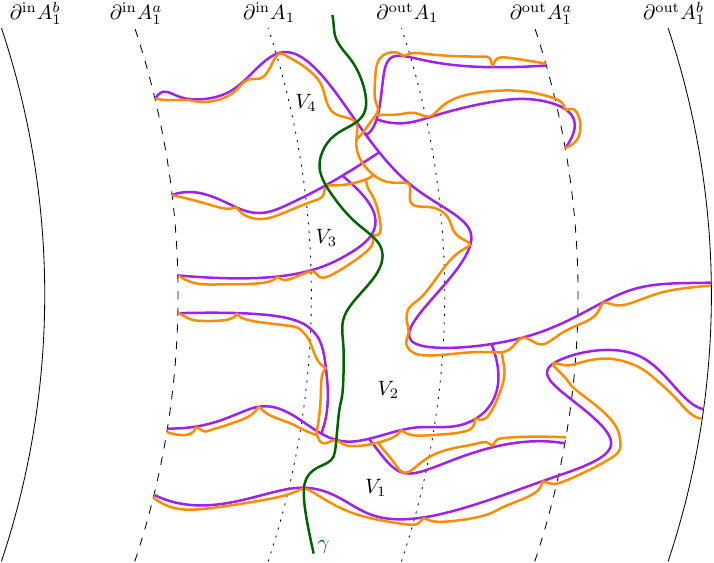}
    \caption{The components $V_i$ and the path $\gg$. Flow lines of angle $-\eps$ are shown in orange. We omit the additional annulus $A_1^c$.}
    \label{fig:annuli_pockets_v_components}
\end{figure}

\begin{lemma}\label{lem:chain_of_pockets}
For any $p,\dd \in (0,1)$ there exists $a > 0, \eps_0 \in (0,1)$ (depending only on $\kk, \dd$ and $p$) such that for every $\eps \in (0, \eps_0)$ the following events hold with probability at least $1 - p$. 
    \begin{enumerate}[(i)]
        \item There exist $x_1,\dots,x_n \in a \Z^2 \cap A_0^a$ and $y_1,\dots,y_m \in a \Z^2 \cap A_2^a$ such that $A_0 \subseteq \cup_{j=1}^n F_j \subseteq A_0^a$ where $F_j$ is the pocket with opening point $x_j$ for every $1 \leq j \leq n$, and $A_2 \subseteq \cup_{i=1}^m L_i \subseteq A_2^a$,  where $L_i$ is the pocket with opening point $y_i$ for every $1 \leq i \leq m$, and where all such pockets exist. 
        \item For each $z \in (a\Z^2 \cap A_0^a) \cup (a\Z^2 \cap A_2^a)$, the flow lines $\eta_z^0$ and $\eta_z^{-\eps}$ are $\dd$-close until $\eta_z^0$ exits $A_4^a$, and the flow lines $\eta_z^{-\pi}$ and $\eta_z^{-\pi-\eps}$ are $\dd$-close until $\eta_z^{-\pi}$ exits $A_4^a$.
    \end{enumerate}
\end{lemma}

\begin{proof}
To prove the this result, we will consider the analogous setup where $\wt{h}$ is instead a whole-plane GFF with values modulo a global multiple of $2\pi\chi$. In this case, two flow lines of the same angle started from two distinct points $x,z \in (a\Z)^2 \cap A_0^a$ merge a.s.
This allows us to to redefine the pocket $A_{x,z}$ as above, but we now remove the requirement that the flow lines merge before leaving $A_0^a$ (the reason we included this condition above was to avoid complications when these flow lines hit $\del D$, but such issues are avoided in this case).

Let $X_a$ be the union of all the flow lines $\eta_x^{0}, \eta_x^{-\pi}$ (of $\wt{h}$) for $x \in (a \Z)^2 \cap A_0^a$.  Then, every connected component of $A_4 \setminus X_a$ whose boundary is contained in $X_a$ has to be contained in a pocket with opening point $x$ for some $x \in (a \Z)^2 \cap A_0^a$.
By \cite[Proposition~4.14]{ms2017ig4},  we obtain that we can choose $a \in (0,1)$ sufficiently small such that with probability at least $1 - p/4$,  the following holds.  For every $x \in (a \Z)^2 \cap A_0^a$,  if the pocket with opening point $x$ intersects $A_0$,  then it is contained in $A_0^a$.

In this case,  there exist $x_1,\dots,x_n \in (a \Z)^2 \cap A_0^a$ such that $A_0 \subseteq \cup_{j=1}^n F_j \subseteq A_0^a$,  where $F_j$ is the pocket with opening point $x_j$ for every $1 \leq j \leq n$.  Similarly,  by possibly taking $a>0$ to be smaller and arguing as above,  we obtain that with probability at least $1 - p/4$,  there exist $y_1,\dots,y_m \in (a \Z)^2 \cap A_2^a$ such that $A_2 \subseteq \cup_{j=1}^m L_j \subseteq A_2^a$,  where $L_j$ is the pocket with opening point $y_j$ for every $1 \leq j \leq m$.  Therefore,  we obtain that we can choose $a \in (0,1)$ sufficiently small such that with probability at least $1 - p/2$,  we have that the first part of the statement holds.

For this $a$, there are a finite number of points in $a\Z^2 \cap A_0^a$ and $a\Z^2 \cap A_2^a$, so we can apply Proposition~\ref{prop:intersections_stopped} to conclude the proof of this result for $\wt{h}$.
To convert this result to the zero-boundary GFF $h$, we note that the event described in the lemma depends only on $h|_{A_0^a}$, so by the mutual absolute continuity of the laws of $h|_{A_0^a}$ and $\wt{h}|_{A_0^a}$ (see \cite[Lemma~4.1]{mq2020geodesics}), it follows that the statement holds for as written for $h$.
(Note that on the event that $F_j \subseteq A_0^a$, then the flow lines which determine it cannot leave $A_0^a$ without merging. This means that this is the pocket with opening point $x_j$ for both $h$ and $\wt{h}$, and the possible discrepancy in our definition of the pockets for $h$ and $\wt{h}$ does not cause any issues.)
\end{proof}

In the next lemma, we deal with the construction of the components $V_j$ and the path $\gg$ as described in (\ref{it :shield_of_flow_lines}).
For technical reasons we introduce a new annulus $A_1^c$ slightly larger than $A_1^b$ but which does not intersect $A_0^a$ or $A_2^a$. Our strategy to prove the lemma will be as follows.
Let $Z$ be the closure of the union of the $\eta_{s_j}^0$'s when they are stopped at the first time that they exit $A_1^c$ and let $Z^{-\eps}$ be the union of $Z$ and the closure of the union of the $\eta_{s_j}^{-\eps}$ when these flow lines are stopped at $A_1^b$. Notice that the $0$-angle and $-\eps$-angle flow lines are stopped at different points; this is merely to fix a technicality and is not important to the core idea of the argument. Note that $W^{-\eps} \subseteq Z^{-\eps}$, and that the flow lines of angle zero in $W^{-\eps}$ are stopped slightly earlier.  In the following and during the proof of the lemma, whenever we refer to a flow line of angle $0$ (resp.\ $-\eps$), it will be understood to be stopped upon hitting $A_1^c$ (resp.\ $A_1^b$) unless stated otherwise.

We will show that with high probability there exist connected components $U_1,\dots,U_k$ of $A_1 \setminus Z$ such that $k \geq 2,  U_1 = U_k$ and $\partial U_i \cap \partial U_{i+1} \neq \emptyset$ for every $1 \leq i \leq k-1$.
Furthermore, $\partial U_i \cap \del U_{i+1}$ contains a segment (greater than a single point) of some flow line $\eta^0_{s_j}$, and $\cup_{j = 1}^n \overline{U}_j$ disconnects $\partial^{\text{in}}A_1$ from $\partial^{\text{out}}A_1$.
We will then show that (with large probability) each $U_j$ contains a unique ``largest" (in a sense we will describe shortly) connected component $V_j$ of $A_1 \sm Z^{-\eps}$. We will show that these $V_j$ fulfill the requirements of (\ref{it :shield_of_flow_lines}) and furthermore that their diameters are large, which will be used to verify (\ref{it :shield_of_flow_lines}) when we complete the proof of Lemma~\ref{lem:annulus_condition}.

\begin{lemma}\label{lem:middle_annulus}
    
Fix $p \in (0,1)$.  There exist $\dd, \eps_0 > 0$ such that for any $\eps \in (0, \eps_0)$, with probability at least $1 - p$, there exist components $V_1, \dots, V_k$ and a path $\gg$ as in (\ref{it :shield_of_flow_lines}) and such that $\diam(V_j) \geq \dd$ for each $1 \leq j \leq k$.
\end{lemma}
\begin{proof}
\noindent{\it Step 1. Existence of the $U_j$.}
We set $A_1^T = A_1 \cap \h$ and $A_1^B = A_1 \cap (\C \setminus \overline{\h})$ and let $x$ (resp.\ $y$) be the midpoint of $\R_- \cap \partial A_1^T$ (resp.\ $\R_+ \cap \partial A_1^T$).  Note that a.s.\ there exists $N \in \N$ such that for every $k \in \N$,  the curve $\eta_{s_k}^0$ merges with $\eta_{s_j}^0$ before exiting $A_1^b$ for some $1 \leq j \leq N$.  It follows that $\{x,y\} \cap Z = \emptyset$ a.s.\  (since flow lines do not hit fixed points a.s.) and let $U^x$ (resp.\ $U^y$) be the connected component of $A_1 \setminus Z$ containing $x$ (resp.\ $y$).  Let also $U^{x,q}$ (resp.\ $U^{y,q}$) be the connected component of $A_1^q \setminus Z$ such that $x \in \partial U^{x,q}$ (resp.\ $y \in \partial U^{y,q}$) for $q \in \{T,B\}$.  Note that $U^{x,q} \subseteq U^x$ and $U^{y,q} \subseteq U^y$ for $q \in \{T,B\}$.  
Let $V^q$ be the union of the connected components of $A_1^q \setminus Z$ which can be reached using a finite chain of connected components of $A_1^q \setminus Z$ starting from $U^{x,q}$ where each boundary between consecutive components in the chain contains a segment of one of the flow lines $\eta^0_{s_j}$. 

We claim that $V^q = A_1^q \setminus Z$ a.s.  Indeed,  suppose that $V^q \neq A_1^q \setminus Z$.  Then,  there exists a connected component $U$ of $A_1^q \setminus Z$ such that $U \cap V^q = \emptyset$ which implies that there exists a point $w \in A_1^q \cap \partial{\overline{V}^q}$.
By combining \cite[Lemma~4.1]{mq2020geodesics} with \cite{dms2021mating},  we obtain that it is a.s.\  the case that finitely many flow lines can intersect at $w$ and there exist a leftmost and a rightmost flow line of the above flow lines.  Let $U_w$ (resp.\ $U_w'$) be the connected component of $A_1^q \setminus Z$ lying to the left (resp.\ right) of the leftmost (resp.\ rightmost) flow line intersecting $w$.  Fix $\zeta > 0$ and suppose that $\dist(w ,  \partial A_1^q) \geq 2 \zeta$.  Then,  \cite[Lemma~4.1]{mq2020geodesics} combined with \cite[Lemma~A.3]{kms2023nonsimpleremove} imply that a.s.\  there exists $N \in \N$ such that $(\cup_{j=1}^N \eta_{s_j}^0 ) \cap B(w,\zeta) = Z \cap B(w,\zeta)$.  Let $I_1,\dots,I_m$ be the arcs of $B(w,\zeta) \cap \eta_{s_j}^0$ for $1 \leq j \leq N$ which contain $w$.  We assume that the arcs are ordered from left to right.  Note that \cite[Lemma~4.1]{mq2020geodesics} combined with the tail decomposition of the flow lines of the whole-plane $\text{GFF}$ \cite[Proposition~3.6]{ms2017ig4} imply that it is a.s.\  the case that the following holds.  For every $j \in \N$,  every time that $\eta_{s_j}^0$ hits $w$,  it has to make a simple loop around $s_j$ and by doing so it exits $A_1$ and hence $A_1^q$.  It follows that $I_1,\dots,I_m$ are simple curves a.s.  
By decreasing $\zeta > 0$ if necessary,  we can assume that $Z \cap B(w,\zeta) = \cup_{j=1}^m I_j$.  Since the arcs $I_1,\dots,I_m$ are simple curves,  the connected component $\wt{U}_w$ of $B(w,\zeta) \setminus I_1$ which is to the left of $I_1$ has $w$ on its boundary and likewise the connected component $\wt{U}_w'$ of $B(w,\zeta) \setminus I_m$ which is to the right of $I_m$ also has $w$ on its boundary. This setup is depicted in Figure~\ref{fig:constructing_gamma_proof}.

Now,  every connected component of $B(w,\zeta) \setminus (\cup_{j=1}^m I_j)$ which is between $I_1$ and $I_2$ is adjacent to $\wt{U}_w$ and the intersection of the boundaries of each component with $\wt{U}_w$ all contain a segment of $I_1$.  
More generally,  every connected component of $B(w,\zeta) \setminus (\cup_{j=1}^m I_j)$ which is between $I_i$ and $I_{i+1}$ is adjacent to a connected component which is between $I_{i-1}$ and $I_i$ (with each pair of adjacent components again containing a segment of $I_i$ in the intersection of their boundaries).  So we obtain that all of the connected components of $B(w,\zeta) \setminus (\cup_{j=1}^m I_j)$ are connected to $\wt{U}_w$ in this way and likewise with $\wt{U}_w'$ in place of $\wt{U}_w$.  Thus, the graph of connected components of $B(w,\zeta) \setminus (\cup_{j=1}^m I_j)$ is connected and the closure of its union is equal to $B(w,\zeta)$.
Note that $\wt{U}_w \subseteq U_w$ (resp.\ $\wt{U}_w' \subseteq U_w'$) and $B(w,\zeta) \cap \partial \wt{U}_w = B(w,\zeta) \cap \partial U_w$ (resp.\ $B(w,\zeta) \cap \partial \wt{U}_w' = B(w,\zeta) \cap \partial U_w'$) for $\zeta > 0$ sufficiently small.  Note also that $U_w$ (resp.\ $U_w'$) can be connected to $U^{x,q}$ through a finite chain of connected components of $A_1^q \setminus Z$.  It follows that $B(w,\zeta) \subseteq \overline{V}^q$ but that is a contradiction since $w \in \partial \overline{V}^q$.  Therefore,  we obtain that $V^q = A_1^q \setminus Z$ a.s.\  for $q \in \{T,B\}$.  In particular,  for every $q \in \{T,B\}$,  there exists $n_q \in \N$ and connected components $U_1^q,\dots,U_{n_q}^q$ of $A_1^q \setminus Z$ such that $U^{x,q} = U_1^q,  U^{y,q} = U_{n_q}^q$ and $\partial U_j^q \cap \partial U_{j+1}^q$ contains a segment of some flow line $\eta_{s_k}^0$ for every $1 \leq j \leq n_q - 1$.

\begin{figure}[t]
    \centering
    \includegraphics[scale=0.8]{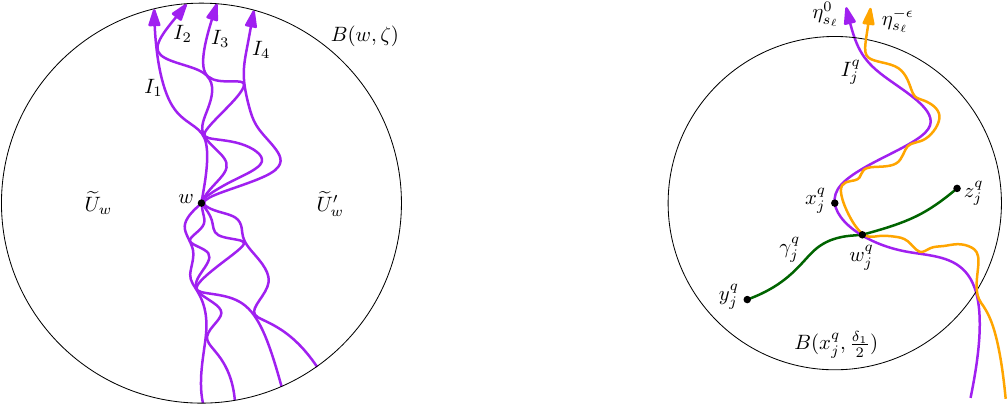}
    \caption{\textbf{Left:} A point $w \in A_1^q$ as described in Step 1 of the proof of Lemma~\ref{lem:middle_annulus}. $I_1, \dots, I_m$ are the segments of those flow lines which intersect $w$, ordered from left to right. Every connected component of $B(w, \zz) \sm (\cup_{j=1}^m I_j)$ can be connected to both $\wt{U}_w$ and $\wt{U}'_w$ via a finite chain of components. \textbf{Right:} Part of the flow-line segment $I_j^q$ which is contained in $\del U_j^q \cap \del U_{j+1}^q$. We assume without loss of generality that $U_j$ lies to the left of $I_j^q$ and that $U_{j+1}^q$ lies to the right. Here, $I_j^q$ is part of the flow line $\eta_{s_\ell}^0$. Not shown are the sets $W_j^q$ and $W_{j+1}^q$, but $y_j^q, z_j^q$ are chosen such that $y_j^q \in W_j^q$ and $z_j^q \in W_{j+1}^q$. We also show the simple curve $\gg_j^q$.}
    \label{fig:constructing_gamma_proof}
\end{figure}

\noindent{\it Step 2.  Existence of the $V_j$ and path $\gg$.}
We will first describe a collection of ``good" events that depend on the realization of the flow lines, each of which happen with high probability. Then, we will assume that these events all occur, and show in this case that the $V_j$ and path $\gg$ exist.

\noindent{\it Step 2.1. Setup.}
First, we fix $\epsilon \in (0,1)$ sufficiently small (to be chosen later).  Note that if $(\eta_{s_j}^{-\epsilon})$ were flow lines of angle $-\epsilon$ of the whole-plane $\text{GFF}$,  then their joint law would be independent of $\epsilon$.
Hence,  it follows from the proof of \cite[Lemma~A.3]{kms2023nonsimpleremove} and \cite[Lemma~4.1]{mq2020geodesics} that there exists $M \in \N$ independent of $\epsilon$ such that with probability at least $1 - p/10$,  we have that $A_1 \setminus \wt{Z}^{-\epsilon} = A_1 \setminus (\cup_{j=1}^M \eta_{s_j}^{-\epsilon})$,  where $\wt{Z}^{-\epsilon}$ is the closure of the union of the $\eta_{s_j}^{-\epsilon}$'s when the latter are stopped at the first time that they exit $A_1^b$. 
Applying the same argument also to the flow lines $(\eta_{s_j}^0)_{s_j \in S}$, we conclude that there exists a possibly larger value of $M$ such that with probability at least $1 - p/10$, we have $A_1 \setminus Z^{-\eps} = A_1 \setminus \cup_{j=1}^M (\eta_{s_j}^0 \cup \eta_{s_j}^{-\eps})$.

For each $q \in \{T, B\}$ and every $1 \leq j \leq n_q - 1$, $\del U_j^q \cap \del U_{j+1}^q$ contains a segment $I_j^q$ of some flow line $\eta_{s_k}^0$, and there must exist $x_j^q \in I_j^q$ (that can be chosen in a measurable way) and $\dd_j^q > 0$ such that $B(x_j^q, \dd_x^q) \cap Z \subseteq I_j^q$. In particular, there exists $\dd_1 > 0$ such that with probability at least $1 - p/10$ each $\dd_j^q$ can be chosen to be greater than $\dd_1$.
Furthermore, there exists $\dd_2 > 0$ such that with probability $1-p/10$ there exist $y_j^q \in B(x_j^q, \dd_1/2) \cap U_j^q$ and $z_j^q \in B(x_j^q, \dd_1/2) \cap U_{j+1}^q$ such that $\dist(y_j^q, Z), \dist(z_j^q, Z) > \dd_2$ for all $j, q$.
By possibly making $\dd_2$ smaller we can assume that with probability at least $1-p/10$ each $U_j^q$ contains a ball of radius $2\dd_2$. We assume from now on that all of these events also occur. The segment $I_j^q$ and points $x_j^q, y_j^q, z_j^q$ are shown in Figure~\ref{fig:constructing_gamma_proof}.

For each $q \in \{T,B\}$ and $1 \leq j \leq n_q$, we can choose a set $W_j^q \subsetneq U_j^q$ such that $\{z \in U_j^q \colon \dist(z,\del U_j^q) > \dd_2\} \subseteq W_j^q$ and $\dist(\del W_j^q, \del U_j^q) > 0$. One method of doing so in a measurable way is to pick a conformal transformation $\phi_{j,q} : U_j^q \rightarrow \D$ which is measurable with respect to $U_j^q$ and then let $W_j^q = \phi_{j,q}\nv({B}(0,r))$ where $r \in (0,1)$ is the infimum over all $s$ such that $\{z \in U_j^q \colon \dist(z,\del U_j^q) > \dd_2\} \subseteq \phi_{j,q}\nv(B(0,s))$.
We can now choose $\dd_3 > 0$ such that with probability at least $1 - p/10$ we have $\dist(W_j^q, \del U_j^q) > \dd_3$ for all $j,q$ (this is because for any given configuration of sets $U_j^q$, this will hold for some random $\dd'_3 > 0$, so we can choose a deterministic value $\dd_3$ such that $\dd'_3 > \dd_3$ with high probability).
Perhaps by making $\dd_2, \dd_3$ smaller, we can similarly assume that with probability $1-p/10$ that $\dist(x, Z) > \dd_2$ (recall that $x$ is the midpoint of $\R_- \cap \del A_1^T$), that there exist a set $W^x \subseteq U^x$ such that $\{z \in U^x \colon \dist(z, \del U_j^q) > \dd_2\} \subseteq W^x$, and that $\dist(W^x, \del U^x) > \dd_3$. We do the same for $y$ (the midpoint of $\R_+ \cap A_1^T$) and we assume further that these events hold.

Finally, let $\dd_4 < \min(\dd_1, \dd_2, \dd_3)/100$ and using Proposition~\ref{prop:intersections_stopped} and \cite[Lemma~4.1]{mq2020geodesics} choose $\eps_0 > 0$ small enough that for any fixed $\eps \in (0,\eps_0)$ the following holds with probability at least $1-p/10$.
For all $1 \leq j \leq M$, $\eta_z^{-\eps}$ and $\eta_z^0$ are $\dd_4$-close until $\eta_z^{-\eps}$ exits $A_1^b$. Note that at each intersection point, $\eta_z^{-\epsilon}$ hits $\eta_z^0$ on the \emph{right} side of $\eta_z^0$ since $-\eps < 0$. We emphasize that $\eta_z^{-\eps}$ plays the role of $\eta_0$, and is considered until it hits $A_1^b$, whereas $\eta_z^0$ may leave $A_1^b$, but not $(A_1^b)_\dd$ and hence will not leave $A_1^c$ either until after the $\dd_4$-close condition loses validity.

\noindent{\it Step 2.2. Conclusion.}
Let us now discuss what happens when these events all occur, which with the correct choice of parameters happens with probability at least $1-p$. Suppose that for $1 \leq \ell \leq M$ a flow line $\eta_{s_\ell}^{-\eps}$ enters $W_j^q$, so that there exists $w \in \eta_{s_\ell}^{-\eps} \cap W_j^q$. By our application of Proposition~\ref{prop:intersections_stopped} there must exist a point $z \in \eta_{s_\ell}^0$ (which is hit before $\eta_{s_\ell}^0$ exits $A_1^c$) such that $z \in B(w,\dd_4)$.
But then,  $\dist(z , \partial U_j^q) \geq \dist(W_j^q, \partial U_j^q) - \dist(w,z) \geq \dd_3 - \dd_4 > 0$ meaning that $z \in U_j^q$, a contradiction.
Therefore, none of the flow lines of angle $-\eps$ can enter $W_j^q$ so we can identify a unique connected component $V_j^q$ of $A_1^q \sm Z^{-\eps}$ such that $\{z \in U_j^q \colon \dist(z, \del U_j^q) \geq \dd_2\} \subseteq W_j^q \subseteq V_j^q \subseteq U_j^q$. Note that since $U_j^q$ contains a ball of radius $2\dd_2$, it follows that $\diam(V_j^q) \geq \dd_2$.

Fix $j,q$ such that $1 \leq j \leq n_q - 1$ and consider the segment $I_j^q$ of a flow line contained in $\del U_j^q \cap \del U_{j+1}^q$, and assume without loss of generality that $U_j^q$ (resp.\ $U_{j+1}^q$) is on the left (resp.\ right) side of $I_j^q$ when it is viewed as a flow line (the other scenario occurring will not change the argument). Let $x_j^q, y_j^q, z_j^q$ be points as described above and shown in Figure~\ref{fig:constructing_gamma_proof}. Note that any flow line $\eta_{s_\ell}^{-\eps}$ with $1 \leq \ell \leq M$ entering $B(x_j^q, \dd_1/2)$ must at all times stay within distance $\dd_4$ of $\eta_{s_\ell}^0$ by our previous application of Proposition~\ref{prop:intersections_stopped}.
Since $B(x_j^q, \dd_1) \cap Z \subseteq I_j^q$, the only possibility is that $\eta_{s_\ell}^0$ is a flow line of which $I_j^q$ is a segment. Furthermore, since $\eta_{s_\ell}^{-\eps}$ must intersect $\eta_{s_\ell}^0$ on the right side of the latter, $-\eps$-angle flow lines can only enter $B(x_j^q,\dd_1/2) \cap U_{j+1}^q$, not $B(x_j^q, \dd_1/2) \cap U_j^q$.
Note that it could be the case that $I_j^q$ is a segment of multiple flow lines $\eta_{s_k}^0$ for $1 \leq k \leq M$, where these flow lines have merged before they trace $I_j^q$, so in particular it is possible that multiple $-\eps$-angle flow lines enter $B(x_j^q, \dd_1/2)$.
However, by Proposition~\ref{prop:intersections_stopped}, each of these flow lines must intersect $I_j^q$ at many points in $B(x_j^q, \dd_1/2)$, and since they cannot cross each other, they are forced to merge. After this possible merging point, there again must be an intersection point where $\eta_{s_\ell}^{-\eps}$ intersects $I_j^q$; call this point $w_j^q$.
Since the flow lines, when restricted to $B(x_j^q, \dd_1/2)$, are all simple curves, we can draw a curve $\gg'$ from $w_j^q$ to $z_j^q$ which intersects these flow lines only at $w_j^q$. As we have previously observed that no flow lines of angle $-\eps$ enter $B(x_j^q,\dd_1/2) \cap U_j^q$, we can draw $\gg''$ connecting $y_j^q$ to $w_j^q$. By concatenating these two curves, we obtain a curve $\gg_j^q$ from $y_j^q$ to $z_j^q$, which intersects $Z^{-\eps}$ at exactly one point, $w_j^q$. The point $w_j^q$ and curve $\gg_j^q$ are shown in Figure~\ref{fig:constructing_gamma_proof}.
Finally, we notice that since $\dist(y_j^q, \del U_j^q), \dist(z_j^q, \del U_j^q) > \dd_3$ we must in fact have that $y_j^q \in V_j^q, z_j^q \in V_{j+1}^q$ from which we conclude that $\del V_j^q \cap \del V_{j+1}^q \neq \varnothing$, and that $\gg_j^q$ is a path starting in $V_j^q$, ending in $V_{j+1}^q$ and passing through the intersection of their boundaries exactly once.

It remains to go from the above components $V_j^q$ and paths $\gg_j^q$ to the components $V_j$ and path $\gg$ as described in (\ref{it :shield_of_flow_lines}), which is slightly technical but not difficult. First we deal with the situation in $U^x$, the connected component of $A_1 \sm Z$. We have assumed that $x, y_1^T, y_1^B$ are all distance at least $\dd_3$ from $\del U^x$, meaning we can argue as above to show that they lie in the same component $V^x$ of $A_1 \sm W^{-\eps}$ (recall that $W^{-\eps} \subseteq Z^{-\eps}$ and in $W^{-\eps}$ the flow lines of angle $0$ are stopped slightly earlier). It follows that we can draw paths $\gg'$ from $x$ to $y_1^T$ such that $\gg'((0,1])$ is in $V^x \cap A_1^T$ and $\gg''$ from $y_1^B$ to $x$ such that $\gg''([0,1))$ is in $V^x \cap A_1^B$. Note that $V^x$ must contain $V_1^T, V_1^B$ meaning in particular that $\diam(V_x) \geq \dd_2$. We perform an analogous procedure in $U^y$. We set $V_1 = V^x$ and $V_{n^T} = V^y$.

Consider now $V_j^T$ where $2 \leq j \leq n_T - 1$. We can choose (in an arbitrary measurable way) a simple curve $\wt{\gg}_j^T$ from $z_{j-1}^T$ to $y_j^T$ in $V_j^T$. let $V_j$ be the connected component of $A_1 \sm W^{-\eps}$ containing $V_j^T$ (this may be $V_j^T$ itself or some larger set).  By concatenating these curves with the curves $\gg_j^T$ we end up with a curve $\gg^T$ in $A_1^T$ which starts at $y_1^T$ and ends at $z_{n_T - 1}^T$. If we have $V_j = V_{j+k}$ for $k>0$, we can modify the path $\gg^T$ accordingly to skip those components $V_{j+1}, \dots, V_{j+k-1}$. We can also modify $\gg^T$ to be simple, if necessary. We perform a similar procedure in $A_1^B$, and for $n_T + 1 \leq j \leq k$ we let $V_j$ be the component of $A_1 \sm W^{-\eps}$ containing $V_{j - n_T + 1}^B$, where $k = n_T + n_B$.

In conclusion, we will have $V_1 = V_k$, $\del V_j \cap V_{j+1} \neq \varnothing$ for all $j$, and we can form the path $\gg$ by concatenating $\gg^T, \gg^B$ and the paths $\gg', \gg''$ as described above. As before, we can modify $\gg$ if necessary to ensure it is simple. Therefore, the $V_j$ and $\gg$ satisfy the conditions described in (\ref{it :shield_of_flow_lines}). Note finally that $\diam(V_j) \geq \dd_2$ for all $j$ by construction, so to conclude the proof we simply identify $\dd = \dd_2$.
\end{proof}

\begin{proof}[Proof of Lemma \ref{lem:annulus_condition}]
By Lemma~\ref{lem:middle_annulus}, there exist $\dd, \eps_0 > 0$ such that for $\eps \in (0,\eps_0)$, with probability at least $1-p/2$ there exist components $U_1, \dots, U_k$ and $V_1, \dots, V_k$ and a path $\gg$ satisfying the properties in \eqref{it :shield_of_flow_lines} with $\diam(V_j) \geq \dd$ for all $1 \leq j \leq k$. By Lemma~\ref{lem:chain_of_pockets}, there exist (a possibly smaller) $\eps_0 > 0$ and $a > 0$ such that with probability at least $1 - p/2$ the conclusions of this lemma hold with $\dd/100$ in place of $\dd$. Therefore, for a certain choice of the parameters $\eps_0, a$ and $\dd$, we have that the above events all hold with probability at least $1 - p$.
We will show that in the case that these events hold, any flow line started from $\del A_3$ with angle in $[-\epsilon,0]$ a.s.\ does not enter $\cup_{j=1}^k V_j$ when it is stopped at the first time that it exits $A_4$.  This will be shown into two main steps.  In the first step,  we will show that if $\eta_w^{\theta}$ enters either $\cup_{j=1}^n F_j$ or $\cup_{j=1}^m L_j$,  then it has to stay in the connected components lying between a flow line of angle $0$ and a flow line of angle $-\epsilon$ up until the first time that it exits $A_4$ after it has entered the above family of pockets.  Note that $\eta_w^{\theta}$ has to enter one of the above pockets in order to enter $\cup_{j=1}^k V_j$.  Hence,  in the second step,  we will show that the choice of  the $V_j's$ implies that $\eta_w^{\theta}$ cannot enter $\cup_{j=1}^k V_j$ and so this will complete the proof of \eqref{it: flow_lines_do_not_enter}.  

\begin{figure}[t]
    \centering
    \includegraphics[scale=0.8]{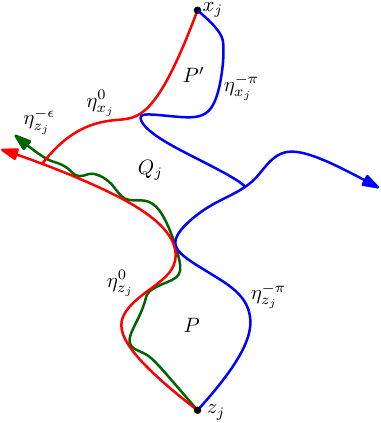}
    \caption{$L_j$ is the pocket with opening point $x_j$ and closing point $z_j$, where $x_j$ comes before $z_j$ in our ordering. If $\eta^\tht$ enters $L_j$, it can only do so by crossing $\eta_{z_j}^0$ to enter $Q$ or $P$, or by crossing $\eta_{z_j}^{-\pi}$ to enter $P$. It cannot enter a pocket of the form $P'$. In any case, $\eta^\tht$ is trapped between $\eta_{z_j}^0$ and $\eta_{z_j}^{-\eps}$.}
    \label{fig:entering_pockets}
\end{figure}

\noindent{\it Step 1.  $\eta_w^{\theta}$ gets trapped between a flow line with angle $0$ and a flow line with angle $-\epsilon$.}
We can assume that $w \in \partial^{\text{out}} A_3$ since a similar argument works when $w \in \partial^{\text{in}} A_3$.  Suppose that $\eta_w^{\theta}$ intersects $\cup_{j=1}^k V_j$ before exiting $A_4$.  Then,  there exists $1 \leq j \leq m$ such that $\eta_w^{\theta}$ enters the interior of $L_j$.  Let $x_j$ (resp.\ $z_j$) be the opening (resp.\ closing) point of $L_j$.  Let also $Q_j$ be the connected component formed when $\eta_{x_j}^0$ (resp.\ $\eta_{x_j}^{-\pi}$) merges with $\eta_{z_j}^0$ (resp.\ $\eta_{_j}^{-\pi}$). 
The flow line interaction rules \cite[Theorem~1.7]{ms2017ig4} imply that $\eta_w^{\theta}$ cannot enter the interior of $Q_j$ by crossing the arc of $\partial Q_j$ corresponding to $\eta_{x_j}^0$ or the arc corresponding to $\eta_{x_j}^{-\pi}$.  Also,  in order to enter the interior of $Q_j$ by crossing the arc corresponding to $\eta_{z_j}^{-\pi}$,  it has to cross $\eta_{z_j}^{-\pi-\epsilon}$ from right to left but this does not occur due to the flow line interaction rules.  Hence,  it can only enter the interior of $Q_j$ by crossing from left to right the arc corresponding to $\eta_{z_j}^0$.
By taking $\dd > 0$ sufficiently small,  we can assume that there are intersection points between $\eta_{z_j}^0$ and $\eta_{z_j}^{-\epsilon}$ in the interior of $Q_j$.  Thus,  the flow line interaction rules imply that after $\eta_w^{\theta}$ enters the interior of $Q_j$,  it has to lie in the connected components whose boundaries consist of an arc of the right side of $\eta_{z_j}^0$ and the left side of $\eta_{z_j}^{-\epsilon}$ up until exiting $A_4$.  Suppose that $\eta_w^{\theta}$ enters the interior of $L_j$ by entering a connected component $P$ lying between $\eta_{x_j}^0$ and $\eta_{x_j}^{-\pi}$ or between $\eta_{z_j}^0$ and $\eta_{z_j}^{-\pi}$,  and which is traced entirely before $\eta_{x_j}^0$ merges with $\eta_{z_j}^0$ and $\eta_{x_j}^{-\pi}$ merges with $\eta_{z_j}^{-\pi}$.
Without loss of generality,  we can assume that $P$ lies between $\eta_{z_j}^0$ and $\eta_{z_j}^{-\pi}$.  Then,  $P$ has two marked points and $\partial P$ consists of either one arc of the right side of $\eta_{z_j}^0$ and one arc of the left side of $\eta_{z_j}^{-\pi}$,  or one arc of the right side of $\eta_{z_j}^{-\pi}$ and one arc of the left side of $\eta_{z_j}^0$.  We note that the flow line interaction rules imply that the second case cannot hold and so necessarily we are in the first case. 
If $\eta_w^{\theta}$ enters $P$ by crossing the arc contained in $\eta_{z_j}^0$,  then the flow line interaction rules imply that after entering $P$,  the curve $\eta_w^{\theta}$ has to lie in the connected components whose boundaries consist of an arc of the right side of $\eta_{z_j}^0$ and an arc of the left side of $\eta_{z_j}^{-\epsilon}$ up until it exits $A_4$.  If $\eta_w^{\theta}$ enters $P$ by crossing the arc contained in $\eta_{z_j}^{-\pi}$,  then it has to exit $P$ by crossing $\eta_{z_j}^{-\epsilon}$ since it cannot cross $\eta_{z_j}^{-\pi}$ again.  Similarly,  it cannot cross $\eta_{z_j}^{-\epsilon}$ for a second time and so combining with the flow line interaction rules,  we obtain that after it crosses $\eta_{z_j}^{-\epsilon}$,  the curve $\eta_w^{\theta}$ has to lie in the connected components whose boundaries consist of an arc of the right side of $\eta_{z_j}^0$ and an arc of the left side of $\eta_{z_j}^{-\epsilon}$ up until it exits $A_4$.
In either case, note that by choosing $\dd$ small enough, we can ensure that $\eta_{z_j}^0, \eta_{z_j}^{-\epsilon}$ do not leave $A_4^a$ before this point.

\noindent{\it Step 2.  Conclusion of the proof of \eqref{it: flow_lines_do_not_enter}.} Suppose that there exists $1 \leq j \leq k$ and $x \in \eta_w^{\theta} \cap V_j$ such that $\eta_w^{\theta}$ hits $x$ before it exits $A_4$.  Then,  there exists $y \in (a \Z)^2 \cap A_2^a$ such that $x \in \overline{V}$,  where $V$ is a connected component with two marked points whose boundary consists of an arc of the right side of $\eta_y^0$ and an arc of the left side of $\eta_y^{-\epsilon}$. 
Since we have assumed the conclusion of Lemma~\ref{lem:chain_of_pockets} with parameter $\dd/100$ and since $\dd \ll \dist(\del A_4, \del A_4^a)$, we have that neither $\eta_y^0$ nor $\eta_y^{-\epsilon}$ exit $A_4^a$ before finishing tracing $\partial V$ and $\diam(V) \leq \dd/50$.
Also,  a.s  there exist $j_1 ,  j_2 \in \N$ such that $\eta_{s_{j_1}}^0$ (resp.\ $\eta_{s_{j_2}}^{-\epsilon}$) merges with $\eta_y^0$ (resp.\ $\eta_y^{-\epsilon}$) before entering $A_1'$,  where $A_1'$ is a fixed annulus centered at $0$ such that $A_1 \subsetneq A_1' \subsetneq A_1^a$. 
Then,  if $\dd > 0$ is sufficiently small,  we have that $V$ is a connected component of the complement in $A_1'$ of the union of $\eta_{s_{j_1}}^0$ and $\eta_{s_{j_2}}^{-\epsilon}$ when they are both stopped at the first time that they exit $A_1^b$.  Note that there exists a connected component $\wh{V}$ of the above set such that $V_j \subseteq \wh{V}$.  Then,  $x \in \wh{V}$ and so we must have that $V_j \subseteq \wh{V} = V$ which implies that $\diam(V_j) \leq \diam(V)$.  
But $\diam(V_j) \geq \dd$ by the construction of the $V_j$ and $\diam(V) < \dd/50$ so we obtain a contradiction.  It follows that $\eta_w^{\theta}$ does not enter $\cup_{j=1}^k V_j$ before exiting $A_4$.  By replacing the pockets $L_j$'s with the pockets $F_j$'s and using a similar argument,  we obtain that $\eta_w^{\theta}$ does not enter $\cup_{j=1}^k V_j$ before exiting $A_4$ when $\theta \in [-\epsilon,0]$ and $w \in \partial^{\text{in}}A_3$.  This completes the proof of the lemma. 
\end{proof}

\begin{lemma}\label{lem:annulus_condition_at_dense_scales}

Fix $\kappa \in (0,4),  b \in (0,1), c,M>0$.  Then,  there exist $a , \epsilon_0 \in (0,1)$ depending only on $\kappa ,  b ,  c$ and $M$ such that for every $\epsilon \in (0,\epsilon_0)$ the following is true.  Let $h^0$ be a zero-boundary $\text{GFF}$ on $\h$ and let $f$ be a deterministic harmonic function on $\h$ whose boundary values are piecewise constant,  they change only finitely many times and $||f||_{\infty} \leq M$.  Set $h = h^0 + f$.  Then,  a.s.\  for every compact set $K \subseteq \h$,  there exists $n_0 \in \N$ such that for every $n \geq n_0,  z \in (e^{-c n} \Z)^2 \cap K$,  we have that $E_{z,2^{-m}}^h(\epsilon)$ occurs for some $bn \leq m \leq n$.
\end{lemma}
\begin{proof}

Let $K \subseteq \h$ be a fixed compact set and set $d = \dist(K ,  \partial \h) > 0$.  Let $m_0 \in \N$ be such that $r = 2^{-m_0} \in (0,d/2)$.  For every $z \in K$ we let $\wt{h}_{z,r}$ be a zero-boundary $\text{GFF}$ on $B(z,r)$ and let $\wt{a} \in (\frac{1}{2} , 1)$ be such that $A_4 \subseteq B(0,\wt{a})$.  Then,  arguing as in the proof of \cite[Lemma~4.1]{mq2020geodesics},  we obtain that if $\CZ_{z,r}$ is the Radon-Nikodym derivative of the law of $h|_{B(z,\wt{a}r)}$ with respect to the law of $\wt{h}_{z,r}|_{B(z,\wt{a}r)}$,  then there exists $p \in (1,\infty)$ depending only on $r$,  $K$ and $M$ such that $\E[\CZ_{z,r}^p] \lesssim 1$,  where the implicit constant depends only on $r$,  $K$ and $M$.

Next we fix $\wt{c} > 0$ and for $N \in \N$ we let $N_{z,r}^h(N,\epsilon)$ (resp.\ $N_{z,r}^{\wt{h}_{z,r}}(N,\epsilon)$) be the number of $1 \leq k \leq N$ for which $E_{z,r_k}^h(\epsilon)$ (resp.\ $E_{z,r_k}^{\wt{h}_{z,r}}(\epsilon)$) occurs where $r_k = r 2^{-k}$ for every $k \in \N$.  It follows by combining Lemma~\ref{lem:annulus_condition} and Remark~\ref{rem:annulus_condition_different scales} with the proof of \cite[Proposition~4.3]{mq2020geodesics} that there exist constant $c_0 > 0,  \epsilon_0 \in (0,1)$ depending only on $\wt{c}$ and $b$ such that $\p[N_{z,r}^{\wt{h}_{z,r}}(N,\epsilon) \leq b N] \leq c_0 e^{-\wt{c}N}$ for every $z \in K,  \epsilon \in (0,\epsilon_0)$ and $N \in \N$.   Therefore,  we obtain that
\begin{align*}
\p[N_{z,r}^h(N,\epsilon) \leq b N]  \leq \E[\CZ_{z,r}^p]^{1/p} \p[N_{z,r}^{\wt{h}_{z,r}}(N,\epsilon) \leq bN]^{1/q} \lesssim e^{-\wt{c}N/q}
\end{align*}
for every $N \in \N,  \epsilon \in (0,\epsilon_0)$ and $z \in K$,  where the implicit constant depends only on $r,K,M,\wt{c}$ and $b$,  and $q$ is such that $q >1$ and $1/p + 1/q = 1$.  Hence,  by taking $\wt{c}$ sufficiently large,  we obtain that for every $\epsilon \in (0,\epsilon_0)$,  we have a.s.\  that there exists $n_0 \in \N$ such that $N_{z,r}^h(n,\epsilon) > bn$ for every $n \geq n_0$ and $z \in (e^{-c n} \Z)^2 \cap K$.  Then,  the claim of the statement of the lemma follows since if $N_{z,r}^h(n,\epsilon) > bn$,  then there exists $bn \leq m \leq n$ such that $E_{z,2^{-m}}^h(\epsilon)$ occurs.
\end{proof}

\subsection{Completing the proof of Theorem~\ref{thm:fan_connectivity}}\label{subsec:connectivity_proof}

\begin{proof}[Proof of Theorem~\ref{thm:fan_connectivity}]
\noindent{\it Step 1.  Overview and setup.} 
Let $h$ be a GFF on $\h$ with boundary values given by $-a$ (resp.\ $b$) on $\R_-$ (resp.\ $\R_+$) and let $\tht'_1 < \tht'_2$ satisfy \eqref{eqn:fan_gen_bd}. We will prove that the adjacency graph of $\h \setminus \fan(\theta'_1,\theta'_2)$ is connected.
We pick $\epsilon \in (0,1)$ sufficiently small (to be chosen and depending only on $a,b, \kk, \tht'_1$ and $\tht'_2$) and such that $n = 4(\tht'_2 - \tht'_1)/{\epsilon} \in \N$.  For $j \in \{0,\dots,n\}$,  we set 
\[\theta_j = \tht'_2 - j \frac{\epsilon}{4}\]
and $h_j = h + \theta_j\chi$,  and note that the boundary conditions of $h_j$ are given by $-a+\theta_j\chi$ on $\R_-$ and $b+\theta_j\chi$ on $\R_+$.  Note also that the boundary conditions of $h_j$ lie in $[-M,M]$ where $M = \max\{|a|,|b|\} + \chi\max\{\tht'_1, \tht'_2\}$ depends only on $a,b, \kk, \tht_1$ and $\tht_2$.  
In the case that we have equality in (\ref{eqn:fan_gen_bd}), if $\tht_1' = -(\la + b)/\chi$ (resp.\ $\tht_2' = (\la + a)/\chi$) then we identify the flow line $\eta_{\tht'_1} \equiv \eta_{\tht_n}$ (resp.\ $\eta_{\tht'_2} \equiv \eta_{\tht_0}$), which is not defined a priori, with $\R_+$ (resp.\ $\R_-$).

Our goal is to show that the graph of connected components of $\h \setminus \F(\tht_1', \tht_2')$ is connected a.s.  This will be achieved through several steps.
In Step 2,  we use the events introduced in Remark~\ref{rem:annulus_condition_different scales} and use Lemma~\ref{lem:annulus_condition_at_dense_scales} to deduce that at a sufficiently dense set of scales we have that we can find annuli such that in every such annulus,  there exists a finite chain of connected components of $\h \setminus \F(\tht_1', \tht_2')$ disconnecting the inner from the outer boundary of the annulus.  In Step 3,  we use the above property to deduce that a.s.\ the following is true.
For every $j \in \{0,\dots,n-1\}$,   any two distinct connected components of $\h \setminus \fan(\theta_{j+1},\theta_j)$ lying between $\eta_{\theta_j}$ and $\eta_{\theta_{j+1}}$ and in the same connected component $G$ of $\h \setminus (\eta_{\theta_j} \cup \eta_{\theta_{j+1}})$,  can be connected via a finite chain of connected components of $\h \setminus \fan(\theta_{j+1},\theta_j)$ which are all contained in $G$. 
Next,  in Step 4,  we show that the graph of connected components of $\h \setminus \cup_{j=0}^n \eta_{\theta_j}$ is connected a.s.  Finally,  combining Steps 3 and 4,  we show in Step 5 that any two distinct connected components of $\h \setminus \F(\tht_1', \tht_2')$ lying in different connected components of $\h \setminus \cup_{j=0}^n \eta_{\theta_j}$ can be connected via a finite chain of connected components of $\h \setminus \F(\tht_1', \tht_2')$,  and thus completing the proof of the lemma.

\begin{figure}[t]
    \centering
    \includegraphics[scale=0.8]{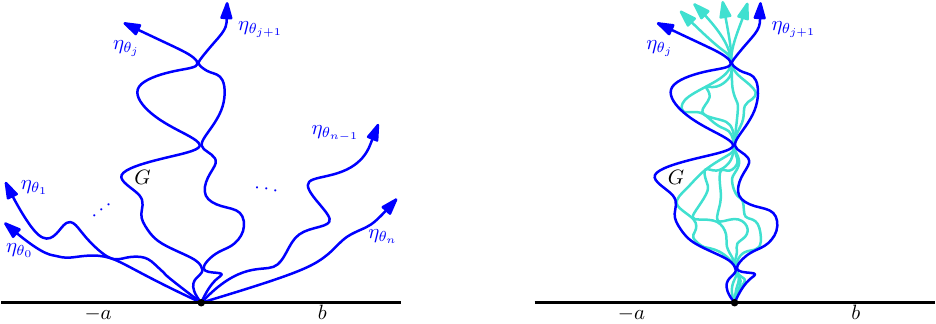}
    \caption{\textbf{Left:} The finite collection of flow lines $\eta_{\theta_j}$. \textbf{Right:} The fan $F(\tht_{j+1}, \tht_j)$.}
    \label{fig:finite_fan_and_smaller_fan}
\end{figure}

\noindent{\it Step 2.  Annulus events.} 
Suppose that we have the setup of Lemma~\ref{lem:annulus_condition} and Remark~\ref{rem:annulus_condition_different scales}.  Since $M$ depends only on $a,b$ and $\kappa$,  combining with Lemma~\ref{lem:annulus_condition_at_dense_scales},  we obtain that if we fix $\wt{a} \in (0,\infty) ,  \wt{b} \in (0,1)$ and choose $\epsilon \in (0,1)$ sufficiently small (depending only on $a,b,\wt{a},\wt{b}$ and $\kappa$),  then for every $1 \leq j \leq n$,  a.s.\ the following holds.
For every compact set $K \subseteq \h$ there exists $n_0 \in \N$ such that for every $n \geq n_0$ and every $z \in (e^{-\wt{a}n}\Z^2) \cap K$,  there exists $m \geq n_0$ with $(1-\wt{b}) n \leq m \leq n$ such that $E_{z,2^{-m}}^{h_j}(\epsilon)$ occurs.  Let $E$ be the event that the above event holds for every $1 \leq j \leq n$ with the choice of $\wt a = 1, \wt b = 1/2$ and from now on,  we assume that $E$ occurs.

\noindent{\it Step 3.  The graph of connected components of $\h \setminus \fan(\theta_{j+1},\theta_j)$ restricted to any connected component $G$ of $\h \setminus (\eta_{\theta_j} \cup \eta_{\theta_{j+1}})$ is connected for every $0 \leq j \leq n-1$.} Fix $0 \leq j \leq n-1$.  We will show that the following is true a.s.  Let $U_1, U_2$ be two distinct connected components of $\h \setminus \fan(\theta_{j+1},\theta_j)$ lying between $\eta_{\theta_j}$ and $\eta_{\theta_{j+1}}$,  and in the same connected component $G$ of $\h \setminus (\eta_{\theta_j} \cup \eta_{\theta_{j+1}})$.  Then,  there exists a finite chain of connected components of $\h \setminus \fan(\theta_{j+1},\theta_j)$ connecting $U_1$ to $U_2$ and contained in $G$. 

Note that by \cite{ms2016imag1}, $\F$ does not hit fixed points a.s. (see also Lemma~\ref{lem:fan_does_not_contain_point}). Now,  to prove the claim,  we fix $z \in \h \cap \Q^2$ and let $U$ be the connected component of $\h \setminus \fan(\theta_{j+1},\theta_j)$ containing $z$.  Suppose that we are working on the event that $U$ is contained in a connected component $G$ of $\h \setminus (\eta_{\theta_{j+1}} \cup \eta_{\theta_j})$ which lies between $\eta_{\theta_{j+1}}$ and $\eta_{\theta_j}$.
Let $K$ be the closure of the union of connected components of $\h \setminus \fan(\theta_{j+1},\theta_j)$ which are contained in $G$ and can be connected to $U$ via a finite chain of connected components of $\h \setminus \fan(\theta_{j+1},\theta_j)$ contained in $G$.  Suppose that $G \nsubseteq K$ and let $V$ be a connected component of $G \setminus K$.
Note that $\partial V \setminus \partial G \neq \emptyset$ since otherwise we would have that $G \subseteq \overline{V}$ but that is a contradiction since $z \notin \overline{V}$.  Thus,  we fix $x \in \partial V \cap G$.  

Since $E$ occurs,  there exist $y \in \h,  r>0$ such that $x \in B(y,r/2),  \overline{B(y,r)} \subseteq G$, $V \setminus B(y, r) \neq \varnothing$ and $E_{y,r}^{h_j}(\epsilon)$ occurs.  Let $V_1,\dots,V_m$ be the components as in \eqref{it :shield_of_flow_lines} in the definition of $E_{y,r}^{h_j}(\epsilon)$ and let $\gamma$ be the corresponding path.  We claim that $\fan(\theta_{j+1},\theta_j)$ does not enter any of the $V_i$'s. 
Indeed,  first we note that a flow line of $h$ with angle in $[\theta_{j+1},\theta_j]$ corresponds to a flow line of $h_j$ with angle in $[-\epsilon/4,0]$.  Fix $\theta \in [\theta_{j+1} ,  \theta_j]$ and suppose that $\eta_{\theta}$ enters $\cup_{i=1}^m V_i$ and let $t$ be such that $\eta_{\theta}(t) \in \cup_{i=1}^m V_i$.  Set $\sigma = \sup\{s \leq t : \eta_{\theta}(s) \in \partial \phi_{z,r}(A_3)\}$.
By the continuity of $\eta_{\theta}$,  we have that $\sigma < \infty$ a.s.  Fix $\mu > 0$ deterministic and small enough that $\dist(\del A_3, \del A_4) > 100\mu$ which ensures that $B(\eta_{\theta}(\sigma) ,  100 \mu r) \cap \partial \phi_{z,r}(A_4) = \emptyset$.  Note that $\eta_{\theta}$ corresponds to a flow line of $h_j$ with angle $\theta - \theta_j \in [-\epsilon/2 ,  0] \subseteq [-\epsilon ,  0]$.
Then,  a.s.\  there exists $j_1 \in \N$ such that $\eta_{\phi_{z,r}(a_{j_1})}^{\theta}$ merges with $\eta_{\theta}$ before it exits $B(\eta_{\theta}(\sigma) ,  \mu r)$ and such that $\phi_{z,r}(a_{j_1}) \in B(\eta_{\theta}(\sigma) ,  \mu r /100)$,  where $\eta_{\phi_{z,r}(a_{j_1})}^{\theta}$ is the flow line of $h$ with angle $\theta$ and starting from $\phi_{z,r}(a_{j_1})$.
Hence,  $\eta_{\theta}(t)$ lies in the range of $\eta_{\phi_{z,r}(a_{j_1})}^{\theta}$ when the latter is stopped at the first time that it exits $\phi_{z,r}(A_4)$.  This implies that the latter set intersects $\cup_{i=1}^m V_i$ but that is a contradiction due to \eqref{it: flow_lines_do_not_enter}.

Therefore,  $\eta_{\theta}$ does not enter $\cup_{i=1}^m V_i$.  Since $\theta \in [\theta_{j+1} ,  \theta_j]$ was arbitrary,  we obtain that $\fan(\theta_{j+1} ,  \theta_j)$ does not enter $\cup_{i=1}^m V_i$ a.s. Therefore,  for every $1 \leq i \leq m$,  there exists a unique connected component $W_i$ of $\h \setminus \fan(\theta_{j+1},\theta_j)$ such that $V_i \subseteq W_i$.  Since $W_i \cap G \neq \emptyset$,  we have that $W_i \subseteq G$ for every $1 \leq i \leq m$.  Also, since $x \in \del V$,  there exists a sequence $(G_n)$ of connected components of $\h \setminus \fan(\theta_{j+1},\theta_j)$ such that $G_n \subseteq G$ and $U$ is connected to $G_n$ via a finite chain of connected components of $\h \setminus \fan(\theta_{j+1},\theta_j)$ contained in $G$ for every $n \in \N$,  and $\dist(x,G_n) \to 0$ as $n \to \infty$.  
Thus,  $W_i$ is connected to $U$ via such a chain for every $1 \leq i \leq m$ and so $\gamma \subseteq K$.  Fix $w \in V \setminus B(y,r)$ and $\wt{w} \in V \cap B(y,r/2)$.  Then there exists a continuous path $\wt{\gamma}$ in $V$ connecting $w$ to $\wt{w}$ and so $\wt{\gamma} \cap \gamma \neq \emptyset$ which implies that $K \cap V \neq \emptyset$.  This contradicts our initial assumption and so $G \subseteq K$. Let $\wt{U}$ be another connected component of $\h \setminus \fan(\theta_{j+1},\theta_j)$ contained in $G$ and fix $w \in \wt{U}$.  Then,  $w \in K$ and so we must have that $\wt{U}$ is connected to $U$ via such a finite chain.  The claim then follows since $z \in \h \cap \Q^2$ was arbitrary.

\begin{figure}[t]
    \centering
    \includegraphics[scale=0.8]{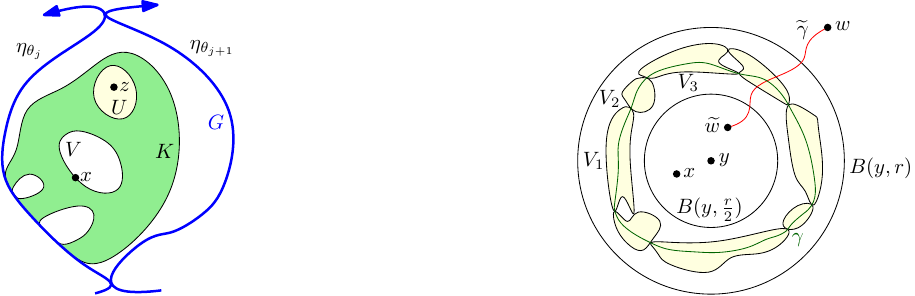}
    \caption{\textbf{Left:} In the setting of Step 3, $G$ is a connected component of $\h \setminus \cup_{j=0}^n \eta_{\theta_j}$ lying between $\eta_{\theta_j}$ and $\eta_{\theta_{j+1}}$ and $U$ is the connected component of $\h \setminus \F$ containing the reference point $z \in \h \cap \Q^2$. $K$ is the closure of the union of connected components of $\h \setminus \fan(\theta_{j+1},\theta_j)$ which are contained in $G$ and can be connected to $U$ via a finite chain of connected components of $\h \setminus \fan(\theta_{j+1},\theta_j)$ contained in $G$. $V$ is a connected component of $G\setminus K$ and $x \in \del V$. \textbf{Right:} $V_1,\dots,V_m$ are the components as in \eqref{it :shield_of_flow_lines} in the definition of $E_{y,r}^{h_j}(\epsilon)$. Since $x \in \del V$, each $V_i$ is contained in a connected component $W_i$ of $\h \setminus F(\tht_{j+1}, \tht_j)$ which is in $K$ ($V$ and the components $G_n$ are not shown). Since $\gg \subseteq K$ and $\wt\gg \subseteq V$, and $\gg$ and $\wt\gg$ must intersect, $V \cap K \neq \emptyset$, which is a contradiction.}
    \label{fig:connectivity_V}
\end{figure}

\noindent{\it Step 4.  The graph of connected components of $\h \setminus \cup_{j=0}^n \eta_{\theta_j}$ is connected.} We will show that the following is true a.s.  Let $U,V$ be two distinct connected components of $\h \setminus \cup_{j=0}^n \eta_{\theta_j}$.
Then there exist distinct connected components of $\h \setminus \cup_{j=0}^n \eta_{\theta_j}$,  $U_1,\dots,U_m$ such that $U = U_1,  V = U_m$ and for every $1 \leq j \leq m-1$,  the following holds. 
There exists $\theta \in \{\theta_0,\dots,\theta_n\}$ such that $U_j$ and $U_{j+1}$ both lie on opposite sides of $\eta_{\theta}$ and $\partial U_j \cap \partial U_{j+1}$ contains a segment of $\eta_{\theta}$ whose distance from $\partial \h$ is positive.

We will prove the claim for $\h \setminus \cup_{i=0}^j \eta_{\theta_i}$ using induction on $0 \leq j \leq n$.  Note that the claim is true a.s.\ for the set $\h \setminus \eta$,  where $\eta$ is an $\SLE_{\kappa}(\rho^{\text{L}};\rho^{\text{R}})$ process in $\h$ from $0$ to $\infty$ with $\rho^{\text{L}},\rho^{\text{R}}>-2$ and the force points located at $0^-$ and $0^+$.
If we have equality in either equation in (\ref{eqn:fan_gen_bd}), note that we can ignore the flow line $\eta_{\tht_0}$ or $\eta_{\tht_n}$ (or both) which corresponds to $\R_-$ or $\R_+$, since this flow line will not effect $\h\setminus \cup_{j=0}^n \eta_{\theta_j}$, so we may assume that we do not have equality in (\ref{eqn:fan_gen_bd}).
Hence the $j=0$ case follows since $\eta_{\tht_0} = \eta_{\tht'_2}$ has the law of such a process with $\rho^L = -1 + (a - \tht'_2 \chi)/\la$ and $\rho^R = -1 + (b+\tht'_2 \chi)/\la,$ each of which is greater than $-2$ since we have a strict inequality in \eqref{eqn:fan_gen_bd}.
Next,  suppose that the claim holds for $j$ where $0 \leq j \leq n-1$.  Note that $\eta_{\theta_j}$ has the law of an $\SLE_{\kappa}(\rho^{\text{L}};\rho^{\text{R}})$ with $\rho^{\text{L}} = -1 + (a - \tht_j \chi)/\la$ and $\rho^{\text{R}} = -1 + (b+\tht'_2\chi)/\lambda$.  Suppose first that $\rho^{\text{R}} \geq \kk/2 - 2$ and so $\eta_{\theta_j}$ does not hit $(0,\infty)$ a.s.
Let $U,V$ be two distinct connected components of $\h \setminus \cup_{i=0}^{j+1} \eta_{\theta_i}$.  Suppose that both of $U$ and $V$ lie to the left of $\eta_{\theta_j}$.
Then,  $U$ and $V$ are both connected components of $\h \setminus \cup_{i=0}^j \eta_{\theta_j}$.  Let $U_1,\dots,U_m$ be the chain of connected components of $\h \setminus \cup_{i=0}^j \eta_{\theta_j}$ connecting $U$ to $V$ and satisfying the properties of the induction hypothesis.
If none of the $U_i$'s is the connected component of $\h \setminus \cup_{i=0}^j \eta_{\theta_i}$ lying to the right of $\eta_{\theta_j}$,  then all of the $U_i$'s are connected components of $\h \setminus \cup_{i=0}^{j+1} \eta_{\theta_i}$ and so the claim holds. 
Otherwise, let $i \in \{1,\dots,m\}$ be such that $U_i$ is the connected component of $\h \setminus \cup_{i=0}^j \eta_{\theta_i}$ lying to the right of $\eta_{\theta_j}$. 
Let also $I_i$ and $I_{i+1}$ be segments of $\eta_{\theta_j}$ such that $\dist(I_i \cup I_{i+1} ,  \partial \h) > 0$,  $I_i \subseteq \partial U_{i-1} \cap \partial U_i$ and $I_{i+1} \subseteq \partial U_i \cap \partial U_{i+1}$.
Note that conditionally on $\eta_{\theta_j}$,  the curve $\eta_{\theta_{j+1}}$ is an  $\SLE_{\kappa}(\rho^L_{j+1}, \rho^R_{j+1})$ from $0$ to $\infty$ in the connected component of $\h \setminus \eta_{\theta_j}$ lying to the right of $\eta_{\theta_j}$, where
\begin{equation}\label{eq:rhoLj}
    \rho_{j+1}^L = \frac{(\theta_j-\theta_{j+1})\chi}{\lambda} - 2, \quad \rho^R_{j+1} = \frac{b+\theta_{j+1}\chi}{\lambda} - 1.
\end{equation}
Note also that we can pick points $x_{i} \in I_i$ and $x_{i+1} \in I_{i+1}$ in a way which is measurable with respect to $\sigma(\eta_{\theta_0},\dots,\eta_{\theta_j})$ and such that $x_{i}$ (resp.\ $x_{i+1}$) lies in the interior of $I_{i}$ (resp.\ $I_{i+1}$).
Since $\SLE_{\kappa}(\rho^{\text{L}};\rho^{\text{R}})$ processes with $\rho^{\text{L}},\rho^{\text{R}} > -2$ do not hit fixed points a.s.,  we obtain that there exist connected components $U^i,V^i$ of the complement of $\eta_{\theta_{j+1}}$ in the connected component of $\h \setminus \eta_{\theta_j}$ lying to the right of $\eta_{\theta_j}$ and segments $J_{i}$ and $J_{i+1}$ of $\eta_{\theta_j}$ such that $x_{i} \in J_{i} \subseteq I_{i} \cap \partial U^i,  x_{i+1} \in J_{i+1} \subseteq I_{i+1} \cap \partial V^i$ and both of $U^i$ and $V^i$ lie to the left of $\eta_{\theta_{j+1}}$.
Then,  the $j=0$ case implies that there exist connected components $U_1^i,\dots,U_{k_i}^i$ of the complement of $\eta_{\theta_{j+1}}$ in the connected component of $\h \setminus \eta_{\theta_j}$ lying to the right of $\eta_{\theta_j}$ such that the following hold.  $U^i = U_1^i,  V^i = U_{k_i}^i$ and for every $1 \leq \ell \leq k_i-1$ the following holds.  $U_{\ell}^i$ and $U_{\ell+1}^i$ lie on opposite sides of $\eta_{\theta_{j+1}}$ and $\partial U_{\ell}^i \cap \partial U_{\ell+1}^i$ contains a segment of $\eta_{\theta_{j+1}}$ whose distance from $\eta_{\theta_j} \cup \partial \h$ is positive.  Note that all of the $U_{\ell}$'s are connected components of $\h \setminus \cup_{i=0}^{j+1} \eta_{\theta_i}$.  Thus,  by repeating the above procedure for every $1 \leq i \leq m$ such that $U_i$ is the connected component of $\h \setminus \cup_{\ell=0}^j \eta_{\theta_{\ell}}$ lying to the right of $\eta_{\theta_j}$,  we obtain the claim.  

Suppose now that $U$ lies to the left of $\eta_{\theta_j}$ and $V$ lies to the right of $\eta_{\theta_j}$.  Let $G$ be the connected component of $\h \setminus \cup_{i=0}^j \eta_{\theta_i}$ lying to the right of $\eta_{\theta_j}$.  Then there exist connected components $U_1,\dots,U_m$ of $\h \setminus \cup_{i=0}^j \eta_{\theta_i}$ lying to the left of $\eta_{\theta_j}$ with the same properties as in the induction hypothesis such that $U = U_1$ and $U_m = G$.  Also,  $\partial U_{m-1} \cap G$ contains a segment $I$ of $\eta_{\theta_j}$ such that $\dist(I,\partial \h) > 0$ and choose $x$ in the interior of $I$ in a way which is measurable with respect to $\sigma(\eta_{\theta_0},\dots,\eta_{\theta_j})$.  Again,  a.s.\ there exists a connected component $F$ of $G \setminus \eta_{\theta_{j+1}}$ lying to the left of $\eta_{\theta_{j+1}}$ and a segment $J$ of $\eta_{\theta_j}$ such that $x \in J \subseteq I \cap \partial F$.  Note that $U_1,\dots,U_{m-1}$ are all connected components of $\h \setminus \cup_{i=0}^{j+1} \eta_{\theta_i}$.  If $F=V$,  then the chain $U_1,\dots,U_{m-1},F$ satisfies the conditions of the claim.  If $F \neq V$,  then the $j=0$ case implies that there exist connected components $V_1,\dots,V_{\ell}$ of $G \setminus \eta_{\theta_{j+1}}$ such that $F=V_1,  V=V_{\ell}$ and for every $1 \leq i \leq \ell-1$,  the following holds.  $\partial V_i \cap \partial V_{i+1}$ contains a segment of $\eta_{\theta_{j+1}}$ with positive distance from $\eta_{\theta_j} \cup \partial \h$ and $V_i,V_{i+1}$ lie on opposite sides of $\eta_{\theta_{j+1}}$.  The claim then follows since all of the $V_i$'s are connected components of $\h \setminus \cup_{i=0}^{j+1} \eta_{\theta_i}$. Finally,  if $U,V$ both lie to the right of $\eta_{\theta_j}$,  then the claim follows similarly since all of the connected components of $G \setminus \eta_{\theta_{j+1}}$ are also connected components of $\h \setminus \cup_{i=0}^{j+1} \eta_{\theta_i}$.

It remains to treat the case that $\rho^{\text{R}} \in (-2,\kappa/2-2)$ and so $\eta_{\theta_j}$ hits $(0,\infty)$ a.s.  This follows by arguing as in the previous paragraphs and noting that conditionally on $\eta_{\theta_j}$,  the curve $\eta_{\theta_{j+1}}$ has the law of an $\SLE_{\kappa}(\rho^L_{j+1}, \rho^R_{j+1})$ independently in each of the connected components of $\h \setminus \eta_{\theta_j}$ which lie to the right of $ \eta_{\theta_j}$ and between their two marked points. Here, $\rho^L_{j+1}$ and $\rho^R_{j+1}$ are as in \eqref{eq:rhoLj}.

\noindent{\it Step 5.  Conclusion of the proof.} Now we complete the proof of the theorem.  Let $U,V$ be two distinct connected components of $\h \setminus \F(\tht_1', \tht_2')$.  Note that the connected components of $\h \setminus (\eta_{\theta_0} \cup \eta_{\theta_n})$ lying to the left (resp.\ right) of $\eta_{\theta_0}$ (resp.\ $\eta_{\theta_n}$) are also connected components of $\h \setminus \F(\tht_1', \tht_2')$.  Also,  for every $0 \leq j \leq n-1$,  we have that every connected component of $\h \setminus \fan(\theta_{j+1},\theta_j)$ lying between $\eta_{\theta_j}$ and $\eta_{\theta_{j+1}}$ is a connected component of $\h \setminus \F(\tht_1', \tht_2')$ due to the flow line interaction rules.  Hence,  it follows from Step 3 that if $U$ and $V$ are both contained in the same connected component of $\h \setminus \cup_{i=0}^n \eta_{\theta_i}$,  then $U$ and $V$ can be connected via a finite chain of components as desired.  Let $\wt{U}$ (resp.\ $\wt{V}$) be the connected component of $\h \setminus \cup_{i=0}^n \eta_{\theta_i}$ containing $U$ (resp.\ $V$),  and suppose that $\wt{U} \neq \wt{V}$.  Then,  there exist connected components of $\h \setminus \cup_{i=0}^n \eta_{\theta_i}$,  $\wt{U}_1,\dots,\wt{U}_m$ such that $\wt{U} = \wt{U}_1,  \wt{V} = \wt{U}_m$ and for every $1 \leq i \leq m-1$ the following is true. 
There exist $j_i \in \{0,\dots,n\}$ and a segment $J_i$ of $\eta_{\theta_{j_i}}$ such that $J_i \subseteq \partial \wt{U}_i \cap \partial \wt{U}_{i+1}$,  $\dist(J_i,\partial \h) > 0$ and $\wt{U}_i,\wt{U}_{i+1}$ lie on opposite sides of $\eta_{\theta_{j_i}}$.  Fix $i \in \{2,\dots,m\}$ and suppose that both of $\wt{U}_{i-1}$ and $\wt{U}_i$ lie between $\eta_{\theta_0}$ and $\eta_{\theta_n}$.  Then,  we have that $2 \leq j_{i-1} \leq m-1$ and we can assume that $\wt{U}_{i-1}$ (resp.\ $\wt{U}_i$) lies to the left (resp.\ right) of $\eta_{\theta_{j_{i-1}}}$.  Note that a flow line of $h$ with angle in $[\theta_{j_{i-1}-1},\theta_{j_{i-1}+1}]$ corresponds to a flow line of $h_{j_{i-1}}$ with angle in $[-\epsilon/2,0]$.  Fix a point $x_{i-1}$ in the interior of $J_{i-1}$ and let $y_{i-1} \in \h,  r>0$ be such that $x_{i-1} \in B(y_{i-1},r/2),  B(y_{i-1},r) \cap \eta_{\theta_{j_{i-1}}} \subseteq J_{i-1}$ and $E_{y_{i-1},r}^{h_{j_{i-1}}}(\epsilon)$ occurs.
Let $\wt{U}_1^i,\dots,\wt{U}_{k_i}^i$ be the connected components as in the definition of $E_{y_{i-1},r}^{h_{j_{i-1}}}(\epsilon)$ and let $\gamma_i$ be the corresponding path.  Arguing as in Step 3,  we obtain that $\cup_{j=1}^{k_i} \wt{U}_j^i \subseteq \h \setminus (\fan(\theta_{j_{i-1}},\theta_{j_{i-1}-1}) \cup \fan(\theta_{j_{i-1}+1},\theta_{j_{i-1}}))$ and since $\gamma_i$ crosses $J_{i-1}$,  it follows that there exist $i_1,i_2 \in \{1,\dots,k_i\}$ such that $\wt{U}_{i_1}^i \subseteq \wt{U}_{i-1}$ and $\wt{U}_{i_2}^i \subseteq \wt{U}_i$,  and $\partial \wt{U}_{i_1}^i \cap \partial \wt{U}_{i_2}^i \subseteq J_{i-1}$.  Let $\wt{V}_{i_1}^i$ (resp.\ $\wt{V}_{i_2}^i$) be the connected component of $\h \setminus \fan(\theta_{j_{i-1}},\theta_{j_{i-1}-1})$ (resp.\ $\h \setminus \fan(\theta_{j_{i-1}+1},\theta_{j_{i-1}})$) containing $\wt{U}_{i_1}^i$ (resp.\ $\wt{U}_{i_2}^i$).
Then $\wt{V}_{i_1}^i \subseteq \wt{U}_{i-1},  \wt{V}_{i_2}^i \subseteq \wt{U}_i$ and $\partial \wt{V}_{i_1}^i \cap \partial V_{i_2}^i \neq \emptyset$.  Moreover,  by Step 3,  any connected component of $\h \setminus \fan(\theta_{j_{i-1}},\theta_{j_{i-1}-1})$ (resp.\ $\h \setminus \fan(\theta_{j_{i-1}+1},\theta_{j_{i-1}})$) contained in $\wt{U}_{i-1}$ (resp.\ $\wt{U}_i$) can be connected to $\wt{V}_{i_1}^i$ (resp.\ $\wt{V}_{i_2}^i$) via a finite chain of connected components of $\h \setminus \fan(\theta_{j_{i-1}},\theta_{j_{i-1}-1})$ (resp.\ $\h \setminus \fan(\theta_{j_{i-1}+1},\theta_{j_{i-1}})$) contained in $\wt{U}_{i-1}$ (resp.\ $\wt{U}_i$).  But these sets are also connected components of $\h \setminus \F(\tht_1', \tht_2')$ and so any two connected components of $\h \setminus \F(\tht_1', \tht_2')$ contained in $\wt{U}_{i-1}$ and $\wt{U}_i$ respectively can be connected via a finite chain of connected components in $\h \setminus \F(\tht_1', \tht_2')$.  A similar argument shows that the same is true when either $\wt{U}_{i-1}$ or $\wt{U}_i$ does not lie between $\eta_{\theta_0}$ and $\eta_{\theta_n}$.  The claim then follows by fixing a connected component $\wh{U}_i$ of $\h \setminus \F(\tht_1', \tht_2')$ contained in $\wt{U}_i$ for every $2 \leq i \leq m-1$.
\end{proof}

\section{The fan determines flow lines}
\label{sec:fan_determines_flow_lines}
In this section we will prove Theorem~\ref{thm:fan_determines_flow_lines}. As we will explain in Lemma~\ref{lem:boundary_conditions}, each connected component $U$ of $\h \sm \F(\tht_1, \tht_2)$ a.s.\ has an associated angle $\tht(U)$ which satisfies the property that if $\phi > \tht(U)$ (resp.\ $\phi < \tht(U)$) then $\eta_\phi$ passes to the left (resp.\ right) of $U$.
The core part of our argument is to show that $\tht(U)$ is a.s.\ measurable with respect to the fan $\F(\tht_1, \tht_2)$ as a set.
Let us first explain why this suffices to prove that for each $\tht \in [\tht_1, \tht_2]$, the flow line $\eta_\tht$ is a.s.\ measurable with respect to the fan.
Fix $\tht \in [\tht_1, \tht_2]$ and define $L_\tht \subseteq \h$ to be the closure of the union of all components $U$ with $\tht(U) > \tht$. This set is also determined by $\F(\tht_1, \tht_2)$.
By Lemma~\ref{lem:fan_does_not_contain_point},  $\F(\tht_1, \tht_2)$ has measure $0$ a.s. Therefore,  for any point $z$ lying to the left of $\eta_\tht$, there exist points $z' \notin \F(\tht_1, \tht_2)$ arbitrarily close to $z$ and also lying to the left of $\eta_\tht$.
If $U = U(z')$ is the connected component of such a point, then $\tht(U) > \tht$, meaning that $z' \in L_\tht$. It follows that $z \in L_\tht$. Similarly, if $z$ is to the right of $\eta_\tht$, then $z \notin L_\tht$ a.s.
Furthermore, notice that the above holds for all points $z \in \h \sm \eta_\tht$ simultaneously on the almost sure events that $\F(\tht_1, \tht_2)$ has measure $0$, $\eta_\tht$ is a well-defined simple curve, and $\tht(U)$ exists as above for every component $U$ of the complement of the fan.
On this event, it follows that $L_\tht$ is exactly that set of points lying to the left of $\eta_\tht$, along with $\eta_\tht$ itself, meaning that $\eta_\tht$ is the right boundary of this set. Therefore, $\eta_\tht$ is a.s.\ determined by $\F(\tht_1, \tht_2).$ 

It remains to show that $\tht(U)$ exists and is determined by $\F(\tht_1, \tht_2)$ a.s. We will carry this out in the following steps.  Throughout, we let $\CF$ be the $\sigma$-algebra generated by the flow lines of $h$ used to generate $\F(\theta_1,\theta_2)$.
\begin{enumerate}[(i)]
\item In Lemma~\ref{lem:fan_does_not_contain_point}, we will record the fact that the probability that any fixed interior point is contained in $\fan(\theta_1,\theta_2)$ is equal to zero.
\item In Proposition~\ref{prop:locally_connected} we show that $\F(\tht_1, \tht_2)$ is locally connected as a set a.s. We will use this to show that the boundaries of certain domains depending on the fan are (not necessarily simple) curves.
\item In Lemma~\ref{lem:form_of_the_boundary}, we will show that for each connected component $U$ of $\h \setminus \fan(\theta_1,\theta_2)$ there exists $\theta = \theta(U) \in [\theta_1,\theta_2]$ and $x = x(U), y = y(U) \in \partial U$ so that the boundary conditions for the conditional law of $h$ given $\CF$ along $\partial U$ are given by those of the right (resp.\ left) side of a flow line of angle $\theta$ along the clockwise (resp.\ counterclockwise) arc of $\partial U$ from $x$ to $y$.   We will not rule out the possibility that one of these two sides is degenerate.  Here, we view the points in $\partial U$ as prime ends in $U$.
\item In Lemmas~\ref{lem:flow_line_representation}--\ref{lem:discover_component}, we will show that the complementary component boundaries can be represented as flow lines of a conditional GFF.  This will allow us to deduce that neither of the two boundary segments is degenerate and deduce in Lemma~\ref{lem:component_boundary_intersection} the manner that the component boundaries interact with each other is the same as for flow lines with the corresponding angles.
\item We will then show in Lemma~\ref{lem:two_points_measurable} that for all connected components $U$ of $\h \setminus \fan(\theta_1,\theta_2)$ the pair of marked boundary points $\{x(U), y(U)\}$ is measurable with respect to $\fan(\theta_1,\theta_2)$ and then use this to complete the proof.
\end{enumerate}

\subsection{General properties of the fan}\label{subsec:general_properties}

\begin{lemma}
\label{lem:fan_does_not_contain_point}
For each $z \in \h$ we have that $\p[ z \in \fan(\theta_1,\theta_2) ] = 0$.
\end{lemma}
\begin{proof}
It will be more convenient to consider the setup in the real strip $\strip = \R \times (0,\pi)$,  where $\fan(\theta_1,\theta_2)$ is the $\SLE$ fan from $0$ to $i\pi$ with angle range in $[\theta_1,\theta_2]$ of the $\text{GFF}$ $h$ on $\strip$,  whose boundary conditions are given by $-a$ (resp.\ $-a-\pi\chi$) on $\R_-$ (resp.\ $\R_- + i\pi$) and $b$ (resp.\ $b+\pi\chi$) on $\R_+$ (resp.\ $\R_+ + i\pi$).  Fix $z \in \strip$.  First,  we assume that the claim of the lemma holds in the case that $\theta_2 - \theta_1 \leq  \pi$
and  $\tht_1 > (-b+\la)/\chi$ and $\tht_2 < (a+\la)/\chi$.
Suppose that $\theta_2 - \theta_1 > \pi$ and  $\tht_1 > (-b+\la)/\chi$ and $\tht_2 < (a+\la)/\chi$.  We fix $n \in \N$ and $\wt{\theta}_1<\cdots<\wt{\theta}_n$ such that $\theta_1 = \wt{\theta}_1,  \theta_2 = \wt{\theta}_n$,  and $\wt{\theta}_j - \wt{\theta}_{j-1} < \pi$ for each $2 \leq j \leq n$.  Then we have that $\fan(\theta_1,\theta_2) = \cup_{j=2}^n \fan(\wt{\theta}_{j-1},\wt{\theta}_j)$ and $\p[z \in \fan(\wt{\theta}_{j-1},\wt{\theta}_j)] = 0$ for all $2 \leq j \leq n$.  It follows that $\p[z \in \fan(\theta_1,\theta_2)]  = 0$.  
To extend the result to the case that $\tht_1 = (-b+\la)/\chi$ and $\tht_2 = (a+\la)/\chi$, let $\tht_1^n \downarrow \tht_1$ and $\tht_2^n \uparrow \tht_2$. Then for a fixed point $z \in \CS$ there a.s.\ exist a random index $n$ such that $z$ lies to the left of $\eta_{\tht_1^n}$ and to the right of $\eta_{\tht_2^n}$.
By the flow line interaction rules, in this case if $z \notin \F(\tht_1^n, \tht_2^n)$ then $z \notin \F(\tht_1, \tht_2)$. Combining the above with the fact that $z \notin \F(\tht_1^n, \tht_2^n)$ a.s.\ for any fixed choice of $n$ shows that $z \notin \F(\tht_1, \tht_2)$ a.s.
Hence,  in order to complete the proof,  it suffices to show the claim in the case that the angle gap of the fan is at most $\pi$ and $\tht_1 > (-b+\la)/\chi$ and $\tht_2 < (a+\la)/\chi$.

Suppose that $\theta_2-\theta_1 \leq \pi$.  Set $\wt{h} = h + \left(\theta_1+\frac{\pi}{2}\right) \chi$ and note that $\fan(\theta_1,\theta_2)$ is the $\SLE$ fan of $\wt{h}$ from $0$ to $i\pi$ with angle range in $[-\pi/2,\theta_2-\theta_1 - \pi/2] \subseteq [-\pi/2,\pi/2]$.  Let $\eta'$ be the counterflow line of $\wt{h}$ from $i\pi$ to $0$.  Then,  $\eta'$ has the law of an $\SLE_{\kappa'}(\rho^L;\rho^R)$ process in $\strip$ from $i\pi$ to $0$ with the force points located at $(i\pi)^-$ and $(i\pi)^+$ respectively and such that 
\begin{align*}
\rho^L = -1+\frac{b + (\tht_1 + \frac{3\pi}{2})\chi}{\lambda'},\quad 
\rho^R = -1+ \frac{a - (\tht_1 - \frac{\pi}{2})\chi}{\lambda'}.
\end{align*}
By assumption, $\tht_1 > (-b+\la)/\chi$ and $\tht_1 < \tht_2 < (a+\la)/\chi$, from which we can deduce that $\rho^L, \rho^R > -2$ meaning that $\eta'$ is well-defined a.s. Furthermore, every flow line of $\wt{h}$ with angle in $[-\frac\pi2, \tht_2 - \tht_1 -\frac\pi2]$ is a.s.\ well-defined and
contained in the range of $\eta'$,  which implies that $\fan(\theta_1,\theta_2) \subseteq \eta'$. 
For $z \in \h$, let $\tau'$ be the first time that $\eta'$ hits $z$.  If $\p[\tau'<\infty] = 0$,  then it follows that $\p[z \in \fan(\theta_1,\theta_2)] =0$.  
If $\p[\tau' < \infty]>0$,  then by applying the same argument as in the one given in the proof of \cite[Proposition~7.33]{ms2016imag1},  we obtain that $\p[\eta'(\tau') \in \fan(\theta_1,\theta_2) \giv \tau'<\infty] = 0$.  
Combining,  we obtain that $\p[z \in \fan(\theta_1,\theta_2)] = 0$ in every case.  This completes the proof of the lemma.
\end{proof}

First we state and prove a reversal symmetry result of the $\SLE$ fan based on the results of \cite{ms2016imag2}.

\begin{lemma}\label{lem:fan_symmetry}
Let $\phi \colon \h \to \h$ be the conformal transformation defined by $\phi(z) = -1/z$.  Then we have that $\phi(\fan(\theta_1,\theta_2))$ when viewed as a set in $\overline{\h}$ has the same law with the $\SLE$ fan with angle range given by $[\wt{\theta}_1,\wt{\theta}_2]$ of the $\text{GFF}$ in $\h$ with boundary conditions given by $0$ (resp.\ $a+b$) in $\R_-$ (resp.\ $\R_+$),  where $\wt{\theta}_1 = -b/\chi-\theta_2$ and $\wt{\theta}_2 = -b/\chi-\theta_1$.
\end{lemma}

\begin{proof}
Let $(\phi_n)$ be an enumeration of $\Q \cap  [\theta_1,\theta_2]$ and let $\eta_n$ be the flow line of $h$ from $0$ to $\infty$ with angle $\phi_n$,  for all $n \in \N$.  We order $\{\phi_1,\cdots,\phi_n\}$ in an increasing way such that $\phi_{i(1)}<\phi_{i(2)}<\cdots<\phi_{i(n)}$,  and we set $\wt{\eta}_j = \phi(\eta_{i(n-j+1)})$,  $\wt{\phi}_j = -b/\chi-\phi_{i(n-j+1)}$,  and we view $\wt{\eta}_j$ as a curve in $\overline{\h}$ from $0$ to $\infty$, for all $1 \leq j \leq n$. Note that $\eta_{i(n)}$ has the law of an $\SLE_{\kappa}(\rho_1;\rho_2)$ process in $\h$ from $0$ to $\infty$ with force points located at $0^-$ and $0^+$ respectively, and where
\[\rho_1 = -1+\frac{a - \phi_{i(n)} \chi}{\la},\quad \rho_2 = -1 + \frac{b + \phi_{i(n)} \chi}{\la}.\]  
It follows from \cite[Theorem~1.1]{ms2016imag2} that $\wt{\eta}_1$ has the law of an $\SLE_{\kappa}(\rho_2;\rho_1)$ process in $\h$ from $0$ to $\infty$ with the force points located at $0^-$ and $0^+$ respectively.  Moreover,  for each $1 \leq j \leq n-1$,  we let $\wt{D}_j$ be the union of the connected components of $\h \setminus \cup_{i=1}^j \wt{\eta}_i$ lying to the left of $\wt{\eta}_j$.  Then we have that $\wt{D}_j = \phi(D_j)$,  where $D_j$ is the union of the connected components of $\h \setminus \cup_{k=1}^j \eta_{i(n-k+1)}$ lying to the right of $\eta_{i(n-j+1)}$.  Also,  the opening (resp.\ closing) point of a component in $D_j$ is mapped via $\phi$ to the closing (resp.\ opening) point of a component in $\wt{D}_j$.  
Furthermore,  it follows from \cite[Proposition~7.4]{ms2016imag1} that the conditional law of $\eta_{i(n-j)}$ restricted to $D_j$ given $\ss(\eta_{i(n)}, \eta_{i(n-1)}, \dots, \eta_{i(n-j+1)})$ (which is equal to $\sigma(\wt{\eta}_1,\cdots,\wt{\eta}_j)$) is that of an $\SLE_{\kappa}(\rho_1;\rho_2)$ process independently in each connected component in $D_j$ from the opening to the closing point of the component,  and with the force points located immediately to left and right of the opening point respectively,  where here
\[\rho_1 = -2+\frac{(\phi_{i(n-j+1)} - \phi_{i(n-j)})\chi}{\la}, \quad \rho_2 = -1 + \frac{b + \phi_{i(n-j)}\chi}{\la}.\]
Next,  \cite[Proposition~7.4]{ms2016imag1} and \cite[Theorem~1.1]{ms2016imag2} together imply that the conditional law of the restriction of $\wt{\eta}_{j+1}$ to $\wt{D}_j$ given $\sigma(\wt{\eta}_1,\cdots,\wt{\eta}_j)$ is that of an $\SLE_{\kappa}(\rho_2;\rho_1)$ process independently in each of the components in $\wt{D}_j$ from the opening to its closing point,  and with the force points located immediately to the left and right of its starting point.  Therefore,  we obtain that the joint law of $(\wt{\eta}_1,\cdots,\wt{\eta}_n)$ can be sampled as follows.  Let $\wt{h}$ be a $\text{GFF}$ on $\h$ with boundary conditions given by $0$ (resp.\ $a+b$) on $\R_-$ (resp.\ $\R_+$).  Then the curve $\wt{\eta}_j$ is the flow line of $\wt{h}$ from $0$ to $\infty$ of angle $\wt{\phi}_j$,  for all $1 \leq j \leq n$.  The proof is then complete by taking $n \to \infty$.
\end{proof}

Next, we prove that the fan $\F(\tht_1, \tht_2)$ is locally connected (with respect to the subspace topology).
A topological space $X$ is \emph{locally connected at $x \in X$} if every neighborhood (in $X$) of $x$ contains a connected open neighborhood of $x$ (see e.g.\ \cite[\S25]{munkres} or \cite[\S27]{willard}).
$X$ is \emph{connected im kleinen at $x$} (or \emph{weakly locally connected at $x$}) if every neighborhood of $x$ contains a connected (not necessarily open) neighborhood of $x$. The space $X$ itself is called \emph{locally connected} or \emph{connected im kleinen} if it is locally connected at every point $x \in X$ or connected im kleinen at every point $x \in X$, respectively.
While local connectedness and connectedness im kleinen at a single point $x$ are not equivalent, a space $X$ is locally connected if and only if it is connected im kleinen \cite[Theorem~27.16]{willard}.
Before showing the fan is locally connected, we prove two auxiliary lemmas. The proof of the first is adapted from \cite{nd-stackexchange}.

\begin{lemma}\label{lem:connected_components_agree}
Let $F \subseteq \C$ be closed and let $K \subseteq F$ be compact. If $C$ is a connected component of $K$ which does not intersect $\del_F K$, the boundary of $K$ in $F$, then $C$ is a connected component of $F$.
\end{lemma}
\begin{proof}
The quasicomponent of a point $x$ in $K$ is the intersection of all sets $D$ which are clopen in $K$ and contain $x$ \cite[\S6.1]{engelking}. Since $K$ is a compact Hausdorff space, its connected components and quasicomponents agree \cite[Theorem~6.1.23]{engelking}. This means that $C$ is the intersection of all sets $D$ which are clopen in $K$ and which contain $C$. Any point $b \in \del_F K$ is not in $C$, so there must exist a set $U_b$, which is clopen in $K$, contains $b$, and is disjoint from $C$.
The sets $U_b$ for $b \in \del_F K$ form an open cover of $\del_F K$, and since $K$ is compact, we can find a finite subcover whose union $U$ is clopen (in $K$), contains $\del_F K$ and is disjoint from $C$. Then $C \subseteq K\sm U \subseteq K \sm \del_F K \subseteq K \subseteq F$. Also, $K \sm U$ is open in $K$ and contained in $K \sm \del_F K$, hence is open in $F$. Similarly, $K \sm U$ is closed in $K$ and hence closed in $F$. Therefore any connected superset $D$ of $C$ in $F$ must be contained in $K \sm U$ (since this set is clopen) and hence must be equal to $C$ since $C$ is a connected component of $K$.
\end{proof}

\begin{lemma}\label{lem:union_of_locally_connected_sets}
Suppose $A, B \subseteq \C$ are two closed locally connected sets (with respect to the subspace topology). Then $A \cup B$ is locally connected.
\end{lemma}
\begin{proof}
It suffices to prove that $A \cup B$ is connected im kleinen at every point. If $x$ is in $A\sm B$ or $B \sm A$ the claim is immediate. If $x \in A \cap B$, a given neighborhood $N$ of $x$ in $A \cup B$ contains the ball $B(x, \eps) \cap (A \cup B)$ for some $\eps > 0$. Since $A$ and $B$ are each connected im kleinen at $x$, it follows that there exist connected neighborhoods $N_A \subseteq B(x,\eps) \cap A$ and $N_B \subseteq B(x, \eps) \cap B$ of $x$ in $A$ and $B$ respectively. $N_A \cup N_B$ is a union of connected sets with nonempty intersection, so is connected, and is furthermore seen to be a neighborhood of $x$ in $A \cup B$. It follows that $A \cup B$ is connected im kleinen at $x$, completing the proof.
\end{proof}

\begin{proposition}\label{prop:locally_connected}
$\F(\tht_1, \tht_2)$ is locally connected a.s.
\end{proposition}
\begin{proof}
By the proof of Theorem~\ref{thm:fan_connectivity} there exists $\eps_0 > 0$ sufficiently small (depending only on $a,b$ and $\kk$) such that if $\tht_2 - \tht_1 < \eps_0$, then a.s.\ for every $x \in \F(\tht_1, \tht_2)$ and $\ee > 0$ there exists a simple loop $\gamma$ contained in $B(x, \eps)$ which disconnects $x$ from $\infty$ and which intersects $\F(\tht_1, \tht_2)$ finitely many times.
If we can prove that the fan is locally connected in this case, then since the union of finitely many closed locally connected sets is locally connected by Lemma~\ref{lem:union_of_locally_connected_sets}, we see that $\F(\tht_1, \tht_2)$ is locally connected for any choice of $\tht_1, \tht_2$. Therefore we may assume for the remainder of the proof that $\tht_2 - \tht_1 < \eps_0$.

We will prove that for all $x \in \F(\tht_1, \tht_2)$, $\F(\tht_1, \tht_2)$ is connected im kleinen at $x$.
Let $N$ be a neighborhood of $x$ in $\F(\tht_1, \tht_2)$ and choose a simple loop $\gg$ as above with the property that the intersection of $\F(\tht_1, \tht_2)$ with $E$, the closure of the bounded component of $\h \sm \gg$, is contained in $N$ (this is possible since $N$ is a neighborhood of $x$ and $\gg$ can be chosen to be arbitrarily small). Let $E$ be the closure of the bounded component of $\h \sm \gg$. We claim that $\F(\theta_1,\theta_2) \cap E$ has finitely many connected components, and call the set of these components $\CC$.
Since $\F(\tht_1, \tht_2)$ intersects $\gg$ at most finitely many times, there are at most finitely many connected components in $\CC$ which intersect $\gg$. Next, we show that these are the only components in $\CC$. 
Suppose $C \in \CC$ is a connected component which does not intersect $\gg$. Then by Lemma~\ref{lem:connected_components_agree} with $K = \F(\tht_1, \tht_2) \cap E$, we have that $C$ is a connected component of $\F(\tht_1, \tht_2)$, which is a contradiction since $\F(\tht_1, \tht_2)$ is connected (and larger than $C$ itself). This proves the claim that the collection $\CC$ is finite.

Let $C_x \in \CC$ be the connected component containing $x$. Since there are only finitely many other components in $\CC$, it must be the case that the union of these components is at a positive distance from $x$, meaning that $C_x$ itself is a neighborhood of $x$.
Furthermore, $C_x \subseteq N$, meaning that $C_x$ is exactly the connected neighborhood of $x$ we are looking for. It follows that $\F(\tht_1, \tht_2)$ is connected im kleinen at $x$ simultaneously for all points $x \in \F(\tht_1, \tht_2)$ a.s. As stated above, this shows that $\F(\tht_1, \tht_2)$ is locally connected a.s.
\end{proof}

\begin{remark}\label{rem:locally_connected}
This remark relates to objects defined later in this section ($\del U, \F(t;\gg), \F(z)$) and can be skipped for now and read when needed. In the case that $\tht_2 - \tht_1 < \eps_0$, the only facts about $\F(\tht_1, \tht_2)$ we have used to prove local connectedness are the connectedness of the fan, and the property that for every $x \in \F(\tht_1, \tht_2)$ and $\ee > 0$ there exists a simple loop $\gamma$ contained in $B(x, \eps)$ which disconnects $x$ from $\infty$ and which intersects $\F(\tht_1, \tht_2)$ finitely many times.
Therefore, assuming $\tht_2 - \tht_1 < \eps_0$, any connected subset of $\F(\tht_1, \tht_2)$ will also satisfy these two properties, and thus be locally connected. In particular, the sets $\del U, \F(t;\gg)$ and $\F(z)$ are all connected subsets of the fan and thus locally connected.
Furthermore, this still holds even if we do not require $\tht_2 - \tht_1 < \eps_0$, since each of the sets $\del U, \F(t;\gg)$ and $\F(z)$ is actually a subset of a smaller fan $\F(\tht_1', \tht_2')$ where $\tht_2' - \tht_1' < \eps_0$ (in fact here we can replace $\eps_0$ by an arbitrarily small quantity), and $\tht_1', \tht_2'$ are random, but can be chosen from a finite set. We will make use of the local connectedness of these sets throughout this section, particularly in Lemma~\ref{lem:discover_component}, by using \cite[Theorem~2.1]{pommerenke1992boundary} to show that the boundaries of certain domains are (not necessarily simple) curves.

\end{remark}

\subsection{Structure of subsets of the fan}\label{subsec:structure}

Let $C$ be a connected component of $\h \sm (\eta_{\tht_1} \cup \eta_{\tht_2})$ and let $\ff$ be a conformal map from $C$ to $\h$ that is measurable with respect to $\ss(\eta_{\tht_1}, \eta_{\tht_2})$ and which sends the opening (resp.\ closing) point of $C$ to $0$ (resp.\ $\infty$).
Then the conditional law of $\wt{h} := h \circ \ff\nv - \chi \arg(\ff\nv)'$ given $\eta_{\tht_1}, \eta_{\tht_2}$ is that of a GFF with boundary conditions $\la - \tht_2\chi$ on $\R_-$ and $-\la - \tht_1\chi$ on $\R_+$.  By \cite[Proposition~7.4]{ms2016imag1}, for $\tht \in (\tht_1, \tht_2)$ the pair $(\ff_C(\F(\tht_1, \tht_2)), \ff_C(\eta_\tht))$ has the law of $(\wt\F(\tht_1, \tht_2), \wt\eta_\tht)$, where $\wt\F(\tht_1, \tht_2)$ is the fan corresponding to the GFF $\wt{h}$ and $\wt\eta_\tht$ is the flow line of $\wt{h}$ with angle $\tht$.
Suppose that $\wt\eta_\tht$ is determined by $\wt\F(\tht_1, \tht_2)$. Since $\eta_{\tht_1}$ and $\eta_{\tht_2}$ are measurable with respect to $\F(\tht_1, \tht_2)$ (as its left and right boundary in $\h$), we obtain that $\eta_\tht|_C$ is determined by $\F(\tht_1, \tht_2)$. Hence, by repeating this argument in each connected component $C$ of $\h \sm (\eta_{\tht_1} \cup \eta_{\tht_2})$, and noting that each component $C$ is itself determined by $\F(\tht_1, \tht_2)$, we find that $\eta_\tht$ is determined by $\F(\tht_1, \tht_2)$.

Therefore, we only need to show that $\wt\F(\tht_1, \tht_2)$ determines $\wt\eta_\tht$ where $\wt h$ is a GFF with boundary conditions $\la - \tht_2\chi$ on $\R_-$ and $-\la - \tht_1\chi$ on $\R_+$, where $\wt\F(\tht_1, \tht_2)$ is the corresponding fan, and where $\wt\eta_\tht$ is the flow line of $\wt{h}$ with angle $\tht$. In particular, it suffices to prove Theorem~\ref{thm:fan_determines_flow_lines} in the case that $\tht_1 = -(b+\la)/\chi, \tht_2 = (a+\la)/\chi$.
Note that by \cite{ms2016imag1}, the set $(\tht_1,\tht_2)$ is then the maximal range of angles $\tht$ where $\eta_\tht$ does not a.s.\ immediately hit the continuation threshold, meaning that the flow line $\eta_\tht$ exists. By Proposition~\ref{prop:bounded_hausdorff}, we have that $\eta_\tht$ a.s.\ converges in the Hausdorff sense to $\R_+$ (resp.\ $\R_-$) as $\tht \downarrow \tht_1$ (resp.\ $\tht \uparrow \tht_2$).
Using this fact and Lemma~\ref{lem:fan_does_not_contain_point}, for any fixed point $z \in \h$, a.s.\ there exists a unique connected component $U$ of $\h \sm \F(\tht_1, \tht_2)$ which contains $z$, and $\del U$ does not contain any interval of $\R$. If instead we considered $\F(\tht_1', \tht_2')$ with $\tht_1 < \tht_1' < \tht_2' < \tht_2$ this would not be the case, since there would exist components $U$ of $\h \sm \F(\tht_1', \tht_2')$ to the right of $\eta_{\tht_1}$. This would complicate some of our proofs, so it is easier to rule this out from the offset.

In the following, we will assume that $\tht_1 = -(b+\la)/\chi, \tht_2 = (a+\la)/\chi$, and we will fix an enumeration $(\phi_n)$ of $(\Q \cap [\tht_1, \tht_2]) \cup \{\tht_1, \tht_2\}$, where $\phi_1 = \tht_1$ and $\phi_2 = \tht_2$. The fan $\F(\tht_1, \tht_2)$ is a.s.\ unchanged by changing the choice of countable dense subset of $[\tht_1, \tht_2]$, but using this particular enumeration to generate the fan simplifies our proofs.  As mentioned in Section~\ref{subsec:sle_fan}, we will define the flow line $\eta_{\tht_1}$ to be $\R_+$, and the flow line $\eta_{\tht_2}$ to be $\R_-$, which is natural from the discussion above.

\begin{lemma}
\label{lem:form_of_the_boundary}
Fix a point $z \in \h$ with rational coordinates and let $U$ be the connected component of $\h \setminus \fan(\theta_1,\theta_2)$ containing $z$.  We view $\partial U$ as a set of prime ends in $U$. Then there exists $\theta = \theta(U) \in \R$ and $x = x(U), y = y(U) \in \partial U$ so that the conditional law of $h$ given $\CF$ in $U$ is that of a \text{GFF} with boundary conditions given by those of the right (resp.\ left) side of a flow line of angle $\theta$ on the clockwise (resp.\ counterclockwise) arc of $\partial U$ from $x$ to $y$ (we define what this means more precisely below).
\end{lemma}

What we mean by the boundary conditions of the right or left side of a flow line is as follows. Suppose first that $x(U)$ and $y(U)$ are distinct prime ends of $\del U$, and let $g \colon U \to \h$ be a conformal map sending $x(U)$ and $y(U)$ to $0$ and $\infty$ respectively. The choice of $g$ will not be important for what we are about to do, but it can be fixed in some measurable way.
Then a harmonic function $\Fh$ on $U$ is said to have boundary values given by the right (resp.\ left) side of a flow line of angle $\tht$ on the clockwise (resp.\ counterclockwise) arc of $\del U$ from $x$ to $y$, if $\Fh$ has boundary conditions $\la - \tht\chi - \chi\arg g'$ (resp.\ $-\la - \tht \chi - \chi \arg g'$) on the clockwise (resp.\ counterclockwise) arc of $\del U$ from $x$ to $y$.
Note that the branch of the argument is not fixed here, so there are many different choices of such a function $\Fh$, each differing from the others by a global multiple of $2\pi\chi$. Clearly, $\arg g' \pmod{2\pi\chi}$ is measurable with respect to $\ss(U, x, y)$, but the choice of branch of $\arg$ is not measurable with respect to $\ss(U, x, y)$ since it depends on the realization of $\F(\tht_1, \tht_2)$ on $\h \sm U$.

If $x = y$, choose a prime end $v \in \del U$ not equal to $y$, and let $g \colon U \to \h$ be a conformal map sending $v$ to $0$ and $y$ to $\infty$. Then the harmonic function $\Fh$ has the boundary values of the left (resp.\ right) side of a flow line of angle $\tht$ if the boundary values are given by $-\la - \tht \chi - \chi\arg g'$ (resp.\ $\la - \tht\chi - \chi\arg g'$) on all of $\del U$.
Note that the choice of $v$ does not affect the boundary conditions, and that, as before, the choice of branch of the argument means that there are many different choices for such a function~$\Fh$.

\begin{proof}
\noindent{\it Step 1.  Overview and setup.}
Fix $z \in \h$ with rational coordinates and let $(\phi_n)$ be our enumeration of $(\Q \cap [\tht_1, \tht_2]) \cup \{\tht_1, \tht_2\}$.  For each $n \in \N$, let $\eta_n$ be the flow line of $h$ from $0$ to $\infty$ with angle $\phi_n$ and work on the a.s.\ event that  $z \notin \F(\tht_1, \tht_2)$.
Recall that $\eta_1 = \R_+$ and $\eta_2 = \R_-$.
For each $n \in \N$, $z$ must either be on the right side of $\eta_n$ (by which we mean that $z$ is in a connected component of $\h\sm\eta_n$ whose boundary contains part of $\R_+$) or on the left side.
By the monotonicity property of flow lines, if $\phi_n < \phi_m$ and $z$ lies to the right of $\phi_n$, then it necessarily lies to the right of $\phi_m$. It follows that we can define the following quantities,
\[\tht = \sup_n\{\phi_n \colon \text{$z$ lies on the left of $\eta_n$}\}, \quad \tht' = \inf_n\{\phi_n \colon \text{$z$ lies on the right of $\eta_n$}\},\]
and that $\tht \leq \tht'$. If this inequality is strict, we can find $\phi_n \in (\tht, \tht')$ and then $z$ must lie either on the left or right side of $\eta_n$, yielding a contradiction (to the definition of either $\tht$ or $\tht'$) in each case. We can conclude therefore that $\tht = \tht'$.

\begin{figure}[t]
\includegraphics[scale=0.8]{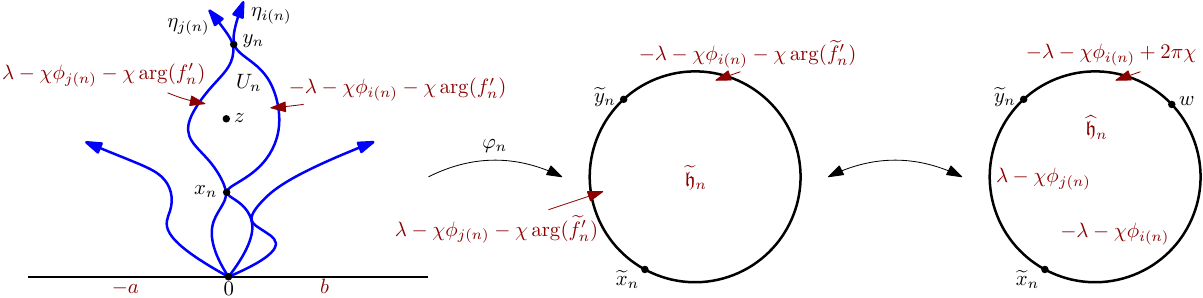}
\caption{On the left we show the finite set of flow lines $\F_n(\tht_1, \tht_2)$ along with the component $U_n$ defined in the proof of Lemma~\ref{lem:form_of_the_boundary}. We include the boundary conditions of $h_2$ restricted to $U_n$. Here, $f_n \colon U_n \to \h$ is a conformal map sending $x_n$ to $0$ and $y_n$ to $\infty$. In the central figure, we show $\wt{x}_n, \wt{y}_n$ and the boundary conditions of $\wt{\Fh}_2^n$ (up to an integer multiple of $2\pi\chi$ which is constant on $\D$ but may depend on $n$). The map $\wt{f}_n \colon \D \to \h$ sends $\wt{x}_n$ to $0$ and $\wt{y}_n$ to $\infty$, and satisfies $\wt{f}_n = f_n \circ \ff_n$. The term $\chi\arg(\wt{f}_n')$ has a jump of $2\pi\chi$ at $\wt{y}_n$. On the right, we depict the same situation, but now we include the point $w$ and show instead the boundary conditions of $\wh{\Fh}_2^n := \wt{\Fh}_2^n + \chi\arg\psi'$ (up to an integer multiple of $2\pi\chi$) in the case that $\wt{y}_n$ lies on the clockwise arc from $\wt{x}_n$ to $w$. This makes it easier to identify the jump in boundary conditions at $\wt{y}_n$. In the case that $\wt{y}_n$ lies on the counterclockwise arc from $\wt{x}_n$ to $w$, which we have not depicted, the boundary conditions of $\wh\Fh_n$ are instead given by. In this case, the boundary conditions of $\wh\Fh_n$ are given (up to an integer multiple of $2\pi\chi$) by $\lambda-\chi \phi_{j(n)}$ on the (clockwise) arc from $\wt x_n$ to $w$, by $\lambda-\chi \phi_{j(n)} - 2\pi \chi$ on the arc from $w$ to $\wt y_n$, and by $-\lambda-\chi \phi_{i(n)}$ on the arc from $\wt y_n$ to $\wt x_n$. In the proof of Lemma~\ref{lem:form_of_the_boundary} we use the local uniform convergence of $\wt{\Fh}_2^n$ to show that the points $\wt{x}_n, \wt{y}_n$ must converge, and using this identify the limit $\wt{\Fh}_2$.}
\label{fig:un_boundary_conditions}
\end{figure}

Let $\fan_n =\fan_n(\theta_1,\theta_2) =  \cup_{j=1}^n \eta_j$ and let~$U_n$ be the connected component of $\h \setminus \fan_n$ which contains $z$. Let $\phi_{i(n)}$ be the maximal angle among $(\phi_k)_{k \leq n}$ for which $z$ is on the left side of $\eta_k$. Let $\phi_{j(n)}$ be the corresponding minimal angle with $z$ on the right of $\eta_{j(n)}$. Since $\eta_1 = \R_+$ and $\eta_2 = \R_-$, $i(n)$ and $j(n)$ are both well-defined as soon as $n \geq 2$, and by the Hausdorff convergence of flow lines (Proposition~\ref{prop:bounded_hausdorff}), a.s. for large enough $n$ we will have that $i(n) \neq 1, j(n) \neq 2$, so that $\eta_{i(n)}, \eta_{j(n)}$ are actual flow lines (and not just $\R_+$ or $\R_-$).
Then there exist $x_n,y_n \in \partial U_n$ distinct which are the opening and closing points of the pocket $U_n$ respectively, so that the counterclockwise (resp.\ clockwise) boundary arc of $U_n$ from $x_n$ to $y_n$ is part of the flow line $\eta_{i(n)}$ (resp.\ $\eta_{j(n)}$).
By the monotonicity of flow lines in their angle, we note that $\phi_{i(n)}$ is increasing in $n$ and therefore converges to some value $\phi$, where necessarily $\phi \leq \theta$. If $\phi < \theta$ then there exists k such that $\phi < \phi_k < \theta$. Suppose $n \geq k$. Then $z$ is on the left side of $\eta_k$, and by the definition of $\eta_{i(n)}$ as the flow line of maximal angle $\phi_j$ with $j \leq n$ for which $z$ is on the left side of $\eta_j$, we must have $\phi_{i(n)} \geq \phi_k > \phi$, a contradiction. We can perform an analogous argument for $\phi_{j(n)}$ to show that $\lim_n \phi_{i(n)} = \lim_n \phi_{j(n)} = \theta = \theta'$. Our setup is depicted in Figure~\ref{fig:un_boundary_conditions}.

For each $n \in \N$, we can write $h|_{\h \setminus \fan_n}$ as $h = h_1^n + h_2^n$ where $h_2^n$ is a distribution which is harmonic in $\h \setminus \fan_n$ and the conditional law of $h_1^n$ given $\CF_n = \sigma(h_2^n ,  \fan_n)$ is that of a $\text{GFF}$ with zero boundary conditions in $\h \setminus \fan_n$.  Also,  we can write $h|_{\h \setminus \fan(\theta_1,\theta_2)}$ as $h = h_1 + h_2$ where $h_2$ is a distribution which is harmonic in $\h \setminus \fan(\theta_1,\theta_2)$ and the conditional law of $h_1$ given $\mathcal{F} = \sigma(h_2,\fan(\theta_1,\theta_2))$ is that of a GFF with zero boundary conditions in $\h \setminus \fan(\theta_1,\theta_2)$. Then \cite[Proposition~6.5]{ms2016imag1} implies that $h_2^n \to h_2$ locally uniformly as $n \to \infty$ a.s.
Let $\Fh_n = h_2^n|_{U_n}$ and let $\Fh = h_2|_U$. In the remainder of the proof we will use the fact that $\Fh_n \to \Fh$ locally uniformly on $U$ and the information we have on the boundary conditions of $\Fh_n$ to describe the boundary conditions of $\Fh$. This is a quite technical argument and the details can be skipped on a first reading. Much of the following is depicted in Figure~\ref{fig:un_boundary_conditions}.
First, let $f_n \colon U_n \to \h$ be a conformal map (chosen in a measurable way) which sends $x_n$ to $0$ and $y_n$ to $\infty$. By \cite[Proposition 6.1]{ms2016imag1}, $\Fh^n = h_2^n|_{U_n}$ is the harmonic function on $U_n$ with boundary values given by $-\la - \chi\phi_{i(n)} - \chi\arg f_n'$ (resp.\ $\la - \chi\phi_{j(n)} - \chi\arg f_n'$) on the counterclockwise (resp.\ clockwise) arc of $\del U_n$ from $x_n$ to $y_n$.
As noted in \cite{ms2016imag1}, the branch of the argument must be chosen so that the boundary values $\Fh_n$ agree with those of the conditional law of $\eta_{\phi_{i(n)}}$ (equivalently $\eta_{j(n)})$) on a segment of $\del U_n$ which agrees with $\eta_{i(n)}$.
In the following argument, often the boundary conditions we specify will be correct up to a choice of branch of $\arg$, or equivalently up to an integer multiple of $2\pi\chi$. We will point out whenever this is relevant, since it does have an affect on our proof.
Note that $\arg f_n'\pmod{2\pi\chi}$ is measurable with respect to $\ss(U_n, x_n, y_n)$, but that the branch of argument we choose is \emph{not} measurable with respect to $\ss(U_n, x_n, y_n)$, since it depends on the behavior of the flow lines $\eta_1, \dots, \eta_n$ outside $U_n$.

For each $n \in \N$, let $\ff_n \colon U_n \to \D$ be the unique conformal map with $\ff_n(z) = 0$ and $\ff_n'(z) > 0$ and let $\ff \colon U \to \D$ be the unique conformal map with $\ff(z) = 0$ and $\ff'(z) > 0$.
The results of \cite[Section~1]{pommerenke1992boundary} imply that $U_n$ converges to $U$ in the Caratheodory kernel sense as seen from $z$ as $n \to \infty$ a.s.\ (see \cite[Section~1]{pommerenke1992boundary} for the definition of Caratheodory kernel convergence).
In particular,  we have that $\varphi_n \to \varphi$ locally uniformly as $n \to \infty$ a.s.
Define $\wt \Fh_n = \Fh_n \circ \ff_n\nv - \chi \arg_0(\ff_n\nv)'$ where $\arg_0$ is the fixed branch of $\arg$ chosen such that $\arg_0 (\ff_n\nv)'(0) = 0$ (this can be done since the derivative is positive here).
Similarly, define $\wt\Fh = \Fh \circ \ff\nv - \chi\arg_0(\ff\nv)'$ where $\arg_0$ is the branch with $\arg_0(\ff\nv)'(0) = 0$.
Since $\Fh_n, \ff_n\nv$ and $(\ff_n\nv)'$ converge locally uniformly to $\Fh, \ff\nv$ and $(\ff\nv)'$, respectively, and since we have fixed the branches of $\arg$ in the same way, we have that $\wt\Fh_n \to \wt\Fh$ locally uniformly as $n \to \infty$.

Define $\wt x_n = \ff_n(x_n), \wt y_n = \ff_n(y_n)$ and let $\wt f_n = f_n \circ \ff_n\nv$ so that $\wt f_n \colon \D \to \h$ and $\wt f_n$ sends $\wt x_n$ to $0$ and $\wt y_n$ to $\infty$. Then $\arg (\wt f_n)' = \arg(f_n' \circ \ff_n\nv) + \arg_0(\ff_n\nv)'$, where the branch of the argument on the left-hand side is determined by the choices of branch on the right, which we have already made (so there is no new choice here).
Then $\wt\Fh_n$ is the unique harmonic function on $\D$ with boundary conditions (up to an integer multiple of $2\pi\chi$ which is constant on $\D$ but may depend on $n$) given by
\begin{align*}
    -\la - \chi\phi_{i(n)} - \chi\arg (\wt f_n)' &\quad\text{on $\del \D_R$,}\\
    +\la - \chi\phi_{j(n)} - \chi\arg(\wt f_n)' &\quad\text{on $\del \D_L$,}
\end{align*}
where $\del \D_R$ (resp.\ $\del \D_L$) is the counterclockwise (resp.\ clockwise) arc of $\del \D$ from $\wt x_n$ to $\wt y_n$. These boundary conditions are depicted in Figure~\ref{fig:un_boundary_conditions}.
In Step 2 we will show that it is a.s.\ the case that both of the sequences $(\wt{x}_n)$ and $(\wt{y}_n)$ converge.
Then,  in Step 3,  we will identify the boundary conditions of $\Fh$ and prove the main claim of the lemma.

\noindent{\it Step 2.  Convergence of $(\wt{x}_n)$ and $(\wt{y}_n)$.} First we will show that a.s.\ $(\wt{x}_n)$ converges to a (random) point $\wt{x} \in \del \D$. The idea is to use the fact that $\wt{\Fh}_n \to \wt{\Fh}_2$ locally uniformly and that the boundary conditions of $\wt{\Fh}_n$ have a discontinuity (which we can describe explicitly) at $\wt{x}_n$ to show that $(\wt{x}_n)$ must converge. Suppose that with positive probability $(\wt{x}_n)$ does not converge.  Then with positive probability,  there exist $d>0,\,  x^\mu, x^\la, y^\mu, y^\la \in \partial \D$ and subsequences $(\wt{x}_{\mu_n}),  (\wt{x}_{\lambda_n})$ such that $\wt{x}_{\mu_n} \to x^\mu,  \wt{x}_{\lambda_n} \to x^\la, \wt{y}_{\mu_n} \to y^\mu$,  and $\wt{y}_{\lambda_n} \to y^\la$ as $n \to \infty$,  and $|x^\mu - x^\la|\geq d$.
We fix $w \notin \{x^\mu, x^\la, y^\mu, y^\la\}$ and let $\psi$ be the conformal transformation mapping $\D$ onto $\h$ such that $\psi(0) = i$ and $\psi(w) = \infty$. 
Let $\arg_0 \psi'$ be a fixed branch of $\arg\psi'$ (the one where $\arg(\psi'(0)) \in (-\pi, \pi]$, say) and define\footnote{Note that $\wh\Fh_n$ is defined on $\D$, not $\h$, so $(\D, \wt\Fh_n)$ and $(\D, \wh\Fh_n)$ are not the same \emph{imaginary surface} in the sense of \cite{ms2016imag1}. We introduce $\wh\Fh_n$ because its boundary conditions are simpler than those of $\wt\Fh_n$.} $\wh\Fh_n = \wt\Fh_n + \chi\arg\psi',\,\wh\Fh = \wt\Fh + \chi\arg\psi'$.
We immediately have that $\wh\Fh_n \to \wh\Fh$ locally uniformly as $n \to \infty$.
Both $f_n \circ \ff_n\nv$ and $\psi$ are conformal maps from $\D$ to $\h$ sending, respectively, $\wt{y}_n$ and $w$ to $\infty$. It can be shown that $\arg(\wt f_n)' - \arg \psi'$ has discontinuities on, $\del \D$, of magnitude $2\pi$ at $\wt{y}_n$ and $w$, and that the boundary conditions of $\wh\Fh_n$ depend on the cyclic order of $\wt x_n, \wt y_n$ and $w$.
In the case that $\wt y_n$ lies on the clockwise arc from $\wt x_n$ to $w$, the boundary conditions (up to a multiple of $2\pi\chi$) of $\wh\Fh_n$ are given by $\lambda- \phi_{j(n)} \chi$ on the (clockwise) arc from $\wt x_n$ to $\wt y_n$, by $-\la - \phi_{i(n)}\chi + 2\pi\chi$ on the arc from $\wt y_n$ to $w$, and by $-\lambda- \phi_{i(n)}\chi$ on the arc from $w$ to $\wt x_n$.
This case is depicted in Figure~\ref{fig:un_boundary_conditions}.
Otherwise, $\wt y_n$ lies on the counterclockwise arc from $\wt x_n$ to $w$, and the boundary conditions are given by $\lambda- \phi_{j(n)}\chi$ on the (clockwise) arc from $\wt x_n$ to $w$, by $\lambda- \phi_{j(n)}\chi - 2\pi \chi$ on the arc from $w$ to $\wt y_n$, and by $-\lambda- \phi_{i(n)}\chi$ on the arc from $\wt y_n$ to $\wt x_n$.
We emphasize that these boundary conditions are correct only up to an integer multiple $2\pi\chi k_n$, with $k_n \in \Z$, coming from the chosen branch of $\arg f_n'$. The quantity $k_n$ is global in the sense that it is constant on $\ov\D$, but it may change in $n$.

Suppose we are in the case that $x^\la\neq y^\la$. By possibly taking $d>0$ to be smaller,  we can also assume that $\dist(w,\{\wt{x}_{\mu_n},\wt{y}_{\mu_n},\wt{x}_{\lambda_n},\wt{y}_{\lambda_n}\}) \geq d$ and that $\dist(\wt{x}_{\la_n}, \{\wt{x}_{\mu_n}, \wt{y}_{\la_n}\}) \geq d$ for all $n$ sufficiently large.
Then the boundary conditions (up to a multiple of $2\pi\chi$) of $\wh{\Fh}_{\lambda_n}$ on $B(\wt{x}_{\lambda_n},d) \cap \partial \D$ are given by $\lambda- \phi_{j(\lambda_n)}\chi$ (resp.\ $-\lambda- \phi_{i(\lambda_n)}\chi$) on the left\footnote{By the \emph{left} side of $\wt x_n$, we mean the arc of $\del \D \cap B(\wt x_n, d)$ on the clockwise side of $\wt x_n$.} (resp.\ right) side of $\wt{x}_{\lambda_n}$. 
If $\wt{y}_{\mu_n} \in B(\wt{x}_{\lambda_n}, d) \cap \del \D$ then the boundary conditions (up to a multiple of $2\pi\chi$) of $\wh{\Fh}_{\mu_n}$ on $B(\wt{x}_{\lambda_n},d) \cap \partial \D$ are given by either $-\lambda- \phi_{i(\mu_n)}\chi+2\pi \chi$ to the left of $\wt{y}_{\mu_n}$ and $\lambda- \phi_{j(\mu_n)}\chi$ to its right, or by $-\lambda- \phi_{i(\mu_n)}\chi$ to the left of $\wt{y}_{\mu_n}$ and $\lambda-\phi_{j(\mu_n)}\chi - 2\pi \chi$ to its right (note that since we are only specifying the boundary conditions up to multiples of $2\pi\chi$, these two possibilities are effectively the same).
If $\wt{y}_{\mu_n}$ is not on this arc, then the boundary values of $\wh{\Fh}_{\mu_n}$ are constant on $B(\wt{x}_{\lambda_n}, d) \cap \del \D$ and are given by one of the four values just listed.
Since $\phi_{i(n)}, \phi_{j(n)} \to \tht$, we can assume $\la_1, \mu_1$ are large enough that all of the angles $\phi_{i(\la_n)}, \phi_{j(\la_n)}, \phi_{i(\mu_n)}, \phi_{j(\mu_n)}$ lie in $(\tht - \DD\tht, \tht + \DD\tht)$ for some small $\DD\tht > 0$.
In any of the cases where we have constant boundary values, it can be checked that there must always be some arc $I_n$ of $B(\wt{x}_{\lambda_n},d) \cap \del \D$ of length at least $d/2$ on which the absolute value of the difference between the boundary values of $\wh{\Fh}_{\mu_n}$ and $\wh{\Fh}_{\la_n}$ is at least $\min(2\pi\chi, 2\la) - 2\chi\DD\tht$. In the case where the boundary values change, by checking various cases we see that there must be an arc $I_n$ of $B(\wt{x}_{\lambda_n},d) \cap \del \D$ of length at least $d/2$ on which the absolute value of the difference between the boundary values is at least $\min(\la, \pi\chi) - 2\DD\tht \chi$. These differences are at least $\chi$ for a small enough choice of $\DD\tht$.
We note that in determining the difference in boundary conditions above, we have accounted for the fact that the boundary conditions of $\wh\Fh_{\la_n}, \wh\Fh_{\la_n}$ have been described only up to multiples of $2\pi\chi$, and that the above statements are true even if we alter $\wh\Fh_{\la_n} - \wh\Fh_{\mu_n}$ by $2\pi\DD K \chi$ for some $\DD K \in \Z$. In particular, these statements are true for the actual values of $\wh\Fh_{\la_n}$ and $\wh\Fh_{\mu_n}$, not just for a particular (possibly incorrect) choice of multiple of $2\pi\chi$.

Fix $\dd \in (0, d/8)$, let $a_n$ be the midpoint of the above arc $I_n$, and let $b_n$ be the point on the segment $[0, a_n]$ such that $|a_n - b_n| = \dd$. It follows that in every case we have that the absolute value of the difference of the boundary values on $B(a_n,d/8) \cap \partial \D$ between $\wh{\Fh}_{\mu_n}$ and $\wh{\Fh}_{\lambda_n}$ is at least $\chi > 0$ for all $n$ sufficiently large.
The Beurling estimate implies that there exists a universal constant $C$ so that the probability that a Brownian motion starting from $b_n$ exits $\D$ in $\C \setminus B(a_n,d/8)$ is at most $C(8\delta/d)^{1/2}$.  Also,  there exists $M<\infty$ depending only on $\kappa$ and $\theta$ so that $||\wh{\Fh}_n||_{\infty} \leq M$ for all $n$.  It follows that
\begin{align*}
|\wh{\Fh}_{\mu_n}(b_n)-\wh{\Fh}_{\lambda_n}(b_n)| \geq \chi - MC(8\delta/d)^{1/2} \geq \frac{\chi}{2} > 0 
\end{align*}
for all $n$ sufficiently large if we choose $\delta$ such that $MC(8\delta/d)^{1/2} < \chi/2$.  But this is a contradiction since 
\begin{align*}
\sup_{z \in \overline{B(0,1-\delta)}} |\wh{\Fh}_{\mu_n}(z)-\wh{\Fh}_{\lambda_n}(z)| \to 0\,\,\text{as}\,\,n \to \infty.
\end{align*}
In the case that $x^\mu \neq y^\mu$ an analogous argument gives a contradiction.

Next,  we assume that $x^\la = y^\la$ and $x^\mu = y^\mu$.  Then, either $\wt{y}_{\lambda_n}$ lies in the counterclockwise arc of $\partial \D$ from $\wt{x}_{\lambda_n}$ to $\wt{x}_{\mu_n}$ for infinitely many values of $n$ or $\wt{y}_{\lambda_n}$ lies in the clockwise arc of $\partial \D$ from $\wt{x}_{\lambda_n}$ to $\wt{x}_{\mu_n}$ for infinitely many values of $n$. Suppose that the former case holds.  Then,  we can assume that $\wt{y}_{\lambda_n}$ lies in the counterclockwise arc of $\partial \D$ from $\wt{x}_{\lambda_n}$ to $\wt{x}_{\mu_n}$ for all $n$.
Let $d>0$ be such that $|x^\la - x^\mu|\geq d$ and $\dist(w,\{x^\la, x^\mu\}) \geq d$.  There exists $n_0$ so that $|\wt{x}_{\lambda_n}-\wt{y}_{\lambda_n}|\leq d/100,  |\wt{x}_{\mu_n}-\wt{y}_{\mu_n}|\leq d/100,  |\wt{x}_{\lambda_n}-x^\la|\leq d/100$,  and $|\wt{x}_{\mu_n}-x^\mu|\leq d/100$,  for all $n \geq n_0$.
The boundary conditions (up to a multiple of $2\pi\chi$) of $\wh\Fh_{\la_n}$ on $B(x^\la, d/2)$ are given by $\la - \phi_{j(\la_n)}\chi$ on the left side of $\wt x_{\la_n}$ and by $\la - \phi_{j(\la_n)}\chi-2\pi\chi$ on the right side of $\wt y_{\la_n}$. The boundary conditions of $\wh\Fh_{\mu_n}$ are constant in this region, so there must exist an arc of $\del \D$ of length at least $d/4$ on which the difference between the boundary values of $\wh\Fh_{\la_n}$ and $\wh\Fh_{\mu_n}$ is at least $\chi$ for all $n \geq n_0$, yielding a contradiction as before. If $\wt{y}_{\lambda_n}$ lies in the clockwise arc of $\partial \D$ from $\wt{x}_{\lambda_n}$ to $\wt{x}_{\mu_n}$ for infinitely many values of $n$,  then a similar argument leads to a contradiction.  This proves the convergence of $(\wt{x}_n)$. 

Now we prove the convergence of $(\wt{y}_n)$ to a point $\wt{y}$ using Lemma~\ref{lem:fan_symmetry}.  Let $\phi$ be as in Lemma~\ref{lem:fan_symmetry} and set $\wt{\fan} = \phi(\fan(\theta_1,\theta_2))$ and $\wt{\fan}_n = \phi(\fan_n)$ for all $n$.  Then we have that $\wt{U}_n = \phi(U_n)$ is the connected component of $\h \setminus \wt{\fan}_n$ containing $\wt{z} = \phi(z)$ for all $n$.  Moreover,  we have that $\phi(y_n)$ (resp.\ $\phi(x_n)$) is the opening (resp.\ closing) point of $\wt{U}_n$.  Let $\wt{\varphi}_n$ be the conformal transformation mapping $\wt{U}_n$ onto $\D$ such that $\wt{\varphi}_n(\wt{z}) = 0$ and $\wt{\varphi}_n'(\wt{z}) > 0$ and note that $\wt{\varphi}_n = e^{-2i\arg(z)} \varphi_n^{-1} \circ \phi^{-1}|_{\wt{U}_n}$.  Combing Lemma~\ref{lem:fan_symmetry} with the a.s.\ convergence of $(\wt{x}_n)$ above,  we obtain that the sequence $(\wt{\varphi}_n(\phi(y_n)))$ converges to some point on $\partial \D$.  But since $\wt{\varphi}_n(\phi(y_n)) = e^{-2i\arg(z)} \wt{y}_n$,  it follows that $(\wt{y}_n)$ converges a.s.

\noindent{\it Step 3. Identifying the boundary conditions of $\Fh$.} 
We define $x = x(U) = \ff\nv(\wt{x}), y = y(U) = \ff\nv(\wt{y})$, both of which are defined as prime ends in $\del U$. Suppose first that $\wt{x} \neq \wt{y}$ so that $x(U)$ and $y(U)$ are distinct prime ends in $\del U$.
In this case, by considering $\wt\Fh_n(0)$, we see that the correct branch of $\arg(\wt f_n)'$ is eventually the same for all sufficiently large $n$, in the sense that the limit $\arg(\wt f_n)'(z)$ exists for all $z \in \D$ and does not depend on $z$ (if this were not the case, $\wt\Fh_n$ would not converge locally uniformly).
Note that $\arg(\wt f_n)'\pmod{2\pi\chi}$ is actually determined only by $\wt y_n$ (the point mapped to $\infty$), and not on $\wt x_n$ or indeed on the implicit scale factor present in the choice of $\wt f_n'$.
Then we can identify $\wt{\Fh}$ (using the explicit descriptions of $\wt{\Fh}_n$) as the harmonic function with boundary conditions given by $\la - \tht\chi - \chi\arg\wt{f}'$ on the clockwise arc from $\wt{x}$ to $\wt{y}$, and by $-\la - \tht\chi - \chi\arg\wt{f}'$ on the counterclockwise arc, where $\wt f \colon \D \to \h$ is a conformal map with $\wt f(\wt{x}) = 0, \wt f(\wt{y}) = \infty$ (compare to Figure~\ref{fig:un_boundary_conditions}).
The corresponding boundary conditions for $\Fh$ are given by $\la - \tht\chi - \chi \arg f'$ on the counterclockwise arc from $x$ to $y$, and by $-\la - \tht\chi - \chi \arg f'$ on the clockwise arc, where $f = \wt f \circ \ff$ is a conformal map sending $x$ to $0$ and $y$ to $\infty$. The correct branch of $\arg f'$ here is determined by the branches of $\arg f_n'$ (via $\arg(\wt f_n)'$), and is measurable with respect to the flow lines $\eta_1, \dots, \eta_n$, but not with respect to $U$ on its own.

It remains then to prove the result when $\wt{x} = \wt{y}$. The issue that can arise here is if, as $\wt{x}_n$ and $\wt y_n$ converge to $\wt{x}$, the side of $\wt{x}_n$ that $\wt{y}_n$ lies on changes at arbitrarily high values of $n$.
Away from any small region around $\wt x$, for sufficiently large $n$ the boundary conditions of $\wh\Fh_n$ (up to a multiple of $2\pi\chi$) are given by $\la - \phi_{j(n)}\chi - \chi\arg(\wt f_n)'$ (resp.\ $-\la - \phi_{i(n)}\chi - \chi\arg(\wt f_n)'$) if $\wt y_n$ is to the left (resp.\ right) of $\wt x_n$.
Since $\wt\Fh_n$ converges locally uniformly as $n \to \infty$, we would like to conclude that $\wt y_n$ must eventually always lie on the same side of $\wt x_n$. However, since the boundary values of $\wh\Fh_n$ are given by the above only up to a multiple of $2\pi\chi$ which may depend on $n$, it is actually not possible to do so at this stage.
In particular, suppose $\wt y_n$ lies to the left of $\wt x_n$ and $\wt y_m$ lies to the right of $\wt x_m$ (for large enough $n$ that these points are very close to $\wt x$ and that $\phi_{i(n)}, \phi_{j(n)}, \phi_{i(m)}, \phi_{j(m)}$ are close to $\tht$). Then, at any point of $\del \D$ far away from $\wt x$, the difference $\wh\Fh_n - \wh\Fh_m$ is given by $2\la  -(\phi_{j(n)} - \phi_{i(m)}) \chi - \chi(\arg(\wt f_n)' - \arg(\wt f_m)')$.
The final term here is of the form $2\pi \DD K \chi$ where  $\DD K = -(\arg(\wt f_n)' - \arg(\wt f_m)') \in \Z$, and in particular may not be $0$. If $2\la + 2\pi\DD K \chi$ is bounded away from $0$ for all $\DD K \in \Z$, then for large enough $n, m$ (so that the difference in angles is small) we can obtain a contradiction to the locally uniform convergence of $\wt\Fh_n$ using similar methods to Step 2.
However, for $\kk = 2, 3$ and more generally $\kk = 4 - 2/\DD K$ for $\DD K \in \N$, we have $2\la + 2\pi \DD K \chi = 0$, and that we cannot rule out at the present that the side of $\wt x_n$ that $\wt y_n$ lies on may change at arbitrarily high values of $n$ (however, in Lemma~\ref{lem:component_jordan_domain} we will actually show that a.s.\ $\wt x \neq \wt y$, so this whole case never actually occurs).

We will now identify the boundary conditions of $\wt\Fh$ and $\Fh$ in the case $\wt x = \wt y$. Fix $\wt{v} \in \del\D$ with $\wt{v} \neq 
\wt x$ and suppose in the following that $n$ is large enough that $\wt x_n, \wt y_n \in B(\wt x, \dist(\wt x, \wt v)/2)$.
Define $\wt g \colon \D \to \h$ to be a conformal map sending $\wt v$ to $0$ and $\wt y_n$ to $\infty$. Up to a choice of branch, $\arg(\wt g_n)'$ does not depend on $\wt v$ and in fact, $\arg(\wt g_n)' = \arg(\wt f_n)'$ on $\del D \sm \{\wt y_n\}$.
Suppose now that $\wt y_n$ lies to the left of $\wt x_n$ for all sufficiently large $n$. Then from the locally uniform convergence of $\wt\Fh_n$, the branch of $\arg(\wt f_n)'$ must eventually stabilize in the sense that $\lim_{n \to \infty} \arg(\wt f_n)'(z)$ exists for all $z \in \D$ and does not depend on $z$.
If the branch of $\arg(\wt g_n)'$ is chosen such that $\arg(\wt g_n)' = \arg(\wt f_n)'$ on $\del \D \sm \{\wt y_n\}$, then we have that $\arg(\wt g_n)' \to \arg(\wt g)'$ where $\wt g \colon \D \to \h$ is a conformal map sending $\wt v$ to $0$ and $\wt y$ to $\infty$.
Therefore, the boundary conditions of $\wt\Fh$ are given by $-\la - \tht\chi - \chi\arg(\wt g)'$ on all of $\del \D \sm \{\wt y\}$ where the branch of the argument is chosen depending on the branches of $\arg(\wt g_n)'$. It follows that the boundary conditions of $\Fh$ on $\del U$ are given by $-\la - \tht \chi - \chi\arg g'$, where $g = \wt g \circ \ff$. Note that these boundary conditions are actually unchanged if we replace $g$ by any conformal map $g_1 \colon U \to \h$ which sends $y$ to $\infty$.
Similarly, in the case where $\wt y_n$ eventually always lies to the right of $\wt x_n$, we can conclude that the boundary conditions of $\Fh$ are given by $\la - \tht \chi - \chi\arg g'$, where $g$ is defined as before.

It remains finally to identify $\wt\Fh$ and $\Fh$ in the case that the cyclic order of $\wt x_n, \wt y_n, \wt v$ changes for arbitrarily high values of $n$, where $\wt v$ is defined as above. Let $\wt g_n, \wt g$ and $g$ be as above, and note that if we restrict to the subsequence $(R_n)$ of integers where $\wt y_{R_n}$ lies to the right of $\wt x_{R_n}$ for all $n$, then, by the locally uniform convergence of $\wt\Fh_n$, the branch of $\arg(\wt f_{R_n})'$ must again stabilize, meaning we have $\arg(\wt g_{R_n})' \to \arg(\wt g)'$ for the correct choices of branches of the argument as in the previous case.
We conclude again that the boundary conditions of $\wt\Fh$ are given by $\la - \tht \chi - \chi\arg(\wt g)'$ on $\del \D \sm \{\wt y\}$, and hence that the boundary conditions of $\Fh$ are given by $\la - \tht \chi - \chi\arg g'$ on $\del U$.
Note that we could have instead determined the boundary conditions of $\wt\Fh$ and $\Fh$ using the subsequence $(L_n)$ where $\wt y_{L_n}$ lies to the left of $\wt x_{L_n}$. The resulting boundary conditions for $\Fh$ are then given by $-\la - \tht \chi - \chi\arg_1 g'$, where the branch of $\arg_1 g'$ is chosen in such a way that these boundary conditions \emph{are the same} as those given by $-\la - \tht \chi - \chi\arg_1 g'$ on $\del U$.
We emphasize that the branches $\arg g'$ and $\arg_1 g'$ must differ by $2\la$ since $\Fh_{R_n}$ and $\Fh_{L_n}$ must converge locally uniformly to the same harmonic function, ensuring that we do not have an issue in this pathological case.
Therefore we have identified the boundary conditions of $\Fh$ in every case, completing the proof of the lemma.
\end{proof}

We emphasize that in Lemma~\ref{lem:form_of_the_boundary} we have not shown that $x(U),y(U) \in \partial U$ are distinct so that $\partial U$ has two non-degenerate boundary arcs. The first step in showing this is the following lemma.
We will make use of the following notation.  Let $\gamma \colon [0,1] \to \ol{\h}$ be a simple curve with $\gamma(0) \in (-\infty,0)$, $\gamma(1) \in (0,\infty)$, and $\gamma(t) \in \h$ for each $t \in (0,1)$.  Let $\fan(\theta_1,\theta_2;\gamma)$ be the closure of the union of the flow lines of $h$ from $0$ to $\infty$ with rational angles in $[\theta_1,\theta_2]$ and stopped upon first hitting $\gamma$.  
This is a local set by Lemma~\ref{lem:union_of_local_sets}.
Let $\Fh$ be a distribution on $\h$ which is harmonic in $\h \setminus \fan(\theta_1,\theta_2;\gamma)$ so that we can write $h = h^0 + \Fh$ where given $\CF_\gamma = \sigma(\Fh, \fan(\theta_1,\theta_2;\gamma))$ we have that $h^0$ is a GFF with zero boundary conditions in $\h \setminus \fan(\theta_1,\theta_2;\gamma)$.  The next lemma describes the boundary conditions of $h$ when restricted to the bounded connected components of $\h \setminus (\fan(\theta_1,\theta_2;\gamma) \cup \gamma)$ whose boundaries intersect $\gamma$.

\begin{lemma}
\label{lem:flow_line_representation}
Suppose that we have the setup described just above.  Suppose that $V$ is a bounded connected component of $\h \setminus (\fan(\theta_1,\theta_2;\gamma) \cup \gamma)$ with $\partial V \cap \fan(\theta_1,\theta_2;\gamma) \neq \emptyset$.  We will view $\partial V$ as a collection of prime ends in $V$.  Then there exists $z(V) \in \partial V$ and $\theta(V) \in \R$ so that the boundary conditions for the conditional law of $h$ given $\CF_\gamma$ on $\partial V \cap \fan(\theta_1,\theta_2;\gamma)$ are given by the right (resp.\ left) side of a flow line of angle $\theta(V)$ starting from $z(V)$ and going in the clockwise (resp.\ counterclockwise) direction until hitting $\gamma$.
\end{lemma}

In this case, by flow line boundary conditions, we mean that the boundary conditions for $h$ given $\CF_\gg$ are $-\la - \tht \chi - \chi\arg \psi'$ on the arc of $\del G$ from $z(V)$ to $\wh z$, and $\la - \tht \chi - \chi\arg \psi'$ on the arc of $\del G$ from $z(V)$ to $\wh y$, where $\wh y, \wh z$ are the prime ends at $y$ and $z$ respectively obtained by approaching $y$ or $z$ from the unbounded component of $\h \sm \gg$, and where the branch of $\arg \psi'$ is chosen such that $\lim_{z \to \infty} \arg \psi'(z) = 0$. Here,  $G$ is the unbounded connected component of $\h \setminus \fan(\theta_1,\theta_2)$ and $y$ (resp.  $z$) is the leftmost (resp.  rightmost) point on $\gamma \cap \partial V$ lying to the left (resp.  right) of $\gamma(t)$ (see also Figures~\ref{fig:gamma_Vn} and~\ref{fig:fntgamma}).

\begin{proof}

\noindent{\it Step 1.  Overview and setup.} Let $t \in (0,1) \cap \Q$ be fixed and note that $\gg(t) \notin \F(\tht_1, \tht_2 ;\gg)$ a.s. We work on this event and let $V$ be the bounded connected component of $\h \sm \fan(\theta_1,\theta_2;\gamma)$ which has $\gg(t)$ on its boundary. We will prove the claim of the lemma for this component $V$.
Let $(\phi_n)$ be the fixed enumeration of $(\Q \cap [\tht_1, \tht_2]) \cup \{\tht_1, \tht_2\}$ from Lemma~\ref{lem:form_of_the_boundary} and for each $n \in \N,  1 \leq j \leq n$,  we let $\eta_j$ be the flow line of $h$ from $0$ to $\infty$ with angle $\phi_j$,  stopped at the first time that it hits $\gamma$.  We then set $\fan_n(\theta_1,\theta_2;\gamma) = \cup_{j=1}^n \eta_j$.  Let also $\psi_n$ be the conformal transformation mapping the unbounded connected component $G_n$ of $\h \setminus \fan_n(\theta_1,\theta_2;\gamma)$ onto $\h$ so that $\psi_n(\gamma(t)) = i$ and $\psi_n(\infty) = \infty$.  
Note that $\fan_n(\theta_1,\theta_2;\gamma)$ converges in the Hausdorff sense to $\fan(\theta_1,\theta_2;\gamma)$ as $n \to \infty$,  and $G_n$ converges in the Caratheodory kernel sense to $G$,  where $G$ is the unbounded connected component of $\h \setminus \fan(\theta_1,\theta_2;\gamma)$.  Also,  for each $n \in \N$,  we have that $\fan_n(\theta_1,\theta_2;\gamma)$ is a local set for $h$ and so $h = h_n^0 + \Fh_n$,  where $h_n^0$ is a zero boundary $\text{GFF}$ on $\h \setminus \fan_n(\theta_1,\theta_2;\gamma)$ and $\Fh_n$ is harmonic on $\h \setminus \fan_n(\theta_1,\theta_2;\gamma)$. 
Similarly,  Lemma~\ref{lem:union_of_local_sets} implies that $\fan(\theta_1,\theta_2;\gamma)$ is a local set for $h$ and so we have that $h = h^0 + \Fh$,  where $h^0$ is a zero boundary $\text{GFF}$ on $\h \setminus \fan(\theta_1,\theta_2;\gamma)$ and $\Fh$ is harmonic on $\h \setminus \fan(\theta_1,\theta_2;\gamma)$.  Since $(\fan_n(\theta_1,\theta_2;\gamma))_n$ is an increasing sequence of local sets,  it follows from \cite[Proposition~6.5]{ms2016imag1} that $\Fh_n \to \Fh$ locally uniformly as $n \to \infty$ a.s. Moreover,  we have that $\psi_n \to \psi$ locally uniformly as $n \to \infty$ a.s.,  where $\psi$ is the conformal transformation mapping $G$ onto $\h$ such that $\psi(\gamma(t)) = i$ and $\psi(\infty) = \infty$.  We set $\wt{h}_n = h \circ \psi_n^{-1} - \chi \arg(\psi_n^{-1})',  \wt{\Fh}_n = \Fh_n \circ \psi_n^{-1} - \chi \arg(\psi_n^{-1})',  \wt{h} = h \circ \psi^{-1} - \chi \arg(\psi^{-1})'$,  and $\wt{\Fh} = \Fh \circ \psi^{-1} - \chi \arg(\psi^{-1})'$.  

Let $V_n$ be the bounded connected component of $\h \setminus (\fan_n(\theta_1,\theta_2;\gamma) \cup \gamma)$ whose boundary contains $\gamma(t)$.  
Then,  there exist unique $1 \leq i(n),  j(n)\leq n$ so that 
$\phi_{i(n)}$ is the maximal angle among $(\phi_k)_{k \leq n}$ for which $\eta_{i(n)}$ first hits $\gg$ to the right of $\gg(t)$, and $\phi_{j(n)}$ is the minimal angle with $\eta_{j(n)}$ hitting $\gg$ to the left of $\gg(t)$ (see Figure~\ref{fig:gamma_Vn}).
Note that since $\eta_{\tht_1} := \R_+$ and $\eta_{\tht_2} = \R_-$, we have $\eta_1 = [0, \gg(1)]$ and $\eta_2 = [\gg(0), 0]$. In particular, $i(n), j(n)$ are well-defined for $n \geq 2$ and, as in Lemma~\ref{lem:form_of_the_boundary}, a.s. for large enough $n$ we will have that $\eta_{i(n)}, \eta_{j(n)}$ are actual flow lines (and not segments of $\R$).
Then $\partial V_n$ consists of the following three arcs.  The first arc is a clockwise arc of $\eta_{j(n)}$ from $x_n$ to $y_n$ (where $x_n$ is the last intersection point between $\eta_{i(n)}$ and $\eta_{j(n)}$ before hitting $\gamma$ for the first time and $y_n$ is the first point of $\gamma$ that $\eta_{j(n)}$ hits),  the second arc is a counterclockwise arc of $\eta_{i(n)}$ from $x_n$ to $z_n$ (where $z_n$ is the first point of $\gamma$ that $\eta_{i(n)}$ hits),  and the third arc is the clockwise arc of $\gamma$ connecting $y_n$ to $z_n$.  
Arguing as in Lemma~\ref{lem:form_of_the_boundary} we can show that $(\phi_{i(n)})$ and $(\phi_{j(n)})$ converge to a common limit $\tht$, which we refer to as $\tht(V) \equiv \tht(\gg, t)$. Furthermore, by definition $y_{n+1}$ lies on the arc of $\gg$ from $y_n$ to $\gg(t)$ (often we have $y_n = y_{n+1}$) so $y := \lim_n  y_n$ must exist and lie on $\gg$. Analogously we can define $z = \lim_n z_n$.
In Step 2,  we will show that the sequences $(\psi_n(x_n)),  (\psi_n(y_n))$,  and $(\psi_n(z_n))$ converge to points $\wt{x}, \wt{y},$ and $\wt{z}$ on $\R$ respectively,  using the local uniform convergence of $\Fh_n$ to $\Fh$.  Then,  using this,  in Step 3 we will identify the boundary conditions of $\Fh$ and hence prove the lemma.

\begin{figure}[t]
\includegraphics[scale=0.8]{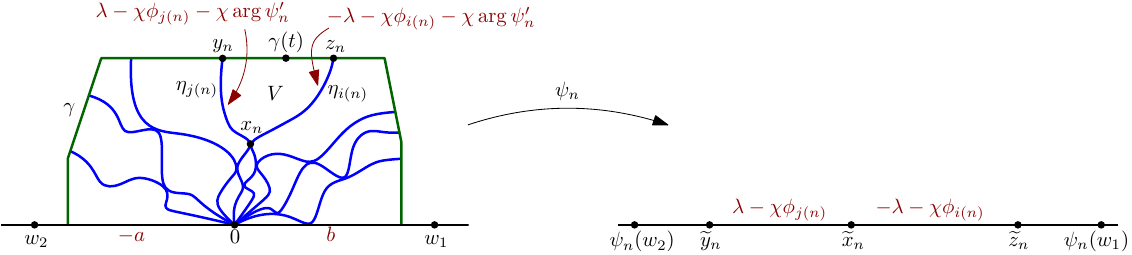}
\caption{On the left we depict the setup described in Lemma~\ref{lem:flow_line_representation} including the curve $\gg$ and the set $\F(\tht_1, \tht_2; \gg)$, along with the flow lines $\eta_{i(n)}, \eta_{j(n)}$ and the points $x_n, y_n$ and $z_n$. We show the boundary conditions of $\Fh_n$ on the left side of $\eta_{i(n)}$ and the right side of $\eta_{j(n)}$. The conformal map $\psi_n$ maps $G_n$, the unbounded connected component of $\h \sm \F_n(\tht_1, \tht_2; \gg)$, onto $\h$. On the right we show the boundary conditions of $\wt{\Fh}_n$ on $[\wt{y}_n, \wt{x}_n]$ and $[\wt{x}_n, \wt{z}_n]$. We use a harmonic measure argument to show that $(\wt{y}_n)$ and $(\wt{z}_n$) converge, and use the local uniform convergence of $(\wt{\Fh}_n)$ to show the convergence of $(\wt{x}_n)$ and identify the boundary conditions of $\wt{\Fh}$.}
\label{fig:gamma_Vn} 
\end{figure}

\noindent{\it Step 2.  Convergence of $\psi_n(x_n),  \psi_n(y_n)$,  and $\psi_n(z_n)$.} Set $\wt{x}_n = \psi_n(x_n),  \wt{y}_n = \psi_n(y_n)$,  and $\wt{z}_n = \psi_n(z_n)$,  for each $n \in \N$.  We will show that a.s.\ there exist $\wt{x},\wt{y},\wt{z} \in \R$ so that $\wt{x}_n \to \wt{x}$,  $\wt{y}_n \to \wt{y}$ and $\wt{z}_n \to \wt{z}$ as $n \to \infty$.  First,  we will show the a.s.\ convergence of $(\wt{y}_n)$ and $(\wt{z}_n)$.  Let $p_n$ (resp.\ $q_n$) be the harmonic measure in $G_n$ as seen from $\gamma(t)$ of the clockwise (resp.\ counterclockwise) arc of $\partial G_n$ from $y_n$ (resp.\ $z_n$) to $\infty$.  Then it suffices to show that both $(p_n)$ and $(q_n)$ converge as $n \to \infty$.  Fix $\epsilon>0$. 
Then,  a.s.\ there exists $n_0 \in \N$ such that for all $n \geq n_0$,  we have that the probability that a Brownian motion $B$ starting from $\gamma(t)$ exits $G_n$ on the arc of $\del G_n$ from $y_n$ to $y_{n_0}$ (i.e.\ $B$ exits $G_n$ either by hitting the left side of $\eta_{j(n)}$, the right side of $\eta_{j(n_0)}$, or either side of a flow line $\eta_k$ where $k \leq n$ and $\phi_{j(n)} < \phi_k < \phi_{j(n_0)}$) is at most $\epsilon$. 
For $n \geq m \geq n_0$, we claim that the probability that $B$ exits $G_n$ on the clockwise arc of $\partial G_n$ from $y_{n_0}$ to $\infty$ and does not exit $G_m$ on the clockwise arc of $\partial G_m$ from $y_m$ to $\infty$ is at most $\epsilon$ if we take $n$ and $m$ sufficiently large.
Indeed,  this follows because the Hausdorff distance (with respect to the spherical metric) between $\F_n(\tht_1,\tht_2;\gg)$ and $\F_m(\tht_1, \tht_2;\gg)$ tends to $0$ as $n , m \to \infty$ and so if the Brownian motion exits $G_n$ on the clockwise arc of $\partial G_n$ from $y_{n_0}$ to $\infty$, but does not exit $\del G_m$ on the clockwise arc from $y_m$ to $\infty$ then it will have to travel macroscopic distance without hitting the clockwise arc of $\partial G_m$ from $y_m$ to $\infty$.  But the latter probability can be made arbitrarily small by the Beurling estimate.  It follows that $p_n \leq p_m + 2\epsilon$ for all $n \geq m \geq n_0$,  and so $\limsup_{n \to \infty} p_n \leq 2\epsilon + \liminf_{n \to \infty} p_n$.  Since $\epsilon>0$ was arbitrary,  we obtain that $(p_n)$ converges a.s.\ and a similar argument shows the a.s.\ convergence of $(q_n)$.  This shows the a.s.\ convergence of $(\wt{y}_n)$ and $(\wt{z}_n)$,  and let $\wt{y},\wt{z}$ be their limits respectively.

Let $\wh{y}$ be the unique prime end of $\del G$ corresponding to approaching the point $y$ from the unbounded component of $\h\sm\gg$. We now show that $\wh y = \psi\nv(\wt y)$ by an argument similar to the above. Let $p$ be the probability that a Brownian motion $B$ started from $\gg(t)$ exits $G$ on the arc of $\del G$ from $\wh{y}$ to $-\infty$. Fix $\eps > 0$ and $n > n_0$ with $n_0$ large enough that the probability that $B$ exits $G$ on the arc of $\del G$ from $\wh{y}$ to (the prime end corresponding to) $y_{n_0}$ is at most $\eps$.
Then the probability that $B$ exits $G$ on the arc from $y_{n_0}$ to $-\infty$ and does not exit $G_n$ on the arc from $y_n$ to $-\infty$ is at most $\eps$ for $n$ large enough by the Beurling estimate, as before, since $\F_n(\tht_1,\tht_2;\gg)\to \F(\tht_1,\tht_2;\gg)$ in the Hausdorff sense.
Similarly, the probability that $B$ exits $G$ on the arc from $\wh y$ to $+\infty$ and that $B$ does not exit $G_n$ on the arc from $y_n$ to $+\infty$ is at most $\eps$ for large enough $n$.
Therefore, we conclude that $\lim_{n\to\infty} p_n - 2\eps \leq p \leq \lim_{n\to\infty} p_n + 2\eps$, which holds for all $\eps > 0$. Therefore $p = \lim_{n\to\infty} p_n$ and $\psi\nv(\wt{y}) = \wh{y}$. Similarly we can show that $\psi\nv(\wt{z}) = \wh{z}$, where $\wh{z}$ is the prime end corresponding to $z$.

Next,  we show the convergence of $(\wt{x}_n)$.  We let $w_1 = \gg(1) + 1$ and $w_2 = \gg(0) - 1$.
Then,  a.s.\ there exists $p>0$ so that for all $n \in \N$,  the harmonic measure of each of the intervals $(-\infty, \psi_n(w_2)],  (\psi_n(w_2),\psi_n(x_n)],  (\psi_n(x_n),\psi_n(w_1)]$,  and $(\psi_n(w_1),\infty)$ in $\h$ as seen from $i$ is at least $p$.  Thus there exists $M>0$ depending only on $p$ so that $-M \leq \psi_n(w_2) < \psi_n(w_1) \leq M$ for all $n$.  Suppose that $(\wt{x}_n)$ is not convergent.  
Then there exist subsequences $(\wt{x}_{\mu_n}),  (\wt{x}_{\lambda_n})$ and points $x^\mu, x^\la \in \R$,  and $\epsilon_0 > 0$ so that $\wt{x}_{\mu_n} \to x^\mu,  \wt{x}_{\lambda_n} \to x^\la$ as $n \to \infty$,  and $|\wt{x}_{\mu_n}-\wt{x}_{\lambda_n}| \geq \epsilon_0$,  for all $n$.  We can also assume that $\wt{x}_{\mu_n} < \wt{x}_{\lambda_n}$ and $[\wt{x}_n -\epsilon_0 ,  \wt{x}_n + \epsilon_0] \subseteq [\psi_n(w_2) ,  \psi_n(w_1)]$ for all $n$. 
Note that $\wt{y}_n \leq \wt{x}_n \leq \wt{z}_n$ for all $n$,  and so since $\wt{y}_n \to \wt{y}$ and $\wt{z}_n \to \wt{z}$ as $n \to \infty$,  a.s.\ there exists $n_0 \in \N$ so that $[a_n ,  b_n] \subseteq [\wt{x}_{\mu_n},\wt{z}_{\mu_n}] \cap [\wt{y}_{\lambda_n},\wt{x}_{\lambda_n}]$ for all $n \geq n_0$,  where $a_n = \wt{x}_{\mu_n}+ \epsilon_0/3$ and $b_n = \wt{x}_{\mu_n}+2\epsilon_0/3$.  Then the boundary values of $\wt{\Fh}_{\mu_n}$ (resp.\ $\wt{\Fh}_{\lambda_n}$) on $[a_n,b_n]$ are given by $-\lambda -  \phi_{i(\mu_n)} \chi$ (resp.\  $\lambda- \phi_{j(\lambda_n)}\chi$). 
Since $\phi_{i(n)}$ and $\phi_{j(n)}$ both converge to $\theta$, by possibly taking $n_0$ to be larger we can assume that there exists $\delta_0>0$ depending only on $\lambda$ and $\chi$ such that the absolute value of the difference of the boundary values of $\wt{\Fh}_{\mu_n}$ and $\wt{\Fh}_{\lambda_n}$ on $[a_n,b_n]$ is at least $\delta_0$. 
Fix $\delta \in (0,\min\{\delta_0,\epsilon_0 / 200\})$ and set $c_n = (a_n+b_n)/2+i \delta$.  Then $c_n \in K_{\delta}$ for all $n \geq n_0$,  where $K_{\delta} = \{z \in \h : \im(z) \geq \delta,  \re(z) \in [-M,M]\}$.  Since $K_{\delta} \subseteq \h$ is compact,  we have that $\sup_{z \in K_{\delta}} |\wt{\Fh}_{\mu_n}(z) - \wt{\Fh}_{\lambda_n}(z)| \to 0$ as $n \to \infty$.  Also,  the Beurling estimate implies that a Brownian motion starting from $c_n$ exits $\h$ on $\R \setminus [a_n,b_n]$ is at most $\lesssim (\delta / \epsilon_0)^{1/2}$,  where the implicit constant is universal.  Moreover,  we have that $\|\wt{\Fh}_n\|_{\infty} \lesssim 1$,  for all $n$,  where the implicit constant depends only on $a,b,\lambda,\chi,\theta_1$,  and $\theta_2$.  Combining,  we obtain that we can choose $\delta>0$ sufficiently small such that $\sup_{z \in K_{\delta}} |\wt{\Fh}_{\mu_n}(z)-\wt{\Fh}_{\lambda_n}(z)| \geq \delta_0 / 2$ for all $n \geq n_0$,  but this contradicts the local uniform convergence of $\wt{\Fh}_n$ to $\wt{\Fh}$.  Therefore,  it is a.s.\ the case that there exists $\wt{x} \in \R$ such that $\wt{x}_n \to \wt{x}$ as $n \to \infty$.
We define $z(V)$ to be the prime end of $\del G$, $\psi\nv(\wt x)$.

\noindent{\it Step 3.  Conclusion of the proof.}
Note that the boundary values of $\wt{\Fh}_n$ on $[\wt{y}_n,\wt{x}_n]$ (resp.\ $[\wt{x}_n,\wt{z}_n]$) are given by $\lambda -  \phi_{j(n)}\chi$ (resp.\ $-\lambda- \phi_{i(n)}\chi$) and so since $\wt{\Fh}_n \to \wt{\Fh}$ locally uniformly and $\wt{x}_n \to \wt{x},  \wt{y}_n \to \wt{y}$,  and $\wt{z}_n \to \wt{z}$ as $n \to \infty$,  we obtain that the boundary values of $\wt{\Fh}$ on $[\wt{y},\wt{x}]$ (resp.\ $[\wt{x},\wt{z}]$) are given by $\lambda-  \theta\chi$ (resp.\ $-\lambda- \theta\chi$).  
Note that either (but not both) of the intervals $[\wt{y}, \wt{x}]$ and $[\wt{x}, \wt{z}]$ may be a single point.
Since $\wt\Fh = \Fh \circ \psi\nv - \chi\arg(\psi\nv)'$, we can identify the boundary conditions of $\wt\Fh$ on the (possibly degenerate) arcs of $\del G$ (viewed as a set of prime ends) from $\wh{y}$ to $z(V)$, and from $z(V)$ to $\wh z$ and see that they are exactly the boundary conditions of the right and left sides of a flow line of angle $\tht$.

We have now shown that for a fixed curve $\gg$ and $t \in [0,1]$, on the almost sure event that $\gg(t) \notin \F(\tht_1, \tht_2; \gg)$, that the boundary conditions of $\Fh$ on the arc of $\del G$ from $\wh y$ to $z(V)$ (resp.\ from $z(V)$ to $\wh z$) are given by those of the right (resp.\ left) side of a flow line of angle $\tht(V)$.
Since for every bounded connected component $V$ of $\h\sm(\F(\tht_1, \tht_2; \gg) \cup \gg)$ there exists $t \in [0,1] \cap \Q$ with $\gg(t) \in \del V \cap \gg$, the lemma statement follows.
Finally we remark that if $z(V)$ is in the bounded component of $\h\sm\gg$, then $z(V), \wh y$ and $\wh z$ are distinct prime ends on $\del G$.
\end{proof}

Fix $t \in (0,1)$ and suppose that we have the setup of the statement of Lemma~\ref{lem:flow_line_representation}. We let $\theta$ be the angle associated with the connected component $V$ containing $\gamma(t)$ on its boundary.  Let 
\[ \fan(t;\gamma) = \cap_{\substack{\phi_1 < \theta < \phi_2\\ \phi_1,\phi_2 \in \Q}} \fan(\phi_1, \phi_2;\gamma). \]

In Lemma~\ref{lem:locality}, we will show that $\fan(t;\gamma)$ is a local set for $h$,  and then in Lemma~\ref{lem:boundary_conditions},  we will identify the boundary conditions of the field $h|_{\h \setminus \fan(t;\gamma)}$.

\begin{lemma}
\label{lem:locality}
Suppose that we have the setup described just above and $t \in (0,1)$ is fixed.  Then $\fan(t;\gamma)$ is a local set for $h$.
\end{lemma}
\begin{proof}
To prove the claim of the lemma , we will use \cite[Lemma~3.9]{schramm2013contour}.   Suppose that $(\phi_n)$ is the enumeration of $(\Q \cap [\theta_1,\theta_2]) \cap \{\tht_1, \tht_2\}$ as before.  For each $j$, we let $\eta_j$ be the flow line of $h$ from $0$ to $\infty$ in $\h$ with angle $\phi_j$, stopped upon hitting $\gamma$.  For each $n$, we let $\fan_n(t; \gamma)$ be given by $\fan(\phi_{i(n)},\phi_{j(n)};\gamma)$ where $\phi_{i(n)}, \phi_{j(n)}$ are, as in Lemma~\ref{lem:flow_line_representation}, the angles of the two flow lines among those with angles in $\phi_1,\ldots,\phi_n$ which make up the right and left boundaries respectively of the bounded connected component of  $\h \setminus (\gamma \cup \eta_1 \cup \cdots \cup \eta_n)$ whose boundary contains $\gamma(t)$.
In other words, the boundary of the bounded connected component of $\h \setminus (\gamma \cup \eta_1 \cup \cdots \cup \eta_n)$ with $\gamma(t)$ on its boundary consists of a segment of $\gamma$, $\eta_{i(n)}$, and $\eta_{j(n)}$.
See Figure~\ref{fig:fntgamma} for a partial representation of $\F_n(t;\gg)$.
Note that $\fan(t,\gamma) = \cap_n \fan_n(t;\gamma)$.  Since $(\fan_n(t; \gamma))_n$ is a countable and decreasing sequence of closed sets which are determined by the field,  by combining with Lemma~\ref{lem:intersection_of_local_sets}, we obtain that it suffices to show that $\fan_n(t;\gamma)$ is a local set for the field for each $n$.  Fix an open set $U \subseteq \h$.  We need to show that the event that $\fan_n(t;\gamma) \cap U \neq \emptyset$ is conditionally independent of the projection of the field onto the functions which are supported in $U$ given the projection onto those functions which are harmonic in $U$.  Let $V_n$ be the bounded connected component of $\h \setminus (\gamma \cup \eta_1 \cup \cdots \cup \eta_n)$ with $\gamma(t)$ on its boundary.  Then there exist $1 \leq i(n) , j(n) \leq n$ so that $\partial V_n$ can be written as the union of a segment of $\gamma$ as well as segments of $\eta_{i(n)}$ and $\eta_{j(n)}$.

We claim that $\eta_{i(n)} \cup \eta_{j(n)}$ is local for $h$.  Upon proving this, it will follow that $\fan_n(t;\gamma)$ is local for $h$ since $\fan_n(t;\gamma)$ is given by the closure of the union of the flow lines of the conditioned $\text{GFF}$ given $\eta_{i(n)} \cup \eta_{j(n)}$ with rational angles in $[\phi_{i(n)},\phi_{j(n)}]$,  stopped at the first time that they hit $\gamma$.  To see that $\eta_{i(n)} \cup \eta_{j(n)}$ is local, for each $j$, we let $\wt{\eta}_j$ be the flow line of $h$ from $0$ to $\infty$ in $\h$ stopped upon either first hitting $\gamma$ or $U$.  Let $\wt{V}_n$ be the bounded connected component of $\h \setminus (\gamma \cup \wt{\eta}_1 \cup \cdots \cup \wt{\eta}_n)$ with $\gamma(t)$ on its boundary.  
Then we have that $(\eta_{i(n)} \cup \eta_{j(n)}) \cap U \neq \emptyset$ if and only if $\partial \wt{V}_n$ does not consist of part of $\gamma$ and parts of $\eta_i,  \eta_j$ so that $\phi_i,\phi_j$ are consecutive numbers in the enumeration $(\phi_m)$.  Hence we obtain that the event $\{(\eta_{i(n)} \cup \eta_{j(n)}) \cap U \neq \emptyset\}$ is determined by $\wt{\eta}_1,\cdots,\wt{\eta}_n$.  But the latter collection of curves is determined by the projection of $h$ onto the functions which are harmonic in $U$.  It follows that $\{(\eta_{i(n)} \cup \eta_{j(n)}) \cap U \neq \emptyset\}$ is a.s.\ determined by the projection of $h$ onto the space of harmonic functions in $U$ and so \cite[Lemma~3.9]{schramm2013contour} implies that $\eta_{i(n)} \cup \eta_{j(n)}$ is local for $h$. This proves the claim hence the lemma.
\end{proof}

\begin{lemma}
\label{lem:boundary_conditions}
Let $\wh{\psi}$ be the unique conformal transformation from the unbounded connected component of $\h \setminus \fan(t;\gamma)$ onto $\h$ which fixes $\infty$ and sends $\gamma(t)$ to $i$.  On the event that $z(V)$ is in the bounded connected component of $\h \setminus \gamma$, there exist $x_1 < x_2 < x_3 < x_4 < x_5$ so that the field $h \circ \wh{\psi}^{-1} - \chi \arg (\wh{\psi}^{-1})'$ has boundary conditions given by:
\begin{align*}
&-a \quad\text{in}\quad (-\infty,x_1],\quad
-\lambda - \theta \chi \quad\text{in}\quad (x_1,x_2],\quad
\lambda - \theta \chi \quad\text{in}\quad (x_2,x_3]\\
&-\lambda - \theta \chi \quad\text{in}\quad (x_3,x_4],\quad
\lambda - \theta \chi \quad\text{in}\quad (x_4,x_5],\quad
b \quad\text{in}\quad (x_5,\infty).
\end{align*}
\end{lemma}
\begin{proof}
Suppose that we have the setup of the proof of Lemma~\ref{lem:locality}.  For each $n \in \N$,  we let $\wh{G}_n$ be the unbounded connected component of $\h \setminus \fan_n(t;\gamma)$ and let $\wh{\psi}_n$ be the conformal transformation mapping $\wh{G}_n$ onto $\h$ such that $\wh{\psi}_n(\gamma(t)) = i$ and $\wh{\psi}_n(\infty) = \infty$.  Similarly,  we let $\wh{G}$ be the unbounded connected component of $\h \setminus \fan(t;\gamma)$ and let $\wh{\psi}$ be the conformal transformation mapping $\wh{G}$ onto $\h$ such that $\wh{\psi}(\gamma(t)) = i$ and $\wh{\psi}(\infty) = \infty$.  Then we have that $\wh{\psi}_n \to \wh{\psi}$ locally uniformly as $n \to \infty$ a.s. We let $\wh{h}_n^0$ be a zero boundary $\text{GFF}$ on $\wh{G}_n$ and let $\wh{\Fh}_n$ be a harmonic function on $\wh{G}_n$ such that $h|_{\wh{G}_n} = \wh{h}_n^0 + \wh{\Fh}_n$.  Also,  since $\fan(t;\gamma)$ is a local set for $h$ (Lemma~\ref{lem:locality}),  we can write $h|_{\wh{G}} = \wh{h}^0 + \wh{\Fh}$,  where $\wh{h}^0$ (resp.\ $\wh{\Fh}$) is a zero boundary $\text{GFF}$ (resp.\ harmonic function) in $\wh{G}$.
We define $\eta_{i(n)}, \eta_{j(n)}, z_n$ and $y_n$ as in Lemma~\ref{lem:flow_line_representation}, and as before let $y = \lim_n y_n, z = \lim_n x_z$ with $\wh y, \wh z$. We have assumed that the prime end $z(V) \in \del V$ is in the bounded connected component of $\h \sm\gg$. Let $u_n$ (resp.\ $v_n$) be the rightmost (resp.\ leftmost) point on $\R_+ \cap \eta_{i(n)}$ (resp.\ $\R_- \cap \eta_{j(n)}$). Let $x_1^n$ and $x_5^n$ be the images of $v_n$ and $u_n$ under $\wh\psi_n$, let $x_2^n, x_4^n$ be the images of (the unique prime ends corresponding to) $y_n,z_n$, let $y_2^n, y_4^n$ be the images of $\wh{y},\wh{z}$ and finally let $x_3^n$ be the image of $z(V)$. See Figure~\ref{fig:fntgamma} for a partial representation of this setup.

Let $J_n$ (resp.\ $I_n$) be the clockwise (resp.\ counterclockwise) arc of $\gamma$ from $y_n$ to $y$ (resp.\ from $z_n$ to $z$), and let $\wh{p}_n$ (resp.\ $\wh{q}_n$) be the probability that a Brownian motion starting from $\gamma(t)$ hits $I_n$ (resp.\ $J_n$) before exiting $\wh{G}_n$.  Since $y_n \to y$ and $z_n \to z$ as $n \to \infty$ a.s., we obtain that $\wh{p}_n \to 0$ and $\wh{q}_n \to 0$ as $n \to \infty$ a.s.
Let $p_n$ (resp.\ $q_n$) be the harmonic measure of the left (resp.\ right) side of $\eta_{j(n)}$ (resp.\ $\eta_{i(n)}$) in $\wh{G}_n$ as seen from $\gamma(t)$.   Next we claim that the sequences $(p_n)$ and $(q_n)$ converge a.s.  Indeed,  fix $\epsilon>0$.  Then,  there exists $n_0 \in \N$ such that $\wh{p}_n + \wh{q}_n < \epsilon$ for all $n \geq n_0$. 
By possibly taking $n_0$ to be larger,  we can assume that the Hausdorff distance with respect to the Euclidean metric between $\eta_{j(n)}$ and $\eta_{j(m)}$ (resp.\ $\eta_{i(n)}$ and $\eta_{i(m)}$) stopped upon first hitting $\gamma$ is at most $\epsilon$,  for all $n,m \geq n_0$. 
Moreover,  there a.s.\ exists $\delta > 0$ such that the distance between the clockwise arc of $\gamma$ connecting $z$ to $\R_+$ and the counterclockwise arc of $\gamma$ from $y$ to $\R_-$ is at least $\delta$.  Fix $m>n\geq n_0$.  Then,  the Beurling estimate implies that the probability that a Brownian motion starting from $\gamma(t)$ hits the right side of $\eta_{i(n)}$ and then exits $\wh{G}_m$ on the left side of $\eta_{j(m)}$ is at most $\lesssim (\epsilon/\delta)^{1/2}$,  where the implicit constant is universal.
Also,  if the Brownian motion exits $\wh{G}_m$ on the left side of $\eta_{j(m)}$ without hitting $I_n \cup J_n$ or the right side of $\eta_{i(n)}$,  we have that it exits $\wh{G}_n$ on the left side of $\eta_{j(n)}$.  It follows that $p_m \leq p_n + \epsilon + C(\epsilon/\delta)^{1/2}$ for all $m>n\geq n_0$,  where $C<\infty$ is a universal constant,  which implies that $\limsup_{m \to \infty} p_m \leq p_n + \epsilon + C(\epsilon / \delta)^{1/2}$ for all $n \leq n_0$ and so $\limsup_{m \to \infty} p_m \leq \liminf_{n \to \infty} p_n + \epsilon + C(\epsilon / \delta)^{1/2}$.
Since $\epsilon>0$ was arbitrary,  we obtain that there exists $p>0$ such that $p_n \to p$ as $n \to \infty$.  Similarly,  there exists $q>0$ such that $q_n \to q$ as $n \to \infty$.
Furthermore,  the harmonic measures of $\R_- \cap \partial \wh{G}_n,  \R_+ \cap \partial \wh{G}_n$,  and the clockwise (resp.\ counterclockwise) arc of $\partial V$ from $z(V)$ to $\wh{y}$ (resp.\ $\wh z$),  all converge as $n \to \infty$ a.s.,  since they are increasing functions in $n$.

\begin{figure}[t]
\includegraphics[scale=0.8]{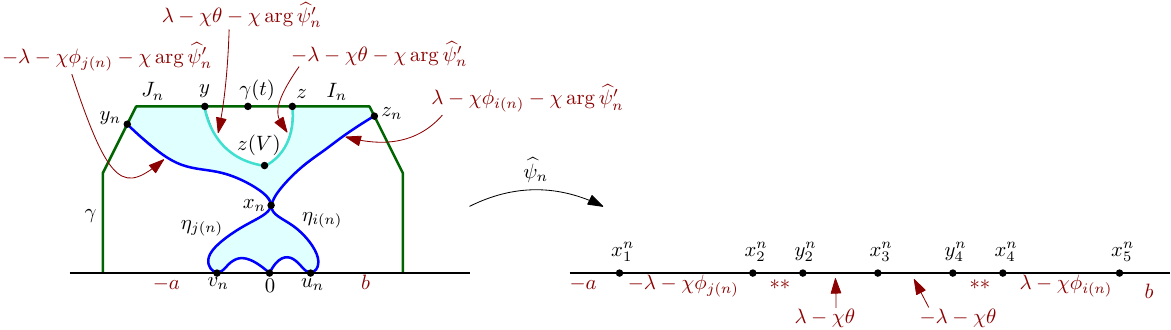}
\caption{\textbf{Left.} This is a schematic depiction of $\F_n(t;\gg)$. Here, $\eta_{i(n)}, \eta_{j(n)}, x_n$ and $y_n$ are as in Figure~\ref{fig:gamma_Vn}. We also include $y = \lim_n y_n, z = \lim_n z_n$ and the prime end $z(V) \in \del V$. $\F_n(t;\gg)$ is contained in the (closed) shaded region. We do not try to depict it since our information about its precise behavior is limited at this point. We emphasize that we do not know how complicated its boundary is in the region near $y,x$ and $z$ (the light blue line), and that we have not yet shown that this boundary is a simple curve. We include also some of the boundary conditions of $\wh\Fh_n$. \textbf{Right.} $\wh\psi_n$ is a conformal map from the unbounded component of $\h \sm \F_n(t;\gg)$ to $\h$ sending $\gg(t)$ to $i$ and fixing $\infty$. We identify the points $x_1^n, x_2^n, y_2^n, x_3^n, y_4^n, x_4^n$ and $x_5^n$ as the images of the seven points (more precisely, prime ends) $v_n, y_n, y, x, z, z_n$ and $u_n$ (defined in the proof of Lemma~\ref{lem:boundary_conditions}). We include the boundary conditions of $\wh{\Fh}_n \circ \wh{\psi}_n^{-1} - \chi \arg(\wh{\psi}_n^{-1})'$. The boundary conditions on intervals marked $**$ may change, but are uniformly bounded in $n$. The convergence of the points $x_j^n$ and $y_j^n$ allows us to determine the boundary conditions of the limiting harmonic function $\wh{\Fh} \circ \wh{\psi}^{-1} - \chi \arg(\wh{\psi}^{-1})'$.}
\label{fig:fntgamma}
\end{figure}

The results of the previous paragraph imply that there a.s.\ exist $x_1 < x_2 \leq x_3 \leq x_4 < x_5$ such that $x_j^n \to x_j$ as $n \to \infty$ for all $1 \leq j \leq 5$ and that $y_j^n \to x_j$ for $j = 2,4$.  We note that on the event that $z(V)$ is in the bounded connected component of $\h \setminus \gamma$,  we have that $x_2 < x_3 <x_4$. 
Also, by a harmonic measure argument similar to the one in Lemma~\ref{lem:flow_line_representation}, we can identify $\wh\psi\nv(x_2)$ (resp.\ $\wh\psi\nv(x_4)$) with the (unique) prime end of $\del \wh{G}$ corresponding to $y$ (resp.\ $z$).
Using Lemma \ref{lem:flow_line_representation} and \cite[Lemma 3.8]{ms2016imag1} we can explicitly describe the boundary conditions of the harmonic functions $\wh{\Fh}_n \circ \wh{\psi}_n^{-1} - \chi \arg(\wh{\psi}_n^{-1})'$ (except on the intervals $[x_2^n, y_2^n]$ and $[y_4^n, x_4^n]$, but by applying Lemma~\ref{lem:flow_line_representation} to a dense set of points on $\gg$, we know the boundary conditions on these intervals are bounded uniformly in $n$). These boundary conditions are depicted in Figure~\ref{fig:fntgamma}.  Finally, since $x_j^n \to x_j$ for each $j$ and $y_j^n \to x_j$ for $j = 2,4$, it follows that these harmonic functions converge locally uniformly as $n \to \infty$
to the harmonic function on $\h$ with boundary conditions given by $-a$ in $(-\infty,x_1]$,  $-\lambda- \theta \chi$ in $(x_1,x_2]$,  $\lambda- \theta \chi$ in $(x_2,x_3]$,  $-\lambda-\theta\chi$ in $(x_3,x_4]$,  $\lambda-\theta\chi$ in $(x_4,x_5]$,  and $b$ in $(x_5,\infty)$ (on the event that $z(V)$ lies in the bounded connected component of $\h \setminus \gamma$).
On the other hand,  we have that $\wh{\Fh}_n \circ \wh{\psi}_n^{-1} - \chi \arg(\wh{\psi}_n^{-1})'$ converges locally uniformly to $\wh{\Fh} \circ \wh{\psi}^{-1} - \chi \arg(\wh{\psi}^{-1})'$ as $n \to \infty$ a.s.  Combining,  we obtain that $\wh{\Fh} \circ \wh{\psi}^{-1} - \chi \arg(\wh{\psi}^{-1})'$ has the required boundary conditions and so this completes the proof of the lemma. 
\end{proof}

Next,  we assume that we have the setup of Lemma~\ref{lem:flow_line_representation}.  Then,  we will show in the following lemma that there exists a connected component $U$ of $\h \setminus \fan(\theta_1,\theta_2)$ which agrees with $\partial V$ near $z(V)$.  This will imply that $x(U) \neq y(U)$ and so the two boundary arcs of $\partial U$ described in Lemma~\ref{lem:form_of_the_boundary} are non-degenerate.  
We also show that $\del U \sm \del V$ is contained in flow lines of a conditional field.
We remark that $U$ may not contain $\gg(t)$.
Figure~\ref{fig:ft_gamma} depicts $\F(t;\gg)$ and these conditional flow lines.

\begin{figure}[t]
\includegraphics[scale=0.8]{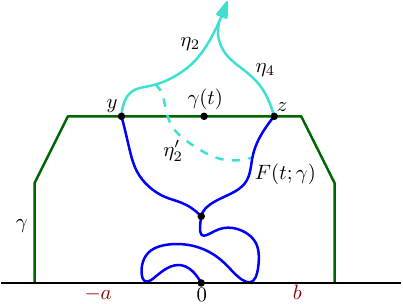}
\caption{The flow lines $\eta_2$ and $\eta_4$ are defined as in Lemma~\ref{lem:v_determines_a_component_u} and merge before either of them merge with $V$. The dashed line $\eta_2'$ is an alternate realization of $\eta_2$ in the case where $\eta_2'$ merges with $\del V$. In either case, $\gg(t)$ may or may not be contained in $U$, but $\del U \sm \del V$ will always be contained in $\eta_2 \cup \eta_4$.}
\label{fig:ft_gamma}
\end{figure}

\begin{lemma}\label{lem:v_determines_a_component_u}
Suppose that $V$ is a bounded connected component of $\h \setminus (\fan(\theta_1,\theta_2;\gamma) \cup \gamma)$ with $\partial V \cap \fan(\theta_1,\theta_2;\gamma) \neq \emptyset$ as in Lemma~\ref{lem:flow_line_representation}, and that $z(V)$ in the bounded connected component of $\h \setminus \gamma$.
\begin{enumerate}[(i)]
\item\label{it:component_contained} There exists a unique connected component $U$ of $\h \setminus \fan(\theta_1,\theta_2)$ so that $\partial U$ agrees with $\partial V \cap \fan(\theta_1,\theta_2;\gamma)$ near $z(V)$.
\item\label{it:rest_of_boundary}
We have that $\partial U \setminus \partial V$ is contained in the union of the flow lines of angle $\theta$ starting from the two prime ends corresponding to $x_2$ and $x_4$ from Lemma~\ref{lem:boundary_conditions} in $\h \setminus \fan(t;\gamma)$.
\item\label{it:marked_points_distinct} We have that $x(U) \neq y(U)$.
\end{enumerate}
\end{lemma}
\begin{proof}
We begin by proving (\ref{it:component_contained}) and (\ref{it:rest_of_boundary}).
Let $\eta_2$ (resp.\ $\eta_4$) be the flow line of $h|_{\h \setminus \fan(t;\gamma)}$ of angle $\theta$ starting from $\wh{\psi}^{-1}(x_2)$ (resp.\ $\wh{\psi}^{-1}(x_4)$).
Then we have that $\eta_2$ has the law of an $\SLE_{\kappa}(\rho^L;\rho^{1,R},\rho^{2,R},\rho^{3,R})$ process in the unbounded connected component $\wh{G}$ of $\h \setminus \fan(t;\gamma)$ with the force points located at $\wh{\psi}^{-1}(x_1),\wh{\psi}^{-1}(x_3),\wh{\psi}^{-1}(x_4)$,  and $\wh{\psi}^{-1}(x_5)$ respectively,  where 
\[\rho^L = -1+\frac{a - \tht \chi}{\lambda}, \quad  \rho^{1,R} = -2, \quad \rho^{2,R} = 2, \quad  \rho^{3,R} = -1+\frac{b + \tht \chi}{\lambda}.\]
Similarly,  we have that $\eta_4$ has the law of an $\SLE_{\kappa}(\rho^{3,L},\rho^{2,L},\rho^{1,L};\rho^R)$ process in $\wh{G}$ with the force points located at $\wh{\psi}^{-1}(x_1),\wh{\psi}^{-1}(x_2),\wh{\psi}^{-1}(x_3)$,  and $\wh{\psi}^{-1}(x_5)$ respectively,  and $\rho^R = \rho^{3,R},  \rho^{1,L} = -2,  \rho^{2,L} = 2$,  and $\rho^{3,L} = \rho^L$.  
It follows from the proofs of \cite[Lemma~2.1, Theorem~3.1]{mw2017intersections} and \cite[Lemma~15]{dubedat2007duality} that $\eta_2$ (resp.\ $\eta_4$) does not hit the clockwise (resp.\ counterclockwise) arc of $\partial V$ from $z(V)$ to $\psi\nv(x_2)$ (resp.\ $\psi\nv(x_4)$) except for at its starting point.  
Moreover,  either $\eta_2$ merges with the counterclockwise arc of $\partial V$ from $z(V)$ to $\gamma$ or it merges with $\eta_4$.  Similarly,  either $\eta_4$ merges with the clockwise arc of $\partial V$ from $z(V)$ to $\gamma$ or it merges with $\eta_2$.  If $\eta_2$ and $\eta_4$ merge before either of them merges with $\partial V$,  we have that both of them do not intersect $\partial V$ except for their starting points. 
Some of these possibilities are shown in Figure~\ref{fig:ft_gamma}.
Let $\wt{U}$ be the connected component of $\h \setminus (\fan(t;\gamma) \cup \eta_2 \cup \eta_4)$ whose boundary contains $z(V)$. 
In any of these cases, neither $\eta_2$ nor $\eta_4$ can hit $\del V$ at $z(V)$, meaning that $\del \wt{U}$ agrees with $\del V$ near $z(V)$. Next we will prove that $\wt{U}$ is a connected component of $\h \sm \F(\tht_1, \tht_2)$ which agrees with $V$ near $z(V)$, and that the boundary of $\wt{U}$ excluding $\del V$ is contained in $\eta_2 \cup \eta_4$, thus proving (\ref{it:component_contained}) and (\ref{it:rest_of_boundary}).

We will first show that $\fan(\theta_1,\theta_2) \cap \wt{U} = \emptyset$ a.s.  This will imply that $\wt{U} \subseteq U$ for some connected component $U$ of $\h \setminus \fan(\theta_1,\theta_2)$.  Suppose that $\eta_2$ merges with the counterclockwise arc of $\partial V$ from $z(V)$ to $\gamma$.  Then $\partial \wt{U}$ consists of part of $\partial V$ and the part of $\eta_2$ up until it merges with the counterclockwise arc of $\partial V$ from $z(V)$ to $\gamma$.  
Then the flow line interaction rules imply that for all $\phi \in \Q \cap [\theta,\theta_2]$,  we have that the flow line of $h$ of angle $\phi$ does not cross $\eta_2$ from left to right and so it does not enter $\wt{U}$. 
We note that the flow line of $h$ with angle $\tht$ (in the unbounded component of $\h\sm\F(t;\gg)$) can be sampled by first sampling $\F(t;\gamma)$ and then sampling the flow line of angle $\tht$ of the field $h\circ\wh\psi\nv - \chi\arg(\wh\psi\nv)'$ introduced in Lemma~\ref{lem:boundary_conditions}. This follows from the martingale characterization in \cite[Theorem~2.4]{ms2016imag1} and is similar to the argument in, for example, \cite[Lemmas~4.7, 7.1]{ms2016imag1}.
Also,  for all $\phi \in \Q \cap [\theta_1,\theta]$,  the flow line of $h$ with angle $\phi$ cannot enter $\wt{U}$ since it can only do so by crossing $\eta_2$ from left to right,  but then it has to exit $\wt{U}$ and it can only do so by crossing $\eta_2$ again.  But the latter cannot happen and so $\fan(\theta_1,\theta_2) \cap \wt{U} = \emptyset$.  
Analogously, $\fan(\theta_1,\theta_2) \cap \wt{U} = \emptyset$ if $\eta_4$ merges with the clockwise arc of $\partial V$ from $z(V)$ to $\gamma$.
Finally,  if $\eta_2$ and $\eta_4$ merge,  we have that $\partial \wt{U}$ consists of part of $\partial V$ and the parts of $\eta_2$ and $\eta_4$ up until they merge.  
Then,  for all $\phi \in \Q \cap [\theta_1,\theta]$,  the $\phi$-angle flow line of $h$ cannot enter $\wt{U}$ since it can only do so by crossing $\eta_4$ from right to left and the latter cannot happen.  
Similarly,  for all $\phi \in \Q \cap [\theta,\theta_2]$,  the $\phi$-angle flow line of $h$ cannot enter $\wt{U}$ since it can only do so by crossing $\eta_2$ from left to right,  and the latter cannot happen.  Combining the above observations,   we obtain that $\wt{U} \cap \fan(\theta_1,\theta_2) = \emptyset$ in every case and so $\wt{U} \subseteq U$ for some connected component $U$ of $\h \setminus \fan(\theta_1,\theta_2)$.

To show that $\wt{U} = U$,  suppose that $\wt{U} \neq U$.  Then there exists $w \in U \cap (\eta_2 \cup \eta_4)$.  
Suppose without loss of generality that $w \in \eta_2$. Then for all $\eps > 0$, using Proposition~\ref{prop:bounded_hausdorff}, there a.s.\ exists $\phi \in \Q \cap [\theta,\theta_2]$ such that the $\phi$-angle flow line of $h$ enters $B(w, \eps)$. This is a contradiction since $U$ is open and hence there exists some $\eps > 0$ such that $F(\tht_1, \tht_2)$ does not intersect $B(w,\eps)$, so $\wt{U} = U$, completing the proof of (\ref{it:component_contained}) and (\ref{it:rest_of_boundary}).

Finally, (\ref{it:marked_points_distinct}) follows from \cite[Lemma~3.8]{ms2016imag1} and the fact that we have explicit descriptions of the local set boundary conditions of $\del U$ and $\del V$ from Lemmas~\ref{lem:form_of_the_boundary} and~\ref{lem:flow_line_representation}. In particular, we see that $x(U) = z(V)$, but that $y(U)$ cannot be equal to this prime end by comparing the boundary conditions.
\end{proof}

Suppose that we have $z \in \h$ fixed and let $\theta = \theta(z)$ be the angle corresponding to the connected component $U$ of $\h \setminus \fan(\theta_1,\theta_2)$ containing $z$ as in Lemma~\ref{lem:form_of_the_boundary}, which exists a.s.  Set
\[ \fan(z) = \cap_{\substack{\phi_1 < \theta < \phi_2\\ \phi_1,\phi_2 \in \Q}} \fan(\phi_1,\phi_2).\]
Let $\CP$ be the (countable) set of simple piecewise linear paths $\gamma \colon [0,1] \to \ol{\h}$ with $\gamma(0) \in (-\infty,0) \cap \Q$, $\gamma(1) \in (0,\infty) \cap \Q$, and $\gamma(t) \in \h$ for each $t \in (0,1)$, and which consist of finitely many line segments and where all intersection points of distinct line segments have rational coordinates. We parameterize each path by arc length divided by its total arc length (so that $\gg$ is defined on $[0,1]$).
Let $\gamma$ be a path in $\CP$.  
Then,  we say that $U$ is discovered by $(\gamma, t)$ if $U$ is the unique connected component of $\h\sm\F(\tht_1, \tht_2)$ so that $\del U$ agrees with $\del V$ near $z(V)$, as in Lemma~\ref{lem:v_determines_a_component_u}, where $V$ is the bounded connected component of $\h \setminus (\fan(\theta_1,\theta_2;\gamma) \cup \gamma)$ whose boundary contains $\gamma(t)$.
Recall that $\tht(V) \equiv \tht(\gg,t)$ is defined in Lemma~\ref{lem:flow_line_representation} as the infimum of angles $\phi \in [\tht_1, \tht_2]$ for which $\eta_\phi$ first hits $\gg$ to the left of $\gg(t)$.
We will show in the following lemma that it is a.s.\ the case that every connected component of $\h \setminus \fan(\theta_1,\theta_2)$ can be discovered using the fixed and countable collection of paths $\CP$.  In particular,  combining with Lemma~\ref{lem:fan_symmetry},  this will imply that the two boundary arcs of any connected component $U$ of $\h \setminus \fan(\theta_1,\theta_2)$ are non-degenerate and that $\partial U$ can be represented as parts of flow lines of a conditional $\text{GFF}$.

\begin{lemma}\label{lem:discover_component}
Suppose that $z \in \h \cap \Q^2$ is fixed.  Then a.s.\ there exists a path $\gamma \in \CP$ and $t \in [0,1] \cap \Q$ so that the connected component $U$ of $\h \setminus \fan(\theta_1,\theta_2)$ containing $z$ is discovered by $(\gamma,t)$.
\end{lemma}
\begin{proof}
\textit{Step 1. $U$ is the connected component of $\h\sm\F(z)$ containing $z$.}
For $\phi_1 < \tht < \phi_2$, the connected component of $\h\sm\F(\phi_1,\phi_2)$ containing $z$ can be seen to be equal to $U$, since by its definition $U$ is determined only by flow lines with angle in $(\phi_1, \phi_2)$. It follows from the definition of $\F(z)$ that the connected component of $\h\sm\F(z)$ containing $z$ is equal to $U$.

\textit{Step 2. Choosing a suitable curve $\gg \in \CP$.}
Now, we will show that for any point $z \in \h \cap \Q^2$, there a.s.\ exist $\gg \in \CP$ and $t \in [0,1] \cap \Q^2$ such that $\tht(\gg, t) = \tht(U)$ and that $U$ intersects the unbounded component of $\h\sm\gg$. We will explain how this completes the proof in Step 3.

First, we show that $\F(z)$ is equal to its left and right boundaries, in a way that we will now make precise.
Let $W$ be the connected component of $\C \setminus \fan(z)$ containing $-i$.  Then,  $W$ is a simply connected set.
We claim that $\fan(z) = \partial W$.
Indeed,  suppose that this does not hold.  Then,  there exists $w \in \fan(z)$ and $\epsilon>0$ such that $\overline{B(w,\epsilon)} \subseteq \C \setminus \overline{W}$.
Let $W_L$ (resp.\ $W_R$) be the union of the connected components of $\h\sm\F(z)$ whose boundaries contain non-trivial intervals of $\R_-$ (resp.\ $\R_+$).
There exists $\phi \in \Q$ such that $\phi \neq \theta$ and the flow line of $h$ of angle $\phi$ intersects $B(w,\epsilon/2)$.   Suppose that $\phi > \theta$.  Then the flow line interaction rules imply that the $\phi$-angle flow line lies in $\overline{W}_L$ and so $B(w,\epsilon/2) \cap \overline{W}_L \neq \emptyset$ which implies that $B(w,\epsilon/2) \cap \overline{W} \neq \emptyset$,  and that is a contradiction.  Suppose that $\phi < \theta$.  Again,  the flow line interaction rules imply that $\eta_{\phi}$ lies in $\overline{W}_R$ and so $B(w,\epsilon/2) \cap \overline{W}_R \neq \emptyset$ which leads to a contradiction.  It follows that $\fan(z) = \partial W$.

Since $W$ is simply connected, there exists a unique conformal map $f \colon W \to \h$ such that $f(0)=0,f(\infty) = \infty$,  and $f(x)=1$,  where $x$ is the last point in the boundary of the connected component of $\h \setminus \eta_{\theta_2}$ containing $-1$ that $\eta_{\theta_2}$ hits.  
Note that $f$ extends to a homeomorphism mapping $\partial W$ onto $\partial \h$,  where the points on $\partial W$ are seen as prime ends. We define the left side $F_L$ (resp.\ right side $F_R$) of $\fan(z)$ to be given by $f^{-1}([0,\infty))$ (resp.\ $f^{-1}((-\infty,0])$). 
By Proposition~\ref{prop:locally_connected} and Remark~\ref{rem:locally_connected}, $\F(z)$ and $\del W$ are locally connected, so by \cite[Theorem~2.1]{pommerenke1992boundary} the map $f$ extends to a continuous map (with respect to the spherical metric) from $\ov{W}$ to $\ov{\h}$. It follows that the left and right boundaries $F_L$ and $F_R$ are (possibly non-simple) curves, naturally parameterized by the map $f$. In this way, $\F(z)$ is exactly the union of its left and right boundaries.
By Step 1, $\del U$ is contained in $F_L \cup F_R$, and since $\del U$ is locally connected, its boundary can be parameterized by a (possibly non-simple) curve \cite[Theorem~2.1]{pommerenke1992boundary}.

\begin{figure}[t]
\includegraphics[scale=0.8]{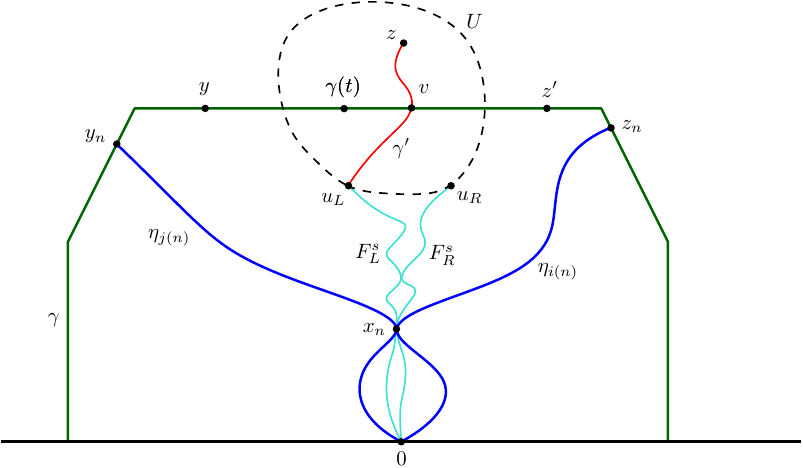}
\caption{Here we depict the setup of Lemma~\ref{lem:discover_component}. This figure is not meant to represent the reality and is included mainly to explain our notation. It is (deliberately) misleading in the sense that $\del U, u_L$ and $u_R$ cannot appear in the configuration shown. In fact, as we will eventually show, $F_L^s$ and $F_L^r$ are equal, and $u_L = u_R$. We choose this layout to emphasize that we have not proven the above at this stage, and so we have to allow for situations where we do not know the configuration of $u_L, u_R$ and $\del U$.} 
\label{fig:discovering_u}
\end{figure}

Next, suppose we have the notation of Lemma~\ref{lem:form_of_the_boundary} (see Figure~\ref{fig:un_boundary_conditions}).
Since $\del U \subseteq F_L \cup F_R$, at least one of $F_L, F_R$ intersects $\del U$. Suppose without loss of generality (for the remainder of this step) that $F_L$ intersects $\del U$ (while allowing that $F_R$ may also do so). Let $u_L$ be the first point at which $F_L$ intersects $\del U$ and let $F_L^s$ be $F_L$ stopped the first time it hits $u_L$. Let $u_R$ be the first point at which $F_R$ intersects $\del U$, if such a point exists, and set $u_R = 0$ otherwise. Let $F_R^s$ be $F_R$ stopped the first time it hits $u_R$.
We claim that $z$ lies in the unbounded connected component of $\C \sm (F_L^s \cup F_R^s)$. Indeed, suppose that $z$ lies in a bounded component $B$ of this set. Then there must exist a point $u \in \del U \sm (F_L^s \cup F_R^s)$, and $\eps > 0$ such that $B(u,\eps)$ is contained in $B$.
Since $u \in F_L \cup F_R$, it is in $\overline{W_L} \cup \ov{W_R}$, meaning that $B \cap (W_L \cup W_R)$ is not empty. But every point in $W_L \cup W_R$ is in $W$, which is an unbounded connected component of $\C \sm \F(z)$. Since $F_L^s \cup F_R^s \subseteq \F(z)$, such a point cannot be in a bounded component of $\C \sm (F_L^s \cup F_R^s)$, yielding a contradiction and proving the claim.
Since $F_L^s \cap F_R^s$ is bounded and $z$ is in the unbounded connected component of its complement, there exists a path $\gg \in \CP$ such that $F_L^s \cup F_R^s$ is in the bounded connected component of $\h\sm\gg$ and that $z$ is in the unbounded component.

To complete this step we show that there exists $t \in [0,1] \cap \Q$ such that $\tht(\gg, t) = \tht(U)$. Since $\del U$ is locally connected, the unique conformal map $g \colon \D \to U$ with $g(0) = z, g'(0) > 0$ extends to a continuous map from $\ov\D \to \ov{U}$ (see \cite[Theorem~2.1]{pommerenke1992boundary}), and there exists a simple path $\gg'$ in $U \cup \{u_L\}$ from $u_L$ to $z$. Let $v \in U$ be the first time this path hits $\gg$.
Let $\gg^*$ be the (not necessarily simple) curve obtained by concatenating $F_L^s$ and $\gg'$ which starts at $0$, ends at $z$ and hits $\gg$ for the first time at $v$.
For each $n$, $\F(z)$ and hence also $\gg^*$ lie between $\eta_{i(n)}$ and $\eta_{j(n)}$ (in the sense that they are contained in the closure of the region between these two curves). Let $z_n$ (resp.\ $y_n$) be the first point at which $\eta_{i(n)}$ (resp.\ $\eta_{j(n)}$) hits $\gg$.
Then $v$ must lie on the clockwise arc of $\gg$ from $y_n$ to $z_n$. Let $y = \lim_{n \to \infty} y_n$ and $z' = \lim_{n \to \infty} z_n$. Since $U$ is open and $v \in U$, $v$ must lie on the (necessarily non-empty) clockwise arc from $y$ to $z'$, and we can choose $t \in [0,1] \cap \Q$ such that $\gg(t)$ also lies on this arc.
If $\phi > \tht(U)$, then $\phi > \phi_{j(n)}$ for some $n$, and hence by the flow line interaction rules $\eta_\phi$ must hit $\gg$ for the first time to the left of $y_n$, and thus to the left of $y$. Analogously, if $\phi < \tht(U)$, then $\eta_\phi$ hits $\gg$ to the right of $z'$.
It follows that $\tht(U) = \tht(\gg,t)$.

\textit{Step 3. Completing the proof.}
Suppose that $\gg \in \CP, t \in [0,1] \cap \Q$, that $\tht(U) = \tht(\gg,t)$ and $z \in U$ is in the unbounded connected component of $\h\sm\gg$. By the previous step such a pair $(\gg,t)$ exists a.s. Let $V$ be the unique bounded component of $\h \sm (\F(\tht_1, \tht_2;\gg) \cup \gg)$ with $\gg(t)$ on its boundary. To complete the proof of the lemma, we will prove that $U$ is discovered by $(\gg, t)$ a.s.
By the arguments used to prove (\ref{it:rest_of_boundary}) in  Lemma~\ref{lem:v_determines_a_component_u}, the set $\F(z)\sm\F(t;\gg)$ is the union of the two flow lines, $\eta_2$ and $\eta_4$, of angle $\tht$ started from the prime ends $\psi\nv(x_2)$ and $\psi\nv(x_4)$ respectively. In particular, since $z$ is in a bounded component of $\h\sm\F(z)$ and $z$ is in the unbounded component of $\h\sm\gg$, the only way this can occur is if $U$ is equal to the region bounded by $\del V, \eta_2$ and $\eta_4$ (or perhaps one of these flow lines if one of them merges into $\del V$, see Figure~\ref{fig:ft_gamma}). This means exactly that $U$ is discovered by $(\gg,t)$.
\end{proof}

As a consequence of Lemmas~\ref{lem:discover_component} and ~\ref{lem:fan_symmetry}, we will now show that for all $z \in \h$, a.s. $\F(z)$ has the form of a simple curve which splits into two curves which in turn merge back into a single curve.
From this we can deduce that a.s.,  every connected component of $\h \setminus \fan(\theta_1,\theta_2)$ is a Jordan domain with distinct marked boundary points.  We state and prove this in the following lemma.

\begin{lemma}\label{lem:component_jordan_domain}
Fix $z \in \h$ and let $U$ be the connected component of $\h \setminus \fan(\theta_1,\theta_2)$ containing $z$. Then $\F(z)$ is a.s. the union of a simple curve $\eta_a$ from $0$ to $x(U)$, two simple curves $\eta_L$ and $\eta_R$ from $x(U)$ to $y(U)$ which form, respectively, the clockwise and counterclockwise boundary arcs of $\del U$ from $x(U)$ to $y(U)$, and a simple curve $\eta_b$ from $y(U)$ to $\infty$. The curves $\eta_a, \eta_L, \eta_R$ and $\eta_b$ do not intersect except at $x(U)$ and $y(U)$, and the boundary conditions of $h$ given $\F(z)$ are those of the left (resp.\ right) side of a flow line of angle $\tht(U)$ on the left (resp.\ right) side of each of these curves.
\end{lemma}
\begin{proof}
The proof of this argument is effectively a bootstrapping of the argument used in Lemma~\ref{lem:discover_component}, and we will only mention here the details that we need to change. Suppose that we have $\gg$ as in Step 3 of the proof of Lemma~\ref{lem:discover_component}.
By Lemma~\ref{lem:boundary_conditions} the boundary conditions of $h$ given $\F(t;\gg)$ are such that $\eta_2$ and $\eta_4$ do not hit any fixed point on $F(t;\gg)$ a.s.\ and in particular do not hit $u_L \in \del U$ a.s.
As explained in the proof of Lemma~\ref{lem:v_determines_a_component_u} either $\eta_2$ and $\eta_4$ merge, or exactly one of $\eta_2$ or $\eta_4$ merges with $\del V$. Define $y' \equiv y'(U)$ to be this merging point.
Now $y'(U) \neq u_L$ a.s.\ so it follows that there exists $\wt{\gg} \in \CP$ such that $F_s^L(\wt{\gg})$ is in the bounded component of $\h\sm\wt\gg$ and that $y'$ is in the unbounded component.
If we now repeat the argument of Step 3 in Lemma~\ref{lem:discover_component}, we find that in this case the flow lines $\eta_2(\wt\gg)$ and $\eta_4(\wt\gg)$ a.s.\ do not merge with $\del V(\wt{\gg})$ and therefore must merge together to create the pocket $U(\wt\gg)$, which as before must be equal to $U$.
This ensures that the form of $F(z)\sm\F(t;\gg)$ consists of two flow lines of the same angle which merge and then go on to $\infty$.
To obtain information about the start of the flow line, we use reversibility and draw a new path $\wh\gg$ which disconnects $y'(U)$ from $0$ and does not intersect $\wt\gg$. By considering $\wh\F(t;\gg')$, the analogue of $\F(t;\wh\gg)$ except we look at the time reversal of the flow lines from $\infty$ to $0$ until they hit $\wh\gg$, we can conclude that the form of $\F(z)\sm\wh\F(\wh\gg)$ consists of two simple curves which merge together.
By combining these observations, we can deduce the structure of $\F(z)$. Finally, the form of the boundary conditions then can be read off using Lemma~\ref{lem:boundary_conditions} and \cite[Lemma~3.8]{ms2016imag1}, which allows us to identify that $y(U)$ and $y'(U)$ are equal, and that $x(U) = u_L = u_R$. 
\end{proof}

\subsection{Interaction of components in the complement of the fan}

We now recall some results from \cite{mw2017intersections}.  Suppose that
\[ 0 < \Delta \theta < \frac{\kappa \pi}{4-\kappa}\]
and let
\[ \rho(\Delta \theta) = \frac{1}{\pi}(\Delta \theta)\left(2 - \frac{\kappa}{2}\right) -2.\]
Then by \cite[Theorem~1.5]{mw2017intersections},  the Hausdorff dimension of the intersection of flow lines with angle difference $\Delta \theta$ is given by
\begin{equation}
\label{eqn:flow_line_intersection_dimension}
2 - \frac{1}{2\kappa} \left( \rho + \frac{\kappa}{2} + 2\right) \left( \rho - \frac{\kappa}{2} + 6\right).
\end{equation}
We in particular note that the above formula determines $\Delta \theta$. We recall also from \cite[Theorem~1.5]{ms2016imag1} that $\tht_c = \pi\kk/(4-\kk)$ is the critical angle difference in the sense that two flow lines of angle difference $\Delta \tht$ may intersect (without crossing) only when $\DD\tht < \tht_c$.

Suppose that $\eta$ is a flow line of a \text{GFF} which has the law of an $\SLE_{\kappa}(\rho)$ process in $\h$ from $0$ to $\infty$ with the force point located at $0^+$ and such that $\rho \in (-2,\frac{\kappa}{2}-2)$.
Then it is a.s.\ the case that for all $0<x_1<x_2$,  on the event that $\eta \cap [x_1, x_2] \neq \emptyset$, the Hausdorff dimension of $\eta \cap [x_1,x_2]$ is equal to \cite[Theorem~1.6]{mw2017intersections}
\begin{equation}
\label{eqn:boundary_flow_line_dimension}
1- \frac{1}{\kappa} \left(\rho+2 \right)\left(\rho + 4 - \frac{\kappa}{2}\right).
\end{equation}
In particular, from the dimension we can recover $\rho$ (or, equivalently, the angle of the flow line). 
In the following it is worth keeping in mind that $(2 - \kk/2)/\pi = \chi/\la$.

Let us now record some consequences of the above and Lemmas~\ref{lem:boundary_conditions}-\ref{lem:discover_component}.

\begin{lemma}
\label{lem:component_boundary_intersection}
Fix $z,w \in \h$ distinct points and let $U_1$ (resp.\ $U_2$) be the connected component of $\h \setminus \fan(\theta_1,\theta_2)$ containing $z$ (resp.\ $w$).  Then we have that:
\begin{enumerate}[(i)]
\item\label{it:intersection_side} On the event that $\partial U_1 \cap \partial U_2 \neq \emptyset$ and $U_1 \neq U_2$, we have that the left (but not the right) side of $\partial U_1$ intersects the right (but not the left) side of $\partial U_2$ or vice-versa. 
\item\label{it:boundary_intersection} $\partial U_1$, $\partial U_2$ can only intersect if their angle difference $|\tht(U_1) - \tht(U_2)|$ is below the critical angle.
\item\label{it:intersection_dimension} On the event that $\partial U_1 \cap \partial U_2 \neq \emptyset$, we have that the Hausdorff dimension of $\partial U_1 \cap \partial U_2 \neq \emptyset$ is a.s.\ equal to~\eqref{eqn:flow_line_intersection_dimension} with $\rho = \frac{1}{\pi}|\theta(U_1)-\theta(U_2)|\left(2-\frac{\kappa}{2}\right) - 2$.
\item\label{it:boundary_intersection_dimension} On the event that $\partial U_1 \cap \R_+ \neq \emptyset$ and that the right boundary of $U_1$ does not contain any open interval of $\R_+$, we have that the Hausdorff dimension of $\partial U_1 \cap \R_+$ is equal to~\eqref{eqn:boundary_flow_line_dimension} with $\rho = -1 +\frac{b + \theta(U_1)\chi}{\lambda}$.
\end{enumerate}
\end{lemma}

\begin{figure}[t]
\includegraphics[scale=0.8]{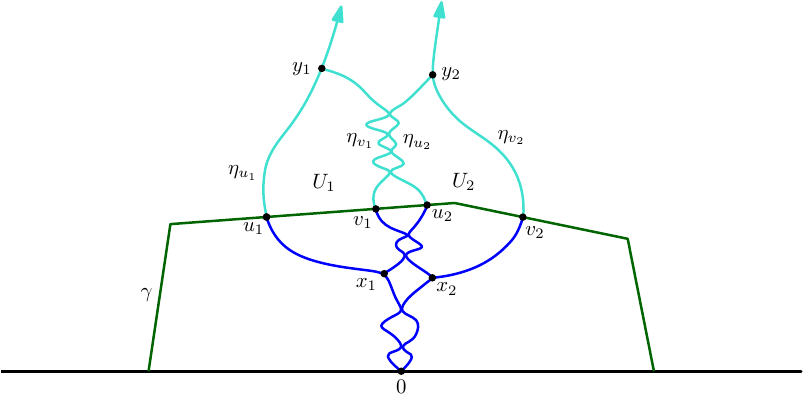}
\caption{This figure corresponds to parts (\ref{it:intersection_side}), (\ref{it:boundary_intersection}) and (\ref{it:intersection_dimension}) of Lemma~\ref{lem:component_boundary_intersection}. Since we know $\eta_{u_1}, \eta_{v_1}, \eta_{u_2}, \eta_{v_2}$ are flow lines of a GFF, we can deduce that (\ref{it:intersection_side}), (\ref{it:boundary_intersection}) and (\ref{it:intersection_dimension}) hold on the parts of $\del U_1, \del U_2$ that are traced after $F(\gg(s))$ and $F(\gg(t))$ leave the bounded component of $\h\sm\gg$, respectively. By varying $\gg$, we can prove the result for all of $\del U_1, \del U_2$. }
\label{fig:component_intersections}
\end{figure}

\begin{proof}

\noindent{\it Step 1.  Proof of (\ref{it:intersection_side}),  (\ref{it:boundary_intersection}),  and (\ref{it:intersection_dimension}).}  Suppose that we are working on the event that $U_1 \neq U_2$ and that $\partial U_1 \cap \partial U_2 \neq \emptyset$.  
By the proofs of Lemmas~\ref{lem:discover_component} and~\ref{lem:component_jordan_domain}, we know that if $z \in U_1$, then $F(z)$ is a simple curve until it hits $x(U_1)$, then ``splits" into two curves and traces $\del U_1$. These two curves merge at $y(U_1)$, after which $F(z)$ is again a simple curve until it reaches $\infty$. Since $F(z)$ depends on $z$ only through $\tht_1 \equiv \tht(z) \equiv \tht(U_1)$, it follows that $\tht(U_1) \neq \tht(U_2)$ if $U_1 \neq U_2$.
Therefore, we can assume that $\theta(U_1) > \theta(U_2)$. 
Fix $\eps > 0$ small enough that $|y(U_j)-x(U_j)| > \epsilon$ for $j=1,2$.
We let $\gamma$ be a simple path in $\overline{\h}$ as in Lemma~\ref{lem:discover_component} such that $\gamma(s) \in U_1, \gamma(t) \in U_2$ for some $s,t \in \Q \cap (0,1)$ with $s<t$ and that $F(\gg(s))$ (resp.\ $F(\gg(t))$) does not hit $\gg$ before it hits $x(U_1)$ (resp.\ $x(U_2)$).
We also assume that $\partial U_j \setminus B(x(U_j) ,  \epsilon)$ is contained in the unbounded connected component of $\overline{\h} \setminus \gamma$ for $j=1,2$. We note that we can always choose such a path for all $\epsilon>0$ sufficiently small. 
The following setup is depicted in Figure~\ref{fig:component_intersections}.
We let $u_1$ (resp.\ $v_1$) be the point on $\fan(s;\gamma)$ which is hit first (resp.\ last) by $\gamma$.  Similarly,  we let $u_2$ (resp.\ $v_2$) be the point on $\fan(t;\gamma)$ which is hit first (resp.\ last) by $\gamma$.  Moreover,  we let $\eta_{u_1}$ (resp.\ $\eta_{v_1}$) be the flow line starting from $u_1$ (resp.\ $v_1$) with angle $\theta(U_1)$ of the conditional field given $\fan(s;\gamma)$.  Then we have that $\eta_{u_1}$ and $\eta_{v_1}$ merge upon hitting $y(U_1)$ and then evolve towards $\infty$.  
Also,  it follows from Lemma~\ref{lem:v_determines_a_component_u} that $\partial U_1 \setminus B(x(U_1),\epsilon)$ is contained in $\eta_{u_1} \cup \eta_{v_1}$.  Similarly,  we let $\eta_{u_2}$ (resp.\ $\eta_{v_2}$) be the flow line starting from $u_2$ (resp.\ $v_2$) with angle $\theta(U_2)$ of the conditional field given $\fan(t;\gamma)$.  Again,  we have that $\eta_{u_2}$ and $\eta_{v_2}$ merge upon hitting $y(U_2)$ and then evolve towards $\infty$.  Also,  $\partial U_2 \setminus B(x(U_2),\epsilon)$ is contained in $\eta_{u_2} \cup \eta_{v_2}$.  Moreover,  we have that $(\partial U_1 \setminus B(x(U_1),\epsilon)) \cap (\partial U_2 \setminus B(x(U_2),\epsilon)) \subseteq \eta_{v_1} \cap \eta_{u_2}$,  since the flow line interaction rules imply that $\eta_{v_1}$ cannot cross $\eta_{u_2}$ when they are both restricted to the unbounded connected component of $\h \setminus (\fan(s;\gamma) \cup \fan(t;\gamma))$.  Furthermore,  the flow line interaction rules imply that $|\theta(U_1) - \theta(U_2)| \leq \theta_c$ and by the proof of Lemma~\ref{lem:v_determines_a_component_u} we have that the left (resp.\ right) side of $\partial U_j$ contained in $\partial U_j \setminus B(x(U_j),\epsilon)$ is contained in the part of $\eta_{u_j}$ (resp.\ $\eta_{v_j}$) stopped at the first time that it hits $y(U_j)$ for $j=1,2$.  It follows that $(\partial U_1 \setminus B(x(U_1),\epsilon)) \cap (\partial U_2 \setminus B(x(U_2),\epsilon))$ is contained in the intersection of the right side of $\partial U_1$ with the left side of $\partial U_2$.  Furthermore,  we have that any fixed segment of the right (resp.\ left) side of $\partial U_1$ (resp.\ $\partial U_2$) is contained in $\partial U_1 \setminus B(x(U_1),\epsilon)$ (resp.\ $\partial U_2 \setminus B(x(U_2),\epsilon)$) for all $\epsilon>0$ sufficiently small,  and so it follows by \cite[Theorem~1.5]{mw2017intersections} that the Hausdorff dimension of $(\partial U_1 \setminus B(x(U_1),\epsilon)) \cap (\partial U_2 \setminus B(x(U_2),\epsilon))$ is given by \eqref{eqn:flow_line_intersection_dimension} with $\rho = \frac{1}{\pi}|\theta(U_1)-\theta(U_2)|\left(2-\frac{\kappa}{2}\right) - 2$.  Claims (\ref{it:intersection_side}),  (\ref{it:boundary_intersection}),  and (\ref{it:intersection_dimension}) then follow since we can always choose such a path $\gamma$ (among a countable and fixed set of choices) for all $\epsilon>0$ sufficiently small.

\noindent{\it Step 2.  Proof of (\ref{it:boundary_intersection_dimension}).} Suppose that we have the setup of the previous paragraph.  First,  we note that Lemma~\ref{lem:boundary_conditions} implies that conditionally on $\fan(s;\gamma)$,  the curve $\eta_{v_1}$ has the law of an $\SLE_{\kappa}(\rho^{3,L},\rho^{2,L},\rho^{1,L};\rho^R)$ process from $v_1$ to $\infty$ in the unbounded connected component of $\h \setminus \fan(s;\gamma)$,  where
\[\rho^{3,L} = -1+\frac{a - \tht(U_1)\chi}{\lambda}, \quad \rho^{2,L} = 2,\quad  \rho^{1,L} = -2,\quad \rho^R = -1+\frac{b + \tht(U_1)\chi}
{\lambda},\]
and the force points are respectively at the leftmost point on $\F(s;\gg) \cap \R_-, u_1, x(U_1)$ and the rightmost point on $F(s;\gg) \cap \R_+$.
Hence,  by conformally mapping the unbounded connected component of $\h \setminus \fan(s;\gamma)$ onto $\h$,  in order to prove (\ref{it:boundary_intersection_dimension}),  it suffices to prove that the following is true.  Fix $-\infty< w_3<w_2<w_1<0$ and let $\eta$ be an $\SLE_{\kappa}(\rho^{3,L},\rho^{2,L},\rho^{1,L};\rho^R)$ process in $\h$ from $0$ to $\infty$ with the force points located at $w_3,w_2,w_1$,  and $1$ respectively.  
Then,  a.s.,  for all $x,y \in \Q \cap (1,\infty)$ with $x<y$,  we have that on the event that $\eta$ does not merge with $[w_2,w_1]$ before hitting $[x,y]$,  we have that the Hausdorff dimension of $\eta \cap [x,y]$ is given by \eqref{eqn:boundary_flow_line_dimension} with $\rho = \rho^R$.
To prove the latter claim,  we fix $x,y$ as above and consider sequences of stopping times $(\tau_n)$ and $(\sigma_n)$ as follows.  We set $\tau_1 = \inf\{t \geq 0 : \eta(t) \in [x,y]\}$,  $\sigma_1 = \inf\{t \geq \tau_1 : |\eta(t)-\eta(\tau_1)| \geq \delta\}$,  and inductively we define $\tau_{n+1} = \inf\{t \geq \sigma_n : \eta(t) \in [x,y]\}$,  $\sigma_{n+1} = \inf\{t \geq \tau_{n+1} : |\eta(t)-\eta(\tau_{n+1})| \geq \delta\}$,  where $\delta>0$ is fixed and such that $1<x-\delta$. 
It follows from \cite[Lemma~2.7]{mw2017intersections} that for all $n \in \N$,  if we condition on the event that $\tau_n < \infty$,  we have that the law of $\eta|_{[\tau_n,\sigma_n]}$ is mutually absolutely continuous with respect to the law of an $\SLE_{\kappa}(\rho^R)$ process in $\h$ from $\eta(\tau_n)$ to $\infty$ with the force point located at $\eta(\tau_n)^+$ and stopped at the first time that it exits $B(\eta(\tau_n),\delta)$.  It then follows from \cite[Theorem~1.6]{mw2017intersections} that the Hausdorff dimension of $\eta([\tau_n,\sigma_n]) \cap [x,y]$ is given by \eqref{eqn:boundary_flow_line_dimension} with $\rho = \rho^R$ a.s.\ on the event that $\tau_n < \infty$,  for all $n \in \N$,  and so the claim follows.  This completes the proof of the lemma.
\end{proof}

We let $\fan_c(z)$ be the part of the left and right boundaries of $\fan(z)$ up until $z$ is disconnected from $\infty$.  We also let $U(z)$ be the connected component of $\h \setminus \fan(\theta_1,\theta_2)$ containing $z$.  We show in the next lemma that $\fan_c(z)$ is a local set for $h$ and identify the boundary values of $h$ restricted to the unbounded connected component of $\h \setminus \fan_c(z)$. 

\begin{lemma}
\label{lem:locality_of_pocket_up_to_hitting}
We have that $\fan_c(z)$ is a local set for $h$.  Let $U = U(z)$ and let $y = y(z)$ be the closing point for $U$.  Let $\varphi$ be the unique conformal map from the unbounded connected component of $\h \setminus \fan_c(z)$ to $\h$ which takes $y, \infty$ and the the rightmost point of $\F_c(z) \cap \R_+$ to $0, \infty$ and $1$, respectively.
Then there exists $x_1 < x_2 < x_3$ so that the boundary conditions for $h \circ \varphi^{-1} - \chi \arg( \varphi^{-1})'$ are given by $-a$ in $(-\infty,x_1]$, $-\lambda - \theta \chi$ in $(x_1,x_2]$, $\lambda - \theta \chi$ in $(x_2,x_3]$, and $b$ in $(x_3,\infty)$.
\end{lemma}
\begin{proof}

It follows from the same argument used to prove Lemma~\ref{lem:locality} that $\fan_c(z)$ is a local set for $h$.  We now focus on the proof of the second claim of the lemma.  It follows from Lemma~\ref{lem:discover_component} that we can choose a path $\gamma$ as in the statement of Lemma~\ref{lem:discover_component} such that $\gamma(t) \in U$ for some $t \in \Q \cap (0,1)$, that $F(\gg(t))$ hits $x(U)$ before hitting $\gg$, and that $\fan(t;\gamma) \subseteq \fan_c(z)$.  Let $\Fh$ (resp.\ $\Fh_c$) be the harmonic extension to $\h \setminus \fan(t;\gamma)$ (resp.\ $\h \setminus \fan_c(z)$) of the boundary values of $h$ on $\partial (\h \setminus \fan(z;\gamma))$ (resp.\ $\partial (\h \setminus \fan_c(z))$).  It then follows by combining the form of the boundary conditions of $h$ on $\fan(z;\gamma)$ (Lemma~\ref{lem:boundary_conditions}) with \cite[Proposition~3.8]{ms2016imag1} that the boundary conditions of $h$ on the part of the left (resp.\ right) boundary of $\F_c(z)$ lying on $\fan(t;\gamma)$ are given by $-\lambda-\theta\chi$ (resp.\ $\lambda-\theta\chi$).  Since we can discover all of $\fan_c(z)$ by varying the path $\gamma$ as in the proof of Lemma~\ref{lem:component_boundary_intersection},  we obtain that the boundary conditions of $h \circ \varphi^{-1} - \chi \arg(\varphi^{-1})'$ have the desired form.
\end{proof}

We are now ready to show that the marked points of each connected component of $\h \setminus \fan(\theta_1,\theta_2)$ are a.s.\ determined by $\fan(\theta_1,\theta_2)$.

\begin{lemma}
\label{lem:two_points_measurable}
Fix $z \in \h$ and let $U$ be the connected component of $\h \setminus \fan(\theta_1,\theta_2)$ containing $z$.  Then $\{x(U), y(U)\}$ are a.s.\ determined by $\fan(\theta_1,\theta_2)$.
\end{lemma}

We emphasize that the statement of Lemma~\ref{lem:two_points_measurable} implies that the pair of points $\{x(U), y(U)\}$ is determined by $\fan(\theta_1,\theta_2)$ which means that the two marked arcs of $\partial U$ are determined but not which one is the clockwise (resp.\ counterclockwise) arc.

\begin{proof}[Proof of Lemma~\ref{lem:two_points_measurable}]
\noindent{\it Step 1.  Overview and setup.} Suppose that $U$ is the connected component of $\h \setminus \fan(\theta_1,\theta_2)$ containing $z$.
We are going to show that $x(U)$, $y(U)$ are singled out among points in $\partial U$ by the property that there exist infinitely many connected components $U_L$, $U_R$ of $\h \setminus \fan(\theta_1,\theta_2)$ so that $U_L$ (resp.\ $U_R$) intersects the left (resp.\ right) side of $\partial U$ and $\partial U_L \cap \partial U_R \neq \emptyset$ and $U$, $U_L$, $U_R$ together surround $y(U)$ (or $x(U)$) and we have that $|\theta(U) - \theta(U_L)| < | \theta(U_L) - \theta(U_R)|$ and the same holds with the roles of $U_L$, $U_R$ swapped.  
We note that the way that boundaries of distinct connected components of $\h \setminus \fan(\theta_1,\theta_2)$ interact (Lemma~\ref{lem:component_boundary_intersection}) implies that the marked points $x(U),y(U)$ are the only points on $\partial U$ for which there exist components $U_L,U_R$ as above.
Since the Hausdorff dimension of the set of intersection points of the boundaries of two distinct connected components of $\h \setminus \fan(\theta_1,\theta_2)$ determines the absolute value of their angle difference (Lemma~\ref{lem:component_boundary_intersection}),  we obtain that the latter is a.s.\ determined by $\fan(\theta_1,\theta_2)$ and so the above event is also determined by $\fan(\theta_1,\theta_2)$.  

To complete the proof, it suffices to show that there are a.s.\ infinitely many such components $U_L, U_R$ for each connected component $U$ of $\h \setminus \fan(\theta_1,\theta_2)$.  By the reversal symmetry of the $\SLE$ fan (Lemma~\ref{lem:fan_symmetry}),  it suffices to show that this is a.s.\ the case for the closing point $y(U)$ of $U$. 
We let $\varphi$ be as in the statement of Lemma~\ref{lem:locality_of_pocket_up_to_hitting} meaning that the boundary conditions for the field $\wt{h} = h \circ \varphi^{-1} - \chi \arg (\varphi^{-1})'$ are as described in the statement of Lemma~\ref{lem:locality_of_pocket_up_to_hitting}. 
For all $k \in \N$, let $A_k$ be the annulus  $B(0, 2^{-k}) \sm B(0, 2^{-(k+1)})$, and fix $k_0 \in \N$ sufficiently large and assume that we are working on the event that the boundary conditions of $\wt{h}$ on $[-2^{-k_0 + 1}, 2^{-k_0 + 1}]$ are given by $-\lambda - \theta\chi$ on $(-2^{-k_0 + 1},0]$ and $\lambda-\theta\chi$ in $(0,2^{-k_0 + 1}]$, where $\theta = \theta(U)$.
We note that we can find $k_0 \in \N$ with the above property a.s. and that in this case, by \cite[Proposition~3.4]{ms2016imag1}, the law of $\wt{h}$ restricted to $A_{k_0}$ is absolutely continuous with respect to the law of the restriction to $A_{k_0}$ of a GFF $\wh h$ on $\h$ with boundary conditions given by $-\la - \tht \chi$ on $\R_-$ and $\la - \tht \chi$ on $\R_+$.

\textit{Step 2. Defining the events $E_k$ and $G_k$.}
For each $k \geq 1$,  we let $\eta_k^L$ (resp.\ $\eta_k^R$) be the flow line of $\wh{h}$ of angle $\theta+\epsilon$ (resp.\ $\theta-\epsilon$) starting from $x_k^L = -3\cdot 2^{-(k+2)}$ (resp.\ $x_k^R = 3\cdot 2^{-(k+2)}$),  and stopped at the first time that it exits $A_k$,  and we let $G_k \equiv G_k(\wh{h})$ be the event that $\eta_k^L$ and $\eta_k^R$ intersect and form a component which disconnects $0$ from $\infty$. 
We note that $\eta_k^L$ has the law of an $\SLE_{\kappa}(\rho^L;\rho^{1,R},\rho^{2,R})$ process in $\h$ from $x_k^L$ to $\infty$ with the force points located at $(x_k^L)^-,(x_k^L)^+$,  and $0$ respectively,  and $\rho^L = \frac{\epsilon}{\pi}\left(\frac{\kappa}{2}-2\right),  \rho^{1,R} = -2 + \frac{\epsilon}{\pi}\left(2 - \frac{\kappa}{2}\right)$,  and $\rho^{1,R}+\rho^{2,R} = \frac{\epsilon}{\pi}\left(2-\frac{\kappa}{2}\right)$.  
Similarly,  $\eta_k^R$ has the law of an $\SLE_{\kappa}(\rho^{2,L},\rho^{1,L};\rho^R)$ process in $\h$ from $x^R$ to $\infty$ with the force points located at $0,  (x^R)^-$,  and $(x^R)^+$ respectively,  where $\rho^R = \rho^L,  \rho^{1,L} = \rho^{1,R}$,  and $\rho^{2,L} = \rho^{2,R}$.
Since $\eta_k^L$ and $\eta_k^R$ intersect a.s.,  it follows by combining Lemmas 2.3 and 2.5 in \cite{mw2017intersections} that $\p(G_k) > 0$.
By Corollary~\ref{cor:annulus_events_happen_io}, it follows that $G_k$ occurs infinitely often a.s. By absolute continuity, the event $G_k(\wt{h})$, which is defined in the same way as $G_k(\wh h)$ but with (flow lines of) $\wh{h}$ replaced by (flow lines of) $\wt{h}$, also occurs infinitely often for the field $\wt h$.

\begin{figure}[ht!]
\begin{center}
\includegraphics[scale=0.85]{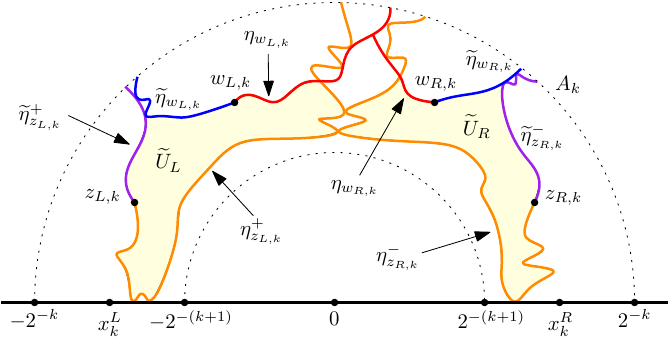}
\end{center}
\caption{\label{fig:flow_line_event} Illustration of the event $E_k$ from the proof of Lemma~\ref{lem:two_points_measurable} (in the $\kk \in (0,2]$ case). Flow lines colored red have angle $\tht$, flow lines colored orange have angle $\tht \pm \eps$, flow lines colored blue have angle $\tht + \pi$, flow lines colored purple have angle $\tht+\pi\pm\eps$. We show that $\ff\nv(\wt U_L)$ and $\ff\nv(\wt U_R)$ are contained in distinct components $U_L, U_R$ of $\h\sm\F(\tht_1, \tht_2)$ which intersect in the way described in Step 1. The event $G_k$ is not depicted for simplicity, but we include the points $x_k^L, x_k^R$.}
\end{figure}

We are now going to define an event $E_k \equiv E_k(\wh h)$ for each $k$ so that if the corresponding event $E_k(\wt h)$ for the field $\wt{h}$ occurs for infinitely many $k$ then there exist infinitely many components $U_L$, $U_R$ as described in Step 1.
As for the event $E_k(\wt{h})$, we will prove that $E_k(\wh h)$ occurs a.s. and use absolute continuity to deduce that $E_k(\wt h)$ also occurs a.s.
Before we define the event, we first need to introduce some notation.  Let $z_L = e^{i 7 \pi/8}$, $w_L = e^{i 5 \pi/8}$, $z_R = e^{i \pi/8}$, and $w_R = e^{i 3\pi /8}$. For each $k$ and $q \in \{L,R\}$, we let $z_{q,k}  = 3\cdot2^{-(k+2)} z_q$, $w_{q,k} = 3\cdot2^{-(k+2)} w_q$.  For each $w \in \h$, we let $\eta_w$ (resp.\ $\wt{\eta}_w$) be the flow line of $\wt{h}$ with angle $\theta$ (resp.\ $\theta+\pi$) starting from $w$. 
Fix $\epsilon > 0$ small. 
Let $\eta_w^+$ (resp.\ $\eta_w^-$) to be the flow line of $\wt{h}$ with angle $\theta+\epsilon$ (resp.\ $\theta-\epsilon$) starting from $w$. 
Finally, let $\wt{\eta}_w^+$ (resp.\ $\wt\eta_w^-$) to be the flow line of $\wt{h}$ with angle $\theta+\pi+\epsilon$ (resp.\ $\theta+\pi-\epsilon$) starting from $w$.
Then we let $E_k$ be the event that the following occur (see Figure~\ref{fig:flow_line_event} for an illustration):
\begin{itemize}
\item $\wt{\eta}_{z_{L,k}}^+$ and $\wt{\eta}_{w_{L,k}}$ intersect before either flow line leaves $A_k$
\item $\wt{\eta}_{z_{R,k}}^-$ and $\wt{\eta}_{w_{R,k}}$ intersect before either leaves $A_k$
\item $\eta_{w_{L,k}}$ merges into the left side of $\eta_{w_{R,k}}$ before either leaves $A_k$

\item $\eta_{z_{L,k}}^+$ intersects $(-\infty,0)$ before leaving $A_k$
\item $\eta_{z_{R,k}}^-$ intersects $(0,\infty)$ before leaving $A_k$
\item $\eta_{z_{L,k}}^+$ intersects $\eta_{z_{R,k}}^-$ before either leaves $A_k$.
\end{itemize}
We note that $E_k$ is determined by $\wt{h}|_{A_k}$ and therefore 
by applying the same argument as for $G_k$ except using  Lemmas 3.8 and 3.9 of \cite{ms2017ig4} in  addition to Lemmas 2.3 and 2.5 of \cite{mw2017intersections},  we obtain that the event $E_1$ occurs with positive probability and hence that $E_k$ occurs for infinitely many values of $k$ a.s.
As before, we conclude that $G_k(\wt h)$ also occurs infinitely often a.s.

\textit{Step 3. Conclusion of the proof.}
To complete the proof we will make repeated use of the flow line interaction rules described in \cite[Theorem~1.7]{ms2017ig4}.
Suppose that $K$ is large enough that $\ff\nv(-2^{-K}), \ff\nv(2^{-K})$ are not in $\R$, that $G_{K}(\wt h)$ occurs and that $E_k(\wt h)$ occurs for some $k > K$. By the above, this occurs a.s. for infinitely many such pairs $(K, k)$ (we may actually keep $K$ fixed).
Let $\wt{U}_L$ (resp.\ $\wt{U}_R$) be the pocket bounded by $\eta_{z_{L,k}}^+,  \wt{\eta}_{z_{L,k}}^+,  \wt{\eta}_{w_{L,k}},\eta_{w_{L,k}}$ (resp.\ $\eta_{z_{R,k}}^-,  \wt{\eta}_{z_{R,k}}^-,\wt{\eta}_{w_{R,k}},\eta_{w_{R,k}}$) as shown in Figure~\ref{fig:flow_line_event}.
When $\kk \in (0,2]$, $\eta_{w_{L,k}}$ and $\wt{\eta}_{w_{L,k}}$ (and the other pairs of flow lines starting from $z_{L,k}, w_{R,k}, z_{R,k}$) do not intersect, meaning that the pocket $\wt{U}_L$ is connected, and similarly for $\wt{U}_R$. When $\kk \in (2,4)$ the topology can be more complicated, and then we include in the event $E_k$ that the final intersection between $\eta_{w_{L,k}}$ and $\wt{\eta}_{w_{L,k}}$ occurs in some small ball around $z_{L,k}$, and that the same occurs for each of the other three pairs of flow lines started from the same points.
Arguing as above, $E_k$ still occurs infinitely often a.s.\ and in this case we redefine $\wt{U}_L, \wt{U}_R$ to be the connected components (of the complement in $\h$ of our collection of flow lines) containing points far away from these starting points.
In this case, the analogue of Figure~\ref{fig:flow_line_event} looks similar to the $\kk \in (0,2]$ case, except in some small balls around the points $z_{L,k}, w_{L,k}, z_{R,k}, w_{R,k}$.

We claim that none of the flow lines in $\varphi(\fan(\theta_1,\theta_2))$ enters $\wt{U}_L \cup \wt{U}_R$. 
Indeed,  suppose that a flow line of angle $\phi \in [\theta_1,\theta_2]$ enters $\wt{U}_L$.  We first prove that since $G_{K}$ occurs,  we have that $\phi \in [\theta-\epsilon,\theta+\epsilon]$. 
Indeed, let $V_{K}$ be the corresponding component formed by $\eta_K^L$ and $\eta_K^R$.  We claim that a flow line in $\fan(\theta_1,\theta_2)$ with angle larger (resp.\ smaller) than $\theta+\epsilon$ (resp.\ $\theta-\epsilon$) does not enter $\varphi^{-1}(V_K)$.  
First we note that if $\eta$ is the flow line of $\wt{h}$ with angle $\theta$ from $0$ to $\infty$,  then $\eta$ is a.s.\ well-defined (since $\fan_z(z)$ is a local set for $h$ by Lemma~\ref{lem:locality}) and hits $\infty$ before hitting its continuation threshold.  Also,  the flow line interaction rules imply that $\eta$ exits $V_K$ at the first intersection point of $\eta_K^L$ and $\eta_K^R$ and stays to the right (resp.\ left) of $\eta_K^L$ (resp.\ $\eta_K^R$).  Note also that a flow line of $h$ with angle larger than $\theta+\epsilon$ cannot cross $\varphi^{-1}(\eta)$ and so the only way that such a flow line can enter $\varphi^{-1}(V_K)$ is by crossing $\varphi^{-1}(\eta_K^L)$ from left to right but this cannot happen due to the flow line interaction rules.  Similarly,  a flow line of $h$ with angle at most $\theta-\epsilon$ cannot cross $\eta$ from right to left and so the only way that such a flow line enters $\varphi^{-1}(V_K)$ is by crossing $\varphi^{-1}(\eta_K^R)$ from right to left.  But the latter does not occur due to the flow line interaction rules.  In particular,  we have that there is no flow line in $\fan(\theta_1,\theta_2)$ with angle larger (resp.\ smaller) than $\theta+\epsilon$ (resp.\ $\theta-\epsilon$) which enters $\varphi^{-1}(A_k)$ for all $k' > K$.

Suppose that $\phi \in [\theta,\theta+\epsilon]$ and let $\eta_\phi$ be the image under $\ff$ of the flow line of this angle.
The flow line interaction rules imply that $\eta_\phi$ cannot cross $\eta_{w_{L,k}}$ from left to right. 
Similarly, it cannot cross $\eta^+_{z_{L,k}}$ from right to left. Since $\phi - (\theta+\pi+\epsilon) \in (-2\pi, -\pi)$, again using \cite[Theorem~1.7]{ms2017ig4}, $\eta_\phi$ cannot cross $\wt{\eta}_{z_{L,k}}^+$ from left to right. Similarly, it cannot cross $\wt\eta_{w_{L,k}}$ from right to left. 
It follows that $\eta_\phi$ cannot enter $\wt U_L$.
Suppose that $\phi \in [\theta-\epsilon,\theta)$.  Then,  in order for the flow line to enter $\wt{U}_L$,  it has to cross $\eta$ (which stays to the left (resp.\ right) of $\eta_{z_{R,k}}^-$ (resp.\ $\eta_{z_{L,k}}^+$)) from right to left and this cannot happen. 
It follows that no flow line in $\ff(\F(\tht_1, \tht_2))$ enters $\wt U_L$. By symmetry, we see that no flow line in $\ff(\F(\tht_1, \tht_2))$ enters $\wt U_R$, either.

Therefore,  we conclude that $\fan(\theta_1,\theta_2) \cap (\varphi^{-1}(\wt{U}_L) \cup \varphi^{-1}(\wt{U}_R)) = \emptyset$,  and so there exist connected components $U_L,U_R$ of $\h \setminus \fan(\theta_1,\theta_2)$ such that $\varphi^{-1}(\wt{U}_L) \subseteq U_L$ and $\varphi^{-1}(\wt{U}_R) \subseteq U_R$. 
Since $\varphi^{-1}(\eta) \subseteq \fan(\theta_1,\theta_2)$ and $\varphi^{-1}(\wt{U}_L)$ (resp.\ $\varphi^{-1}(\wt{U}_R)$) lies to the left (resp.\ right) of $\varphi^{-1}(\eta)$,  we obtain that $U_L \neq U_R$.  
Since the above holds for infinitely many pairs $(K,k)$, it remains only to show that there are infinitely many distinct components $U_L = U_L(K,k), U_R = U_R(K,k)$ (i.e.\ that $\wt U_L(K,k), \wt U_R(K,k)$ do not give rise to the same components $U_L, U_R$ for infinitely many pairs $(K,k)$).
But this can be ruled out by using Proposition~\ref{prop:bounded_hausdorff} to show that a.s.\ there exists $\phi \in (\tht, \tht+\ee)$ so that the flow line of $h$ with angle $\phi$ started from $0$ will disconnect $\ff\nv(\wt U_L)$ from some neighborhood of $y(U)$ (since this flow line will not hit $y(U)$ a.s.).
This completes the proof of the lemma.
\end{proof}

Now we are ready to prove Theorem~\ref{thm:fan_determines_flow_lines}.  It will be a consequence of Lemmas~\ref{lem:component_boundary_intersection},  ~\ref{lem:two_points_measurable} and the fact that the graph of connected components of $\h \setminus \fan(\theta_1,\theta_2)$ is connected a.s.

\begin{proof}[Proof of Theorem~\ref{thm:fan_determines_flow_lines}]
As explained at the beginning of this section, it suffices to show that $\tht(U)$ as defined in Lemma~\ref{lem:boundary_conditions} exists and is determined by $\F(\tht_1, \tht_2)$ a.s.
Fix such a component $U$ and let $\theta = \theta(U)$.
There exists a finite chain of components $U_1,\ldots, U_n$ so that $\partial U_i \cap \partial U_{i+1} \neq \emptyset$ for each $1 \leq i \leq n-1$, $\partial U_n \cap \partial \h \neq \emptyset$, and $U_1 = U$.  For each $1 \leq i \leq n$ we let $\theta_i = \theta(U_i)$.  Lemma~\ref{lem:component_boundary_intersection} implies that $\fan(\theta_1,\theta_2)$ determines $|\theta_{j+1} - \theta_j|$ for each $1 \leq j \leq n-1$.  Moreover,  Lemma~\ref{lem:component_boundary_intersection} implies that $\fan(\theta_1,\theta_2)$ determines $\theta_n$ and Lemma~\ref{lem:two_points_measurable} implies that it determines the marked points of $U_n$ as well,  which implies that it also determines which arc of $\partial U_n$ is the clockwise (resp.\ counterclockwise) arc.  Since the counterclockwise arc of a component boundary can only intersect the clockwise arc of another components boundary and vice-versa (Lemma~\ref{lem:component_boundary_intersection}) and the marked points of each component are determined by $\fan(\theta_1,\theta_2)$ (Lemma~\ref{lem:two_points_measurable}),  it follows that the clockwise and counterclockwise arcs of each of the $\partial U_j$ is determined.  This, in turn implies that $(\theta_{j+1} - \theta_j)$ for each $1 \leq j \leq n-1$ is determined.  Since
\[ \theta = \theta_1 =  \sum_{j=1}^{n-1} (\theta_j - \theta_{j+1}) + \theta_n,\]
it follows that $\fan(\theta_1,\theta_2)$ determines $\theta$.  The result follows by combining with Lemma~\ref{lem:form_of_the_boundary} since $U$ was an arbitrary complementary component.
\end{proof}

\appendix

\section{Local set lemmas}

\begin{lemma}\label{lem:intersection_of_local_sets}
Suppose that $h$ is a \text{GFF} on a domain $D \subsetneq \C$ and $(A_n)$ is a sequence of local sets of $h$ so that $A_n$ is a.s.\ determined by $h$ and $A_{n+1} \subseteq A_n$ for each $n$.  Then $A = \cap_n A_n$ is local for $h$.
\end{lemma}

\begin{proof}
We will prove the claim of the lemma in the case that $D$ is bounded since then the result will follow for general domains using the conformal invariance of the $\text{GFF}$.  We are going to use the characterization of local sets given in \cite[Lemma~3.9]{schramm2013contour}.  Fix $B \subseteq D$ an open set.  We note that for all $n$,  the event $\{A_n \cap B = \emptyset\}$ is a.s.\  determined by $h$ and \cite[Lemma~3.9]{schramm2013contour} implies that conditionally on the projection of $h$ onto $H_{\text{harm}}(B)$,  we have that $h$ and $\one_{\{A_n \cap B =\emptyset\}}$ are independent,  since under this conditioning,  the field $h$ is determined by its projection onto $H_0^1(B)$.  Hence,  combining with \cite[Proposition~10.7]{dms2021mating},  we obtain that $\{A_n \cap B = \emptyset\}$ is a.s.\ determined by the projection of $h$ onto $H_{\text{harm}}(B)$.  Set $K = D \setminus B$ and note that $K$ is closed with respect to the relative topology on $D$.
Set also $U_m = \{z \in D : \dist(z,  K) > 1/m\}$ for all $m \in \N$,  and let $F_m = D \sm U_m$.
Note that $B = \cup_{m\geq 1} U_m$ and $\{A \cap B = \emptyset\} = \cap_{m \geq 1} \cup_{n \geq 1} \{A_n \cap U_m = \emptyset\}$. 
Define $\CF_U^{h}, \CF_{K^+}^{h}$ as in \cite{ms2016imag1}.
For all $U \subseteq D$,  we set $\CF_U^h = \sigma((h,f) : f \in C_0^{\infty}(U))$.
For all $m, n$, we can show as above that $\{A_n \cap U_m = \emptyset\} \in \CF_{F_m^+}^{h}$. 
It follows that $\{A \cap B = \emptyset\} \in \cap _{m \geq 1} \CF_{F_m^+}^h$.
Note that $\CF_{K^+}^{h} \subseteq \CF_{F_m^+}^{h}$ for all $m$ and if $F_m \subseteq V$ for an open set $V$, then $\CF_{F_m^+}^{h} \subseteq \CF_{V}^{h}$. Since every open set $V$ containing $K$ must also contain one of the $F_m$, it follows from \cite[Proposition~3.2]{ms2016imag1} that $\CF_{K^+}^h = \cap_{U \subseteq D,  U \text{open}}\CF_U^h = \cap _{m \geq 1} \CF_{F_m^+}^h$, and hence that $\{A \cap B = \emptyset\} \in \CF_{K^+}^h$.
It then follows from \cite[Lemma~3.9]{schramm2013contour} that $A$ is a local set.
\end{proof}

\begin{lemma}
\label{lem:union_of_local_sets}
Suppose that $h$ is a \text{GFF} on a domain $D \subsetneq \C$ and let $(A_n)$ be a sequence of local sets of $h$ so that $A_n \subseteq A_{n+1}$ for each $n \in \N$.  Then $A = \overline{\cup_n A_n}$ s a local set for $h$.
\end{lemma}

\begin{proof}
As in the proof of Lemma~\ref{lem:intersection_of_local_sets},  we are going to use \cite[Lemma~3.9]{schramm2013contour}.  Fix $U \subseteq D$ open set and let $h_1$ (resp.\ $h_2$) be the projection of $h$ onto $H_0^1(U)$ (resp.\ $H_{\text{harm}}(U)$).  Fix also measurable function $G \colon H_{\text{loc}}^{-1}(D) \to \R$ and a Borel set $B \subseteq \R$.  Then \cite[Lemma~3.9]{schramm2013contour} implies that $\one_{A_n \cap U \neq \emptyset}$ and $\one_{G(h_1) \in B}$ are conditionally independent given $h_2$,  and so we have that
\begin{align*}
\E[\one_{\{A \cap U \neq \emptyset\}} \one_{\{G(h_1) \in B\}} \giv h_2]&= \lim_{n \to \infty} \E[\one_{\{A_n \cap U \neq \emptyset\}} \one_{\{G(h_1) \in B\}} \giv h_2]\\
&= \lim_{n \to \infty} \E[\one_{\{A_n \cap U \neq \emptyset\}} \giv h_2] \E[\one_{\{G(h_1) \in B\}} \giv h_2]\\
&=\E[\one_{\{A \cap U \neq \emptyset\}} \giv h_2] \E[\one_{\{G(h_1) \in B\}} \giv h_2].
\end{align*}
Since $G,B$ were arbitrary,  we obtain that $\{A \cap U \neq \emptyset\}$ and $h_1$ are conditionally independent given $h_2$ and so \cite[Lemma~3.9]{schramm2013contour} implies that $A$ is a local set.
\end{proof}

\begin{lemma}\label{lem:gff_ergodic}
Suppose that $r > 0, r \neq 1$ and $h$ is a GFF on $\h$ with boundary conditions given by $-a$ on $\R_-$ and $b$ on $\R_+$. Let $\tht_r$ be the transformation from the space of distributions on $\h$ to itself given by $(\tht_r(h))(\phi) = h((1/r^2)\phi(\cdot/r))$. Then $\tht_r$ is ergodic and measure-preserving on $(\CD'(\h), \CF, \p)$, where $\CD'(\h)$ is the space of distributions on $\h$, $\CF$ is the $\ss$-algebra generated by the random variables $(h, \phi)$ for $\phi \in C_0^\infty(\h)$ and $\p$ is the law of $h$.
\end{lemma}
\begin{proof}
By considering the map $z \mapsto -1/z$ if necessary, we may assume that $r < 1$. The fact that $\tht_r$ is measure-preserving comes from the fact that the law of $h$ is invariant under scaling, see e.g.\ \cite[Theorem~1.57]{bp_gff}, (note that this holds if $h$ is a zero boundary GFF or if the boundary conditions are constant on each of $\R_-$ and $\R_+$).
To prove ergodicity, let $A \in \CF$ be an event that is invariant under $\tht_r$, by which we mean that $\tht_r\nv(A) = A$. By the martingale convergence theorem, $\E[\one_A \giv \ss(h|_{B(0, m) \cap \h})]$ converges a.s.\ and in $L^1$ to $\one_A$ as $m \to \infty$. Therefore, for any $\eps > 0$, we can find $m > 0$ such that if $B$ is the event $\{\E[\one_A \giv \ss(h|_{B(0, m) \cap \h})] > 1/2\}$, then $\p(A \triangle B) < \eps$, where $A \triangle B$ is the symmetric difference $(A \cap B^c) \cup (A^c \cap B)$.
Note that for any $N \in \N$,
\[\left(A \cap \tht_r^{-N}(A)\right) \triangle \left(B \cap \tht_r^{-N}(B)\right) \subseteq (A \triangle B) \cup \left(\tht_r^{-N}(A) \triangle \tht_r^{-N}(B)\right),\]
and that the two events on the right hand side have probability less than $\eps$ by scale invariance.
Therefore,
\[\p(A) = \p(A \cap \tht_r^{-N}(A)) \leq \p(B \cap \tht_r^{-N}(B)) + 2\eps\]
for all $N$.
Now, the probability on the right-hand side can be written as 
\[\p(\tht_r^{-N}(B))\p(B \giv \tht_r^{-N}(B)) = \p(B)\p(B \giv \tht_r^{-N}(B))\]
for all $N$.
Notice that since $B$ depends on $h|_{B(0, m) \cap \h}$, the event $\tht_r^{N}(B)$ depends on $h|_{B(0, mr^{-N}) \cap \h}$. The proof of \cite[Lemma~7.2]{dms2021mating} applies also to the case of the  zero-boundary GFF, meaning that $\cap_{\dd < 0} \ss(h|_{B(0,\dd) \cap \h})$ is trivial and hence that $\p(B\giv \tht_r^{-N}(B)) \to \p(B)$ as $N \to \infty$.
Therefore we can deduce that
\[\p(A) \leq \p(B)^2 + 2\eps \leq \p(A)^2 + \eps^2 + 4\eps,\]
for all $\eps > 0$ and hence that $\p(A) \in \{0,1\}$, proving ergodicity.
\end{proof}

\begin{corollary}\label{cor:annulus_events_happen_io}
Suppose that we have the setup of Lemma~\ref{lem:gff_ergodic}. Let $E \in \CF$ be an event which is measurable with respect to the restriction of $h$ to $B(0, 1) \cap \h$. For $k \in \N$, let $E_k$ be the event that $\tht_{1/2}^{-k}(E)$ holds (note that this means that the event $E$ holds for the GFF $\tht_{2^{-k}} h$ and that $E_k$ is measurable with respect to $h|_{B(0, 2^{-k}) \cap \h}$). Then the sum $(1/N)\sum_{1 \leq k \leq N} \one_{E_k} \to \p(E)$ a.s.\ as $N \to \infty$, and in particular if $\p(E) > 0$ then $E_k$ occurs infinitely often a.s.
\end{corollary}
\begin{proof}
Due to Lemma~\ref{lem:gff_ergodic}, this is an immediate consequence of Birkhoff's ergodic theorem.
\end{proof}

\bibliographystyle{abbrv}
\bibliography{references}

\end{document}